\documentclass[reqno,11pt,tbtags,twoside]{amsart}
\usepackage{amsmath,amssymb,mathrsfs,amsthm,amsfonts}
\usepackage[inline]{enumitem}
\usepackage[usenames,dvipsnames]{xcolor}
\usepackage{hyperref, cleveref}
\usepackage[percent]{overpic}
\usepackage{comment}
\usepackage[foot]{amsaddr}
\usepackage{csquotes}
\usepackage{algorithm}
\usepackage{algorithmic}
\usepackage{mathtools}
\usepackage{subcaption}
\usepackage{tikz}
\usepackage{bm}
\usetikzlibrary{calc}
\usetikzlibrary{arrows.meta,backgrounds}
\usepackage{multirow,array,longtable,booktabs}
\usetikzlibrary{arrows}

\hypersetup{
	colorlinks=true, linkcolor=blue,
	citecolor=ForestGreen
}
\usepackage[paper=letterpaper,margin=1in,marginparwidth=1.5cm]{geometry}
\newtheorem{theorem}{Theorem}[section]
\newtheorem{lemma}[theorem]{Lemma}
\newtheorem{proposition}[theorem]{Proposition}

\newtheorem{corollary}[theorem]{Corollary}
\theoremstyle{definition}
\newtheorem{definition}[theorem]{Definition}
\newtheorem{assumption}[theorem]{Assumption}
\newtheorem{example}[theorem]{Example}
\crefname{corollary}{Corollary}{Corollaries}
\crefname{assumption}{Assumption}{Assumptions}
\crefname{convention}{Convention}{Conventions}
\theoremstyle{remark}

\newtheorem{remark}[theorem]{Remark}
\newtheorem{convention}[theorem]{Convention}
\numberwithin{equation}{section}
\AddToHook{env/proposition/begin}{\crefalias{theorem}{proposition}}
\AddToHook{env/lemma/begin}{\crefalias{theorem}{lemma}}
\AddToHook{env/corollary/begin}{\crefalias{theorem}{corollary}}
\AddToHook{env/definition/begin}{\crefalias{theorem}{definition}}
\AddToHook{env/example/begin}{\crefalias{theorem}{example}}
\AddToHook{env/remark/begin}{\crefalias{theorem}{remark}}
\AddToHook{env/assumption/begin}{\crefalias{theorem}{assumption}}
\AddToHook{env/convention/begin}{\crefalias{theorem}{convention}}

\usepackage{acronym}

\acrodef{LDP}{Large Deviation Principle}

\newcommand{\E}{\mathbb{E}}	% expectation
\renewcommand{\d}{\mathrm{d}}	% differential
\newcommand{\be}{\begin{equation}}
\newcommand{\ee}{\end{equation}}
\newcommand{\bea} {\begin{array}{rl}}
\newcommand{\eea} {\end{array}}
\newcommand{\bepa}{\left\{ \begin{array}{l}}
\newcommand{\eepa} {\end{array}\right.}
\makeatletter

\newcommand\norm[1]{\lVert#1\rVert}

\newcommand{\R}{\mathbb{R}} % real numbers
\newcommand{\N}{\mathbb{N}}
\newcommand{\op}{\mathrm{op}}

\newcommand{\ds}{\displaystyle}

\newcommand{\calC}{\mathcal{C}}

\newcommand{\TT}{\mathsf{T}} %transpose
\newcommand{\1}{\bm 1}
\newcommand{\defeq}{\coloneqq}

\renewcommand{\hat}{\widehat}

\renewcommand{\bar}{\overline}

\usepackage{graphicx}

\newcommand{\bd}{{\bf d}}
\newcommand{\bw}{\bm{w}}
\newcommand{\bL}{\bm{L}}
\newcommand{\ov}{\overline}
\newcommand{\bx}{\bm{x}}
\newcommand{\ba}{\bm{a}}
\newcommand{\bp}{\bm{p}}
\newcommand{\bxi}{\bm{\xi}}
\newcommand{\bzeta}{\bm{\zeta}}
\newcommand{\bX}{\bm{X}}
\newcommand{\bY}{\bm{Y}}
\newcommand{\bP}{\mathbb{P}}
\newcommand{\cP}{\mathcal{P}}
\newcommand{\sL}{\mathcal{L}}
\newcommand{\eps}{\epsilon}

\newcommand{\linf}{L^{\infty}}

\newcommand{\tr}{\text{tr}}
\newcommand{\cA}{\mathcal{A}}
\newcommand{\cl}{\CL}
\newcommand{\ol}{\OL}
\newcommand{\dist}{\mathrm{dist}}
\newcommand{\dis}{\mathrm{disp}}
\newcommand{\lip}{\mathrm{Lip}}
\newcommand{\bv}{\bm{v}}
\newcommand{\bF}{\bm{F}}
\newcommand{\bG}{\bm{G}}
\newcommand{\bZ}{\bm{Z}}
\newcommand{\balpha}{\boldsymbol{\alpha}}
\newcommand{\bA}{\bm{A}}
\newcommand{\cF}{\mathcal{F}}
\newcommand{\cG}{\mathcal{G}}
\newcommand{\irr}{\text{Irr}}
\newcommand{\cL}{\mathcal{L}}
\newcommand{\bbL}{\mathbb{L}}
\newcommand{\bbF}{\mathbb{F}}
\newcommand{\bmb}{\boldsymbol{b}}
\newcommand{\bu}{\bm{u}}
\newcommand{\wt}{\widetilde}
\newcommand{\lm}{\mathrm{LL}}
\newcommand{\by}{\bm{y}}
\newcommand{\mf}{\mathrm{MF}}
\newcommand{\bH}{\bm{H}}

\newcommand{\OL}{\mathrm{OL}}
\newcommand{\CL}{\mathrm{CL}}
\newcommand{\Dis}{\dist}
\newcommand{\scrF}{\mathscr{F}}
\newcommand{\cH}{\mathcal{H}}
\newcommand{\bbH}{\mathbb{H}}

\title[Non-asymptotic approach to games with many players]{A non-asymptotic approach to {monotone} stochastic differential games with many players} 

\author[M.\ Cirant]{Marco Cirant\textsuperscript{(1)}\!}
\address{\textsuperscript{(1)}Dipartimento di Matematica ``T.\ Levi-Civita'' \\ Università degli Studi di Padova \\ Via Trieste 63 \\ 35121 Padova, Italy.}

\author[J.\ Jackson]{Joe Jackson\textsuperscript{(2)}\!}
\address{\textsuperscript{(2)}Department of Mathematics \\ University of Chicago \\
5734 S.~University Avenue \\ Chicago, Illinois 60637, USA.}

\author[D.\ F.\ Redaelli]{Davide Francesco Redaelli\textsuperscript{(3)}}
\address{\textsuperscript{(3)}Dipartimento di Matematica \\ Università di Roma Tor Vergata \\ Via della Ricerca Scientifica 1\\ 00133 Roma, Italy.}

\email{cirant@math.unipd.it, jsjackson@uchicago.edu, redaelli@mat.uniroma2.it}

\begin{document}

	\begin{abstract} 
        We consider stochastic differential games with a large number of players, with the aim of quantifying the gap between closed-loop, open-loop and distributed equilibria. We show that, under two different monotonicity conditions, the equilibrium trajectories are close when the interactions between the players are weak. Our approach is non-asymptotic in nature, in the sense that it does not make use of any a priori identification of a limiting model, like in mean field game (MFG) theory. The main technical step is to derive bounds on solutions to systems of PDE/FBSDE characterizing the equilibria that are independent of the number of players. 
        
        When specialized to the mean field setting, our estimates yield quantitative convergence results for both open-loop and closed-loop equilibria without any use of the master equation. In fact, our main bounds hold for games in which interactions are much sparser than those of MFGs, and so we can also obtain some ``universality" results for MFGs, in which we show that games governed by dense enough networks converge to the usual MFG limit. Finally, we use our estimates to study a joint vanishing viscosity and large population limit in the setting of displacement monotone games without idiosyncratic noise.
        
%In addition to the application of these estimates to the quantification of proximity between different notions of equilibria in non-symmetric settings, we establish new results on the universality of MFG equilibria and on the joint vanishing viscosity and large population limit. The first set of results shows in particular that if initial conditions are i.i.d., then the MFG equilibria approximate well both the open-loop and the closed-loop ones.
	\end{abstract}
	
%	\subjclass[2020]{}
%	\keywords{}
	
	\maketitle

  \tableofcontents

%\marginpar{\tiny {JJ: additional question: for the estimates in Section 6 and 7, should we also state what they imply about convergence of empirical measures?}}

\section{Introduction}

This paper is concerned with $N$-player stochastic differential games, as well as the partial differential equations (PDEs) and forward-backward stochastic differential equations (FBSDEs) which describe their Nash equilibria. Our focus is on games with a large but finite number of players, in which each player controls a private state process. Such games have received significant attention in recent years in the context of mean field game (MFG) theory. The games we study in this paper do not generally possess any asymptotic structure, so MFG theory does not directly apply. However, we view the mean field case as an important benchmark example, and we present several applications of our results to MFG limit theory.%: we aim on the one hand to extend some conclusions of MFG theory to a more heterogeneous setting, and on the other hand to prove non-asymptotic estimates which provide novel insights even when specialized to the mean field case.

Our main results are estimates which show that for \textit{monotone}\footnote{Or, more generally, \emph{semi-monotone}, as we note at the beginning of Section~\ref{ssec.mess} below.} $N$-player games with \textit{weak interactions}, the open-loop, closed-loop, and distributed equilibria are similar. Crucially, these estimates are independent of $N$, which makes them useful for games with many players. Most of our estimates also rely on a non-degenerate idiosyncratic noise, but we track the dependence on the noise in anticipation of an application to first order MFGs. 

There are two distinct motivations for studying the gap between open-loop, closed-loop, and distributed equilibria. The first is that estimates on this gap allow one to show that certain results from MFG limit theory hold for games with weak interactions, \emph{even without any asymptotic structure}. In this sense, they can be viewed as part of a growing literature devoted to analogues of mean field theory for high-dimensional systems without any asymptotic structure; this perspective will be discussed in more detail in Section~\ref{subsec.approxind}.

The second motivation is that our estimates provide a useful toolkit when specialized to various asymptotic regimes, because they allow one to use open-loop and distributed equilibria as tractable approximations to closed-loop equilibria. To understand this point, consider the fact that the very first convergence result for mean field games was proved for \textit{distributed equilibria} in one of the original papers of Lasry and Lions (see \cite[Theorem 2.2]{lasry2007mean} and \cite{LionsSeminar}), under the structural condition which is now known as Lasry--Lions monotonicity. This argument is remarkably simple compared to later works aimed at open-loop or closed-loop equilibria. 
Similarly, the recent work \cite{JabinPoyatoSoler} treats a convergence problem in the setting of heterogeneous particle systems by first estimating the difference between the $N$-particle system and its ``independent projection" (a simpler analogue of distributed equilibria in the context of interacting diffusions, which will be discussed in Section~\ref{subsec.approxind}), and then proving the convergence of these independent projections.

When our results are applied in the setting of MFG limit theory, they yield self-contained proofs of quantitative convergence results for closed-loop, open-loop, and distributed equilibria, for MFGs with either displacement monotone or Lasry--Lions monotone data. In addition, we establish in Section~\ref{sec.universality} some ``universality results" (in the sense, e.g., of \cite{LackerSoretLabel}), which show that for games whose interactions are governed by dense enough networks, the usual mean field approximation still applies.

The main analytical ingredients in the proofs are a priori %\marginpar{\tiny DF: Changed dimension-free to a priori becasue we haven't explained what dimension-free means yet. {OK}}
estimates on the PDEs which are related to closed-loop and open-loop equilibria of the game (\eqref{nash} and \eqref{pontpde} in Section~\ref{ssec.cod}, respectively). The estimates on \eqref{nash} are similar to the ones obtained in \cite{CR24b}, but hold under weaker conditions on the ``strength of interaction" between players. %\marginpar{\tiny DF: Also weaker regularity assumptions, no? {Yes, but since this part is mostly technical I'd be OK without mentioning it} }
Our proofs borrow several ideas from \cite{CR24b}, but also rely on a key new technique: we first obtain estimates for the simpler PDE \eqref{pontpde}, and then use the ``weak interactions" between players to view (the derivatives of) \eqref{nash} as a perturbation of \eqref{pontpde}. This allows us to obtain the key estimates on \eqref{nash} in a perturbative manner, which leads to a significant generalization and simplification of the arguments in \cite{CR24b}. %We refer to Section~\ref{sec.outline} for an outline of the proofs. 

%Most of our results also rely on a non-degenerate idiosyncratic noise impacting the players. The idea that different formulations of the game should be nearly equivalent when players interact weakly has appeared before; e.g., the following is observed in \cite[Remark 5.7]{CarmonaLectures}.
%\begin{displayquote}
%\emph{``From a mathematical standpoint, open loop equilibriums are more
% tractable than closed loop equilibriums because players need not consider how their opponents would react to deviations from the equilibrium path. With this in mind, one should
% expect that when the impact of players on their opponents’ costs/rewards is small, open
% loop and closed loop equilibriums should be the same."}
%\end{displayquote}
%Our results quantify this intuition outside the mean field framework for the first time.

\subsection{Main estimates in a simplified setting} \label{ssec.mess}

In order to give the flavor of our main results before introducing too many technicalities, we now explain what our results say in a concrete setting. In particular, we are going to focus in the introduction on the case without common noise, with separable Hamiltonians, and featuring displacement rather than Lasry--Lions monotonicity. We emphasize however that we prove analogous results for the Lasry--Lions monotone case, which is equally interesting, and more technically demanding. We also point out that our results actually hold for \emph{semi}-monotone data, in the sense that we can allow for some degree of non-monotonicity (see Assumptions~\ref{assump.disp} and \ref{assump.LL}); in particular,  our assumptions become milder as the time horizon $T$ decreases, while data are required to be ``almost'' monotone as $T$ goes to infinity.

We work on a filtered probability space $\big(\Omega, \scrF, \bbF = (\scrF_t)_{t \in [0,T]}, \bP\big)$ hosting independent Brownian motions $(W^i)_{i = 1,\dots,N}$. We consider a game in which players $i = 1,\dots,N$ each control an $\R^d$-value diffusion $X^i$. The game is specified by functions 
\begin{equation*}
    L^i \colon \R^d \times \R^d \to \R, \quad F^i,\, G^i \colon (\R^d)^N \to \R, \quad i = 1,\dots,N, 
\end{equation*}
together with initial conditions $(\zeta^i)_{i = 1,\dots,N}$, which are square-integrable, $\scrF_0$-measurable, $\R^d$-valued random vectors. 
The \emph{closed-loop} formulation of the game is arguably the most natural one from a modeling perspective: player $i$ chooses a feedback function $\alpha^i(t,\bx) \colon [0,T] \times (\R^d)^N \to \R^d$ satisfying appropriate growth and measurability conditions, and $\bX = (X^1,\dots,X^N)$ evolves according to 
\begin{align} \label{dynamics.intro}
    \d X_t^i = \alpha^i(t,\bX_t)\, \d t + \sqrt{2}\, \d W_t^i \quad (t \in [0,T]), \qquad X_0^i = \zeta^i.
\end{align}
The goal of player $i$ is to minimize the cost functional 
\begin{align} \label{cost.intro}
    J^i(\bm\alpha) = \E\bigg[\int_0^T \Big(L\big(X_t^i,\alpha^i(t,\bX_t)\big) + F^i(\bX_t) \Big) \d t + G^i(\bX_T) \bigg].
\end{align}
We will also consider two other formulations of this stochastic differential game. In the \emph{open-loop}\ formulation, player $i$ chooses an $\R^d$-valued, square-integrable process $\alpha^i = (\alpha_t^i)_{t \in [0,T]}$ adapted with respect to the filtration generated by $(W^i)_{i = 1,\dots,N}$; thus, in the dynamics \eqref{dynamics.intro} and the cost \eqref{cost.intro}, $\alpha^i\big(t,\bX_t\big)$ must be replaced with $\alpha_t^i$. In the \emph{distributed} formulation, player $i$ chooses instead a feedback function $\alpha^i(t,x^i) \colon [0,T] \times \R^d \to \R^d$ which depends only on their own state, and thus in the dynamics \eqref{dynamics.intro} and the cost \eqref{cost.intro}, $\alpha^i\big(t,\bX_t\big)$ must be replaced with $\alpha^i(t,X_t^i)$. For all three formulations of the game, the notion of Nash equilibrium is the usual one; a tuple $\balpha = (\alpha^1,\dots,\alpha^N)$ such that $\alpha^i$ minimizes $J^i$ when $(\alpha^j)_{j \neq i}$ are held fixed. We refer to Section~\ref{sec.formulation} for a more precise discussion of these three formulations of the game.

We assume that $L^i$, $F^i$ and $G^i$ are $C^2$, and that there are constants $C_{\bm{L}}>0$, $C_{\bm{F}}$, $C_{\bm{G}}$ such that 
\begin{equation} \label{disp.intro.L}
 \frac{1}{C_{\bm{L}}} \begin{pmatrix}
      0 & 0 \\ 0 & I_d
  \end{pmatrix} \leq D^2 L^i \leq C_{\bm{L}} \begin{pmatrix}
      I_d & 0 \\ 0 & I_d
  \end{pmatrix}, \quad i = 1,\dots,N
\end{equation}
as well as 
\begin{equation}  \label{disp.intro.FG}
0 \leq \big(D_{x^jx^i} F^i\big)_{i,j = 1,\dots,N} \leq C_{\bm{F}} I_{Nd}, \quad 0 \leq \big(D_{x^jx^i} G^i\big)_{i,j = 1,\dots,N} \leq C_{\bm{G}} I_{Nd}, 
\end{equation}
where $D^2 L^i$ is the Hessian of $L^i$ in the variables $(x,a) \in \R^d \times \R^d$, we are viewing $(D_{x^jx^i} F^i)_{i,j}$ and $(D_{x^jx^i} G^i)_{i,j}$ as elements of $\R^{Nd \times Nd} \simeq  (\R^{d \times d})^{N \times N}$ in a natural way, and $I_{k}$ ($k \in \N$) denotes the $k$-dimensional identity matrix. 
We define the quantities
\begin{equation*} 
\begin{gathered}%\label{def.delta}
    \delta^i \coloneqq \Big\|\sum_{j \neq i} |D_{x^j} G^i|^2 \Big\|_{\infty} + \Big\| \sum_{j \neq i} |D_{x^j} F^i|^2 \Big\|_{\infty}\,, \quad 
       \kappa^i \coloneqq \Big\|\sum_{j \neq i} |D_{x^jx^i} G^i|^2 \Big\|_{\infty} + \Big\| \sum_{j \neq i} |D_{x^jx^i} F^i|^2 \Big\|_{\infty}\,,
%\\ %\label{def.tildekappa}
%       \tilde\kappa^i \coloneqq \Big\| \sum_{j \neq i} |D_{x^ix^j} G^j|^2 \Big\|_{\infty} + \Big\| \sum_{j \neq i} |D_{x^ix^j} F^j|^2 \Big\|_{\infty}\,, \quad  \tilde\kappa \coloneqq \max_{i=1,\dots,N}  \tilde\kappa^i\,.
\end{gathered}
\end{equation*}
and we denote by 
\begin{equation*}
    \bX^{\cl} = \big(X^{\cl,1},\dots,X^{\cl,N} \big), \quad  \bX^{\ol} = \big(X^{\ol,1},\dots,X^{\ol,N} \big), \quad  \bX^{\Dis} = \big(X^{\Dis,1},\dots,X^{\Dis,N} \big)
\end{equation*}
an admissible (in the sense discussed in Section~\ref{sec:admiss}) closed-loop, open-loop, and distributed equilibrium of the $N$-player game, respectively (see Remark \ref{rmk:exequilibria} on their existence). 
\begin{theorem}[Abdridged version of \Cref{thm.oldistdisp}]
    \label{cor2.intro} 
    Suppose that \eqref{disp.intro.L} and \eqref{disp.intro.FG} hold, that $(\zeta^i)_{i}$ are independent, and that there is a constant $C_P > 0$ such that, for each $i$, $m^i = \cL(\zeta^i)$ satisfies a Poincar\'e inequality with constant $C_P$.\footnote{We recall in \Cref{rmk.poincaredirac} the definition of Poincar\'{e} inquality.} Then, there is a constant $C > 0$ (which depends only on $C_{\bm L}, C_{\bm F}, C_{\bm G}, C_P, d$, and $T$)
    such that
    \begin{equation} \label{oldist.intro}
        \E\bigg[ \sup_{t \in [0,T]} \big|\bX_t^{\ol} - \bX_t^{\Dis} \big|^2 \bigg] \leq C \sum_{i = 1}^N \kappa^i
    \end{equation}
\end{theorem} 
The strength of this result is that $C$ does not depend on $N$.
Since the left-hand side of \eqref{oldist.intro} is a priori of size $N$, this says that under \eqref{disp.intro.L} and \eqref{disp.intro.FG}, open-loop and distributed equilibria are close provided that $\sum_{i = 1}^N \kappa^i = o(N)$. We note that, in light of the dynamics \eqref{dynamics.intro}, the distributed equilibrium is such that $(X^{\Dis,i})_i$ are independent; thus, we deduce that if $\sum_{i = 1}^N \kappa^i = o(N)$, then $\bX^{\ol}$ has approximately independent components. %We refer to \Cref{thm.oldistLL} for the (unabridged) analogue for Lasry--Lions monotone data.  

The next result concerns the gap between open-loop and closed-loop equilibria.
\begin{theorem}[Abdridged version of \Cref{thm.olcldisp}]
    \label{cor1.intro}
    Suppose that \eqref{disp.intro.L} and \eqref{disp.intro.FG} hold. Then, there are constants $\eps, C > 0$ (which depend only on the quantities $C_{\bm{L}}$, $C_{\bm{F}}$, $C_{\bm{G}}$, $d$, and $T$) such that the following holds: if 
    \begin{equation} \label{weakinteraction.intro}
        \max_{i = 1,\dots,N} \delta^i \cdot \sum_{i = 1}^N \delta^i < \eps,
    \end{equation}
     then we have the estimate
    \begin{equation*}% \label{olcl.dist}
        \E\bigg[ \sup_{t \in [0,T]} \bigl|\bX_t^{\cl} - \bX_t^{\ol} \bigr|^2 \bigg] \leq C  \max_{i = 1,\dots,N} \delta^i \cdot \sum_{i = 1}^N \delta^i.
    \end{equation*}
\end{theorem} 
The interpretation of this result is the following: for displacement monotone games, open-loop and closed-loop equilibria are close, provided that $\max_{i = 1,\dots,N} \delta^i \cdot \sum_{i = 1}^N \delta^i = o(1)$ as $N \to \infty$. In addition, by the triangle inequality and \eqref{oldist.intro}, we see that if $\sum_{i = 1}^N \kappa^i = o(N)$ and $\max_{i = 1,\dots,N} \delta^i \cdot \sum_{i = 1}^N \delta^i = o(1)$, then $\bX^{\cl}$ is close to $\bX^{\Dis}$, and thus have approximately independent components. %We refer to \Cref{thm.olclLL} for the Lasry--Lions monotone version. 

\subsection{Benchmark examples}

We now discuss two benchmark examples to which our results can be applied.

\subsubsection{The mean field case}
While our estimates do not require any asymptotic structure, it is useful to see what they imply in the well-studied mean field case:
\begin{equation*}% \label{mf.intro}
  L^i = L, \quad  F^i(\bx) = \cF\big(x^i, m_{\bx}^{N,i}\big), \quad G^i(\bx) = \cG\big(x^i,m_{\bx}^{N,i}\big), \quad (\zeta^i)_{i = 1,\dots,N} \overset{\text{i.i.d.}}{\sim} m_0,
\end{equation*}
for some functions $\cF, \cG : \R^d \times \cP_2(\R^d) \to \R$ and some $m \in \cP_2(\R^d)$, and where $m_{\bx}^{N,i} = \frac{1}{N-1} \sum_{j \neq i} \delta_{x^j}$. 

Assume that $L : \R^d \times \R^d \to \R$, $\cF, \cG : \R^d \times \cP_2(\R^d) \to \R$ are $C^2$, that $L$ satisfies the convexity condition asked of $L^i$ in \eqref{disp.intro.L}, and that $\cF$ and $\cG$ are displacement monotone, i.e. 
\begin{align*}% \label{mfdisp.intro}
    \E\Big[ \Big(D_x \cF\big(X, \cL(X)\big) - D_x \cF\big(\ov X, \cL(\ov X)\big) \Big) \cdot (X - \ov X) \Big] \geq 0
\end{align*}
for all square-integrable random vectors $X,\ov X$. Then, \eqref{disp.intro.L} and \eqref{disp.intro.FG} are satisfied uniformly for all large enough $N$ (up to some errors which vanish in $N$, and are covered by our semi-monotonicity assumptions). Moreover, 
\begin{equation*}
    D_{x^j} F^i(\bx) = \frac{1}{N-1} D_m \cF\big(x^i,m_{\bx}^{N,i},x^j\big), \quad D_{x^jx^i} F^i(\bx) = \frac{1}{N-1} D_m D_x \cF\big(x^i,m_{\bx}^{N,i},x^j\big),
\end{equation*}
from which it follows that if $D_m \cF$ and $D_m D_x \cF$ are bounded, then both $\delta^i$ and $\kappa^i$ are of order $1/N$. In particular, \eqref{weakinteraction.intro} is satisfied for large enough $N$. As a consequence, when $\cF$ and $\cG$ are displacement monotone and smooth enough, our results give
\begin{equation} \label{formulations.disp}
     \E\bigg[ \sup_{t \in [0,T]} \bigl|\bX_t^{\cl} - \bX_t^{\ol} \bigr|^2 \bigg] = O\Bigl(\frac1N\Bigr), \qquad \E\bigg[ \sup_{t \in [0,T]} \bigl|\bX_t^{\ol} - \bX_t^{\Dis} \big|^2 \bigg] = O(1)
\end{equation}
since both quantities are a priori of size $N$, both estimates should be viewed as convergence result, albeit at different rates. 

To see why this is useful, we note that the closed-loop formulation of the game is arguably the most reasonable from a modeling perspective, but is often the most challenging to deal with mathematically. In the MFG setting, quantitative convergence of closed-loop equilibria is typically approached through an analysis of an infinite-dimensional PDE called the \emph{master equation}, as in the seminal work \cite{CDLL}; but in light of \eqref{formulations.disp}, we can first prove a limiting result for open-loop or distributed equilibria, and then transfer it to a result for closed-loop equilibria, bypassing completely the master equation.

For example, \cite[Theorem 2.8]{JacksonTangpi} establishes quantitative convergence for open-loop equilibria under the displacement monotonicity conditions \eqref{disp.intro.L}, \eqref{disp.intro.FG}, in addition to some technical regularity conditions; as a consequence, our estimates immediately imply a quantitative convergence result for the closed-loop equilibria 
  %$\bX^{\text{CL},N} = \big(X_t^{\text{CL},N,1},\dots,X_t^{\text{CL},N,N}\big)$ 
  of the $N$-player games.
We also recall that the early limiting results for mean field games were proved for distributed equilibria in one of the original papers of Lasry--Lions (see \cite[Theorem~2.2]{lasry2007mean}), under the structural condition which is now known as Lasry--Lions monotonicity; hence, by combining this argument with our Theorems \ref{thm.olclLL} and \ref{thm.oldistLL}, we immediately obtain convergence results for MFGs under the Lasry--Lions monotonicity condition.

Our quantitative convergence results for mean field games are not new: they can be derived as a consequence of the smoothness of the master equation; for example, one can use the results of \cite{gangbomeszaros} to show that under displacement monotonicity conditions (along with various technical regularity requirements), the master equation associated to the model has a smooth solution, and then follow the arguments in \cite{CDLL} to deduce a quantitative convergence result for closed-loop equilibria. We merely wish to illustrate the strength of our main estimates, by pointing out that when specialized to the mean field setting, they provide an alternative, self-contained route to quantitative convergence for closed-loop equilibria.

\subsubsection{Pairwise interactions} \label{subsec.pairwise}
For another simple benchmark, consider the case 
\begin{equation} \label{ex.network}
F^i(\bx) = \sum_{j \neq i} w_{ij} f(x^i-x^j), \qquad G^i(\bx) = \sum_{j \neq i} w_{ij} g(x^i-x^j), 
\end{equation}
where $f,g : \R^d \to \R$ are $C^2$, convex, even, and Lipschitz continuous, and $w = (w_{ij})_{i,j = 1,\dots,N} \in \R^{N \times N}$ is symmetric and has non-negative entries. Suppose that there are constants $C_0, C_1$ such that
\begin{equation*} %\label{f0g0.conditions}
    0 \leq D^2 f \leq C_0, \quad  0 \leq D^2 g \leq C_0, \quad  \|D f\|_{\infty} + \|D g\|_{\infty} \leq C_1.
\end{equation*}
Then one can check (following, e.g., the computations in the proof of \cite[Corollary 6.3]{JacksonLacker}) that the condition \eqref{disp.intro.FG} holds, with the constants $C_{\bm{F}}$ and $C_{\bm{G}}$ depending only on the operator norm $|w|_{\mathrm{op}}$ of $w$ (see \eqref{eq:normsdef}). Moreover, one can check that in this case 
\begin{equation*}
    \delta^i + \kappa^i \leq C \sum_{j \neq i} |w_{ij}|^2, \quad C = C(C_0,C_1).
\end{equation*}
Theorem \ref{thm.oldistdisp} thus implies the following.
\begin{corollary} \label{cor.network1}
    Let $F^i$ and $G^i$ be given by \eqref{ex.network}, with $f$, $g$, and $w$ as above. Moreover, let $(L^i)_{i = 1,\dots,N}$ be $C^2$ functions satisfying \eqref{disp.intro.L}. Then, there is a constant $C > 0$ (which depends only on $C_{\bm{L}}$, $C_0$, $C_1$, $d$, and $T$) such that
      \begin{equation*} 
        \E\bigg[ \sup_{t \in [0,T]} \bigl|\bX_t^{\ol} - \bX_t^{\Dis} \bigr|^2 \bigg] \leq C \sum_{\substack{i,j\\i \neq j}} |w_{ij}|^2.
    \end{equation*}
\end{corollary}
Thus distributed and open-loop equilibria are close under the condition $\sum_{i,j:\, i \neq j} |w_{ij}|^2 = o(N)$. Meanwhile, for closed-loop equilibria, \Cref{thm.olcldisp} yields:
\begin{corollary}
    \label{cor.network2} Grant the conditions of \Cref{cor.network1}. Then, there are constants $\eps, C  > 0$ (which depend only on $C_{\bm{L}}$, $C_0$, $C_1$, $d$, and $T$) such that the following holds: if 
    \begin{equation*}% \label{weak.network}
    \max_{i = 1,\dots,N} \sum_{j \neq i} |w_{ij}|^2 \cdot \sum_{\substack{i,j\\i \neq j}} |w_{ij}|^2 < \eps, 
\end{equation*}
then 
\begin{equation*}
     \E\bigg[ \sup_{t \in [0,T]} \bigl|\bX_t^{\cl} - \bX_t^{\ol} \bigr|^2 \bigg] \leq C\max_{i = 1,\dots,N} \sum_{j \neq i} |w_{ij}|^2 \cdot \sum_{\substack{i,j\\i \neq j}} |w_{ij}|^2.
\end{equation*}
\end{corollary}
In particular, we see that open-loop and closed-loop equilibria are close when $|w|_{\mathrm{op}} = O(1)$ \emph{and} $\max_{i = 1,\dots,N} \sum_{j \neq i} |w_{ij}|^2 \cdot \sum_{i,j:\, i \neq j} |w_{ij}|^2 = o(1)$.

We do not claim that these estimates are sharp in their dependence on $w$; in particular, it would be interesting to know whether the condition $\max_{i = 1,\dots,N} \sum_{j \neq i} |w_{ij}|^2 \cdot \sum_{i,j:\, i \neq j} |w_{ij}|^2  = o(1)$ can be relaxed. However, it is clear that they go beyond the \emph{mean-field-like} case, in which $w_{ij} = O(1/N)$; as a consequence, we can use these estimates (or their more general versions stated below) to obtain a ``universality" result for MFGs, which states that for games governed by dense enough stochastic matrices, the usual mean field approximation is still valid. In particular, suppose we are given for each $N \in \N$, a matrix $w^N$ satisfying 
\begin{equation} \label{wijN.cond}
    w^N_{ij} = w^N_{ji} \geq 0 \ \ \forall \, i,j = 1,\dots,N, \qquad w^N_{ii} = 0, \quad \sum_{j = 1}^N w^N_{ij} = 1, \ \ \forall \, i = 1,\dots,N.
\end{equation}
In addition, we are given a Lagrangian $L$ and $f,g$ as above.
Denote by $\bX_t^{\ol,N}$ and $\bX_t^{\cl,N}$ the open-loop and closed-loop Nash equilibria, respectively, for the $N$-player game in the case \eqref{ex.network}, with $w = w^N$, $L^i = L$, and let $(X_t^{\mf,i})_{i \in \N}$ be i.i.d.\ copies of the unique equilibrium of the corresponding MFG, which are characterized by the McKean--Vlasov FBSDEs
\begin{equation*}% \label{mf.iid}
\begin{cases}
    \d X_t^{\mf,i} = - D_pH\bigl( X_t^{\mf,i}, Y_t^{\mf,i} \bigr) dt + \sqrt{2}\, \d W_t^i, 
   \vspace{.1cm} \\
    \d Y_t^{\mf,i} = \Bigl( D_x H\bigl(X_t^{\mf,i}, Y_t^{\mf,i}\bigr) - D_x \cF\bigl(X_t^{\mf,i}, \cL(X_t^{\mf,i}) \bigr) \Bigr)\, \d t + Z_t^{\mf,i} \, \d W_t^i, \\[3pt]
    X_0^{\mf,i} = \zeta^i, \quad Y_T^i = D_x \cG\big( X_T^{\mf,i}, \cL(X_T^{\mf,i}) \big).
\end{cases}
\end{equation*}
We note that the $(X_t^{\mf,i})_{i \in \N}$ are i.i.d.\ with common law $m_t$, $(m_t)_{t \in [0,T]}$ being the Nash equilibrium of the corresponding MFG; we refer to \cite[Section 4.5]{cardelar} for details.

\begin{theorem} \label{cor.network3}
    Suppose that $L$ is $C^2$ and satisfies \eqref{disp.intro.L}, and for each $N \in \N$, that $(\zeta^i)_{i \in \N}$ are i.i.d.\ with common law $m \in \cP_p(\R^d)$, for some $p > 4$. Suppose moreover that $f,g$ are as described above, and $w^N$ satisfies \eqref{wijN.cond} for each $N$.
    Then, with implied constants independent of $N$, we have\footnote{Here we use the notation $|w_{i\,\cdot}^N|^2 = \sum_j |w_{ij}^N|^2 = \sum_{j \neq i} |w_{ij}^N|^2$, while $\bd_2$ is the $2$-Wasserstein distance, and $m^N_{\bx} \defeq \frac1N\sum_{j} \delta_{x^j}$. We also recall that $X \lesssim Y$ means that $X \leq CY$ for some positive constant $C$ which is referred to as the \emph{implied constant}.}  
    \begin{equation*}
    \E\Big[ \sup_{t \in [0,T]} \bigl|\bX_t^{\ol,N} - \bX_t^{\mf}\bigr|^2 \Big] \lesssim \sum_{i = 1}^N \rho_d\bigl( |w_{i\,\cdot}^N|^{-2} \bigr), 
    \end{equation*}
    with
    \begin{equation} \label{est.mf.intro}
    \rho_d(K) \coloneqq \begin{cases}
       K^{-1/2} & d < 4 
        \\
        K^{-1/2} \log( 1 + K ) & d = 4
       \\
       K^{-2/d} & d > 4,
       \end{cases}
    \end{equation}
    and as a consequence 
    \begin{equation*}% \label{est.mf.intro}
    \E\Big[ \sup_{t \in [0,T]} \bd_2\big( m_{\bX_t^{\ol,N}}^N, m_t \big)^2 \Big] \lesssim \rho_d(N) + \frac{1}{N} \sum_{i = 1}^N  \rho_d\big( |w_{i\,\cdot}^N|^{-2} \big).  
\end{equation*}
   Moreover, if we know in addition that
    \begin{equation*}
    \max_{i=1,\dots,N}|w^N_{i\,\cdot}|^2 \cdot \sum_{i=1}^N |w^N_{i\,\cdot}|^2 \,\xrightarrow{N \to \infty}\, 0, 
    \end{equation*}
    then we have
    \begin{equation*}
         \E\bigg[ \sup_{t \in [0,T]} \bigl|\bX_t^{\cl,N} - \bX_t^{\ol,N} \bigr|^2 \bigg] \lesssim \sum_{i = 1}^N \rho_d\bigl( |w_{i\,\cdot}^N|^{-2} \bigr) + \max_{i=1,\dots,N}|w^N_{i\,\cdot}|^2 \cdot \sum_{i=1}^N |w^N_{i\,\cdot}|^2 ,
    \end{equation*}
    and as a consequence
    \begin{equation*}% \label{est.mf.intro3}
    \E\Big[ \sup_{t \in [0,T]} \bd_2\big( m_{\bX_t^{\cl,N}}^N, m_t \big)^2 \Big] \lesssim \rho_d(N) + \frac{1}{N} \sum_{i = 1}^N  \rho_d\big( |w_{i\,\cdot}^N|^{-2} \big) + \frac{1}{N} \max_{i=1,\dots,N}|w^N_{i\,\cdot}|^2 \cdot \sum_{i=1}^N |w^N_{i\,\cdot}|^2 .
\end{equation*}
\end{theorem}

This result says that open-loop equilibria converge as $N \to \infty$ to the usual mean field limit provided that 
    $\frac{1}{N} \sum_{i = 1}^N \rho_d(|w_{i\,\cdot}^N|_2^{-2}) = o(1)$ as $N \to \infty$,
and closed-loop equilibria converge provided that, \emph{in addition},
    $ \max_{i=1,\dots,N}|w^N_{i\,\cdot}|^2 \cdot \sum_{i=1}^N |w^N_{i\,\cdot}|^2 = o(1)$.
    
It is worth noting that if $w^N = (1/{K_N}) \Theta^N$, where $\Theta^N$ is the adjacency matrix of a $K_N$-regular graph, then we find $|w_{i\,\cdot}^N|_2^2 = 1/K_N$, so we get convergence of the open-loop equilibria as soon as $K_N \to \infty$, and convergence of the closed-loop equilibria provided that ${\sqrt{N}}/{K_N} = o(1)$.

The more general version is \Cref{thm.universality}, which also handles the case 
\begin{equation} \label{graph.intro} \tag{MF\textsubscript{net}}
 L^i = L, \quad    F^i(\bx) = \cF\Bigl(x^i, \sum_{j=1}^N w_{ij}^N \delta_{x^j} \Bigr)\,, \quad G^i(\bx) = \cG\Bigl(x^i,\sum_{j=1}^N w_{ij}^N \delta_{x^j}\Bigr)\,, 
\end{equation}
with $w$ not necessarily symmetric, under either displacement or Lasry--Lions semi-monotonicity conditions on $F^i$ and $G^i$. The only two previous works we are aware of which study this universality property are \cite{DelarueESAIM, LackerSoret}, both of which require a linear-quadratic structure.

\subsection{Approximate independence in high-dimensional systems} 

\label{subsec.approxind}

Our estimates can be viewed as part of a growing literature which seeks to answer, in a variety of settings, the following question:
\begin{displayquote}
\emph{Under what conditions can a high-dimensional system be approximated by one with independent components?}
\end{displayquote}
One goal of this paper is to provide an answer to this question when the ``system" of interest is a stochastic differential game. In order to motivate this point of view, we discuss some recent answers to the same question in two related settings.

\subsubsection{Mean field approximation}

Given a function $f \colon (\R^d)^N \to \R$, consider the \emph{Gibbs measure}
\begin{equation*}
    P(\d\bx) \coloneqq Z^{-1} e^{f} \mu^{\otimes N}(\d\bx), \quad Z \coloneqq \int_{(\R^d)^N} e^f \,\d \mu^{\otimes N}. 
\end{equation*}
A natural question is the following: under what conditions on $f$ and $\mu$ can $P$ be closely approximated by a product measure? When the distance is measured by the \emph{relative entropy} $H(\,\cdot\mid P)$, one is led to the \emph{entropic projection problem}
\begin{equation} \label{entropic.proj}
    \inf_{Q \in \cP_{\mathrm{pr}}} H( Q \mid P ),
\end{equation}
where $\cP_{\mathrm{pr}}=\cP_{\mathrm{pr}}((\R^d)^N)$ is the set of product probability measures on $(\R^d)^N$. This problem has received significant study in recent years because it is equivalent to the mean field approximation of the log partition function $\log Z$, which plays a key role in the emerging theory of non-linear large deviations as introduced in \cite{Dembo}; see \cite{LackerMukherjeeYeung} for further discussion and a summary of the relevant literature. In statistics, problem \eqref{entropic.proj} is often called \emph{mean field variational inference}, and has applications in Bayesian inference, as discussed in \cite{Bhattacharya}.

There is one setting in which it is well known that $P$ is close to a product measure: in the \textit{mean field case}, where 
\begin{equation*}
    f(\bx) =  \sum_{i = 1}^N \cF\big(x^i, m_{\bx}^N\big)
\end{equation*}
for some $\cF : \R^d \times \cP(\R^d) \to \R$. In this case, it is known that under appropriate conditions on $\mu$ and $\cF$, $P$ is well approximated by the product measure $\nu^{\otimes N}$, where $\nu$ solves the fixed-point relation 
\begin{equation*}% \label{MF.gibbs}
    \nu(\d\bx) = \frac{1}{Z} \exp\big( \cF(x,\nu) \big) \mu(\d\bx).
\end{equation*}
Addressing the entropic projection problem \eqref{entropic.proj} \textit{without} any asymptotic structure is the focus of much of the literature on nonlinear large deviations; for example, \cite{LackerMukherjeeYeung} has a result of this type, which states that, at least for $d = 1$, when $\mu$ is strongly log-concave and $f$ is concave, there is an estimate of the form 
\begin{equation} \label{mfapprox.gibbs}
 \inf_{Q \in \cP_{\text{pr}}(\R^N)} H(Q \mid P ) \lesssim \Big\|  \sum_{i,j:\, i \neq j} |D_{x^i x^j} f|^2 \Big\|_{\infty}, 
\end{equation}
where the implied constant is \emph{dimension-free} in an appropriate sense.\footnote{In fact the estimate in \cite{LackerMukherjeeYeung} is slightly stronger, in that it does not require an $\linf$ estimate on $D_{x^i x^j} f$, but we prefer to present a simplified version.} Since the left-hand-side of \eqref{mfapprox.gibbs} is a priori of order $N$, this can be interpreted as meaning that $Q$ is close to a product measure when 
\begin{equation} \label{littleo.1}
    \Big\| \sum_{i,j:\,i \neq j} |D_{x^i x^j} f|^2 \Big\|_{\infty} = o(N),
\end{equation}
even without any asymptotic structure on $f$.

\subsubsection{Independent projection of a high-dimensional diffusion}

The mean field approximation discussed above is a static problem; it is also natural to ask similar questions in a dynamic setting. For example, consider an $(\R^d)^N$-valued diffusion $\bX = (X^1,\dots,X^N)$, which evolves according to 
\begin{equation*} 
    \d X_t^i = b^i(\bX_t)\,\d t + \sqrt{2}\,\d W_t^i \quad ( t \geq 0 ), \qquad X_0^i = \xi^i, \qquad i = 1,\dots,N, 
\end{equation*}
where $\xi^1,\dots,\xi^N$ are some independent (but not identically distributed) initial conditions. 
A natural question is the following: under what conditions on $\bm{b} = (b^1,\dots,b^N) \colon (\R^d)^N \to (\R^d)^N$ can $\bX$ be well-approximated by a diffusion which has independent components? When $\bm{b} = D f$, for $f \colon (\R^d)^N \to \R$ satisfying appropriate conditions, %\marginpar{\tiny DF: Isn't appropriate conditions a bit too vague here? {I believe for qualitative convergence, it is sufficient to have $f$ concave and $e^f$ integrable, we could say this if you want. Or strongly concave implies a quantitativee convergence}} 
the invariant measure of $\bX$ is exactly the Gibbs measure $Z^{-1} e^f \d \bx$, so this is a dynamic analogue of the question discussed in the preceding subsection (up to the fact that the reference measure is the Lebesgue measure, rather than a probability measure). 

Again, the answer is classical in the \textit{mean field case}, which in this setting means that 
\begin{equation*}
    b^i(\bx) = B(x^i,m_{\bx}^N), \quad \cL(\bm\xi) = \mu^{\otimes N},
\end{equation*}
for some $B : \R^d \times \cP(\R^d) \to \R^d$ and $\mu \in \cP(\R^d).$ In this case, we expect that $\cL\big(\bX_t\big)$ is well-approximated by $m_t^{\otimes N}$, where $m_t = \cL(Y_t)$, and $Y$ solves the McKean--Vlasov equation 
\begin{equation*}
    \d Y_t = B(Y_t, \cL(Y_t))\, \d t + \sqrt{2}\, \d W_t \quad (t \geq 0), \qquad \cL(Y_0) = \mu
\end{equation*}
thus, $\bX$ has approximately independent components when $N$ is large, thanks to the phenomenon of propagation of chaos.

In \cite{Lackerindependentproj}, a canonical approximation of $\bX$ with independent components was studied, and termed the \textit{independent projection}; it is defined by the McKean--Vlasov equation
\begin{equation*} 
    \d Y_t^i = \E\bigl[ b^i(\bY_t) \mid Y_t^i \bigr]\, \d t + \sqrt{2}\, \d W_t^i \quad (t \geq 0), \qquad Y_0^i = \xi^i, \qquad i = 1,\dots,N. 
\end{equation*} 
We note that \cite[Theorem~2.11]{Lackerindependentproj} implies that under appropriate conditions on $\bmb$ and $m^i = \cL(X^i)$, we have an estimate of the form 
\begin{equation} \label{independent.proj}
    H\bigl( \cL(\bY_t) \mid  \cL( \bX_t) \bigr) \lesssim \Bigl\|\sum_{i,j:\,i \neq j} \big|D_{x^j} b^i\big|^2 \Big\|_{\infty},  
\end{equation}
with implied constant which depends on $\bm{b}$ only through its Lipschitz constant (which is $O(1)$ as $N\to\infty$ in many cases of interest). Since $ H( \cL(\bY_t) \mid \cL( \bX_t) ) $ is a priori of order $N$, this can be interpreted to mean that (the laws of) $\bX$ and $\bY$ are close when 
\begin{equation} \label{littleo.2}
    \Big\|\sum_{i,j:\,i \neq j} \big|D_{x^j} \bm{b}\big|^2 \Big\|_{\infty} = o(N).
\end{equation} 

In fact, the independent projection was already introduced in \cite{JabinPoyatoSoler}, where the proximity between $\bX$ and $\bY$ was used in as a key step on the way to deriving the asymptotics of a heterogenous particle system. Moreover, the recent work \cite{LackerYeungZhou} established sharp, local versions of \eqref{independent.proj} in the case of pairwise interactions of the form $b^i(\bx) = V(x^i) + \sum_{j \neq i} \xi_{ij} K(x^i,x^j)$.

\subsubsection{Our perspective on stochastic differential games}

The examples discussed above share key features: in both cases, there is a natural \emph{mean field case}, and classical arguments show that in that case, the model is well-approximated by ``independent copies" of its mean field limit. To go beyond the mean field case, a more general approximation of the model by one with independent components is proposed: the entropic projection onto the space of product measures or the independent projection of a diffusion. The main results are non-asymptotic estimates (\eqref{mfapprox.gibbs} and \eqref{independent.proj}) which identify simple conditions (\eqref{littleo.1} and \eqref{littleo.2}) under which this model with independent components is a good approximation of the original model. We view estimates like those appearing in \Cref{cor1.intro,cor2.intro} as natural game-theoretic analogues of \eqref{mfapprox.gibbs} and \eqref{independent.proj} with the conditions
\begin{equation*}
    \max_{i = 1,\dots,N} \delta^i \cdot \sum_{i = 1}^N \delta^i = o(1), \qquad \sum_{i = 1}^N \kappa^i = o(N)
\end{equation*}
replacing \eqref{littleo.1} and \eqref{littleo.2}. We also note that the main results of \cite{JacksonLacker} also contain control-theoretic results of the same flavor.

\subsection{A joint vanishing viscosity and large \texorpdfstring{$N$}{N} limit} 

We now discuss a further application of our estimates in the setting of MFG limit theory. In most of our arguments, we rely heavily on the non-degenerate idiosyncratic noise, i.e.\ the fact that $\sigma > 0$. However, we are careful to track the dependence on $\sigma$ in all of our main non-asymptotic bounds, by specifying how every estimate depends both on $\sigma$ and on $N$.

Our motivation for this is as follows: it turns out that the master equation (which is the expected limit of the $N$-player Nash system in an appropriate sense) is known to be well-posed even when $\sigma = 0$, provided that appropriate monotonicity conditions are met; in the displacement monotone case, \cite{BMMdet} obtains classical solutions without idiosyncratic noise, while in the Lasry--Lions monotone case, a notion of weak solutions is proposed in \cite{cardsougmonotone}. On the other hand, we are not aware of any results on the well-posedness of the Nash system \eqref{nash} when $\sigma = 0$, and it is not clear whether monotonicity alone is sufficient to obtain well-posedness in this case; thus, it is difficult to make sense of the (closed-loop) convergence problem for MFGs with $\sigma = 0$, because even if the limiting master equation is well-posed, we do not know anything about existence or uniqueness of Nash equilibria for the finite-player game.

One way around this technical issue is to choose a sequence $\sigma_N \downarrow 0$, and argue that if $\sigma_N$ decreases slowly enough, then finite-player games with idiosyncratic noise of intensity $\sigma_N$ converge to the MFG with zero idiosyncratic noise.
We use our non-asymptotic bounds to carry out such a program in the displacement monotone case in \Cref{sec.jointlimit}, where we show that if $\sigma_N$ decreases to $0$ slower than $1/\sqrt{N}$, then in fact the closed-loop equilibria converge. We note that obtaining an analogous result in the Lasry--Lions monotone seems to be more challenging, and we leave this question to future work.

\subsection{Further prospects} \label{subsec.furtherprospects} Perhaps the most interesting question left open by this paper is whether the ``weak interaction" condition \eqref{weakinteraction.intro} is sharp. In particular, when applied to network-based case \eqref{graph.intro} with $w$ the adjacency matrix of a regular graph, the condition \eqref{weakinteraction.intro} boils down to $\deg^{-1} = o(N^{-1/2})$, where $\text{deg}$ denotes the common degree of the vertices. Thus, for regular graphs, our universality results require that the degree grows faster than $\sqrt{N}$; meanwhile, the results of \cite{LackerSoret} suggest that universality should hold as long as $\deg^{-1} = o(1)$, so, while the condition \eqref{weakinteraction.intro} is much less stringent than a mean-field-like scaling, it is possible that one can further weaken it. 

%\marginpar{\tiny DF: I commented out the paragraph on where we are losing the $\sqrt N$ factor, since it referenced the outline section, which is now commented out as well. {Good call.}}
%In fact, we can even point to a place in the proof of \eqref{l2bound.intro} where we may be losing a factor of $\sqrt{N}$, namely when we bound the operator norm of the matrix $\bigl(D_{x^jx^i}u^i - D_{x^j} v^i\bigr)_{i,j =1 ,\dots,N}$ by its Frobenius norm in the proof of \eqref{implication.intro}. However, the bound on the Frobenius comes very naturally from the noise, upon expanding $\sum_{1 \leq i \leq N} |v^i - D_{x^i} u^i|^2$ along an equilibrium trajectory, and we do not know how to leverage the noise to obtain a bound directly on the operator norm (in an $L^2$ sense along optimal trajectories).

{We would also like to note that our study of distributed equilibria makes use of the particular dynamics \eqref{dynamics.intro} of the game. If instead, we considered dynamics of the form
\begin{equation*}
    \d X_t^i = b^i\bigl(\bX_t, \alpha^i(t,\bX_t) \bigr)\,\d t + \sqrt{2}\, \d W_t^i, 
\end{equation*}
it would be harder to work with distributed equilibria, since the admissible trajectories for the distributed formulation would no longer have independent components. It would be interesting to derive an FBSDE or PDE description for distributed equilibria for such a model, and prove that under appropriate conditions, equilibria are \textit{approximately} independent, and \cite[Theorem~2.11]{Lackerindependentproj} could be helpful in this regard.
}

{In addition to these main points, several other technical improvements are possible; for example, in the Lasry--Lions monotone case it should be possible to allow a non-constant idiosyncratic noise.}

\subsection{Organization of the paper}
In Section~\ref{sec.formulation}, we discuss our notations and main assumptions, and state precisely the main results of the paper. Sections \ref{subsec.pontryagin.disp} and \ref{subsec.pont.LL} contain our analysis of the (open-loop) Pontryagin system in the displacement and Lasry--Lions semi-monotone cases, respectively; sections \ref{subsec.nash.disp} and \ref{subsec.nash.LL} deal with the (closed-loop) Nash system in those two regimes. \Cref{sec.olcl} contains the comparison between different formulations of the game. \Cref{sec.universality,sec.jointlimit} contain the main applications to mean field games: the former is concerned with universality results for the graph-based model \eqref{graph.intro}, and the latter contains our analysis of the joint vanishing viscosity and large $N$ limit (in the displacement semi-monotone case).

\subsection*{Acknowledgements} M.C.\ and D.F.R.\ are partially supported by the Gruppo Nazionale per l’Analisi Matematica, la Probabilit\`a e le loro Applicazioni (GNAMPA) of the Istituto Nazionale di Alta Matematica (INdAM), and by the project funded by the EuropeanUnion–NextGenerationEU under the National Recovery and Resilience Plan (NRRP), Mission 4 Component 2 Investment 1.1 -- Call PRIN 2022 No.\ 104 of February 2, 2022 of Italian Ministry of University and Research; Project 2022W58BJ5 (subject area: PE -- Physical Sciences and Engineering) ``PDEs and optimal control methods in mean field games, population dynamics and multi-agent models''. M.C.\ partially supported by the King Abdullah University of Science and Technology (KAUST) project CRG2021--4674 “Mean-Field Games: models, theory, and computational aspects”. J.J.\ is supported by the NSF under Grant No.\ DMS2302703.

\section{Notation, preliminaries, and main results}
 \label{sec.formulation}

\subsection{Notation}
We start by explaining the main notations we will be using throughout the paper.

\subsubsection{Vectors and matrices}
We work extensively with the space $(\R^d)^N$. The typical element is $\bx = (x^1,\dots,x^N)$,
and if necessary we can expand each coordinate as $x^i = (x^i_1,\dots,x^i_d) \in \R^d$.
We use $|\cdot|$ for the usual Euclidean ($\ell^2$) length on both $\R^d$ and $(\R^d)^N$, so that we can write $|\bx|^2 = \sum_{i=1}^N |x^i|^2$.
Other $\ell^p$ norms will be denoted by $|\cdot|_p$ (we will only use $|\cdot|_1$, in \Cref{sec.universality}). The transpose of $\bx$ will be denoted by $\bx^\TT$, and the same notation will be used for matrices as well.

At times, we will work with matrices in $(\R^{d \times d})^{N \times N}$, which we identify with $\R^{Nd \times Nd}$ in a natural way. In particular, given $\bm B = (B^{ij})_{i,j = 1,\dots,N} \in (\R^{d \times d})^{N \times N}$ (that is, each $B^{ij} \in \R^{d \times d}$) we will identify $\bm B$ with the element which is written in blocks as 
\begin{equation*}
    \begin{pmatrix} 
    B^{11} & \dots  & B^{1N}\\
    \vdots & \ddots & \vdots\\
    B^{N1} & \dots  & B^{NN} 
    \end{pmatrix} \in (\R^{d\times d})^{N \times N} \simeq \R^{Nd \times Nd} \,.
\end{equation*}
The $d \times d$ identity matrix will be denoted by $I_d$, and accordingly $I_{Nd}$ will be the identity in $(\R^{d\times d})^{N \times N} \simeq \R^{Nd \times Nd}$. 
The operator norm of a matrix will be denoted by $|\cdot|_{\op}$, and the Frobenius norm will be denoted by $|\cdot|_{\mathrm{Fr}}$; note that we can write such norms of a matrix $\bm B \in (\R^{d \times d})^{N \times N}$ as
\begin{equation}\label{eq:normsdef}
    |\bm B|_{\op}^2 = \sup_{\substack{\bxi \in (\R^d)^N \\ |\bxi|^2 = 1}} \sum_{i=1}^N \,\biggl| \sum_{j=1}^N B^{ij} \xi^j \biggr|^2\,, \qquad |{\bm B}|_{\mathrm{Fr}}^2 = \sum_{i,j=1}^N |B^{ij}|_{\mathrm{Fr}}^2 = \sum_{i,j=1}^N \sum_{p,q=1}^N |(B^{i,j})_{p,q}|^2\,, 
\end{equation}
where $B^{ij} \xi^j$ indicates the usual multiplication of a matrix and a vector. We recall the standard fact that 
    $| \bm B |_{\op} \leq | \bm B |_{\mathrm{Fr}} \leq \sqrt{Nd}\,| \bm B |_{\op}$
for $\bm B \in (\R^{d \times d})^{N \times N}$, although we will use only the first of these inequalities in the present paper.

\subsubsection{Space and time derivatives}
Given a function $\phi \colon (\R^d)^N \to \R$, we denote by $D_i \phi = D_{x^i} \phi$ the gradient in the variable $x^i \in \R^d$, so
$D_{i} \phi = (D_{i_1}\phi ,\dots,D_{i_d} \phi) = (D_{x^i_1}\phi ,\dots,D_{x^i_d} \phi)$
takes values in $\R^d$.
Likewise, $D_{ji}\phi=D_{x^j x^i} \phi = (D_{j_r i_q} \phi)_{q,r,=1,\dots,d}$
will be the $d \times d$ matrix of second derivatives (with respect to $x^j$ and $x^i$). Note that since this is not a symmetric matrix, we are making a choice here, which we feel is natural because with this definition $D_{ji} \phi$ is the Jacobian matrix in the variable $x^j$ of the $\R^d$-valued function $D_{i} \phi$. 

In particular, when differentiating the Nash system \eqref{nash} or the Pontryagin PDE system \eqref{pontpde}, whose solutions are denoted by $u^i$ and $v^i$, respectively, we will have $D_{kj} u^i$ and $D_{j} v^i$ are $\R^{d \times d}$-valued functions, which write in components as
    $(D_{kj} u^i)_{qr} = D_{k_r j_q} u^i$ and $(D_j v^i)_{q,r} = D_{j_r} v^{i_q}$, with $q,r = 1,\dots,d$.
%The same convention also applies to mixed derivatives of the other functions, such as the Lagrangians $L^i$ or the Hamiltonians $H^i$.

We also write $\Delta_i = \Delta_{x^i}$ for the Laplacian in the variable $x^i$, i.e.\ $\Delta_{i} \phi = \sum_{q=1}^d D_{i_q i_q} \phi$.
Finally, if $\phi$ is (also) a function of time $t \in [0,T]$ we will denote its (partial) derivative with respect to $t$ by $\partial_t \phi$.

\subsubsection{Measures and derivatives}
We denote by $\cP(\R^d)$ the space of Borel probability measures on $\R^d$, and by $\cP_2(\R^d)$ the Wasserstein space of Borel probability measures with finite second moment, equipped with the usual $2$-Wasserstein distance $\bd_2$. In \Cref{sec.universality}, we will work with functions $\cG \colon \R^d \times \cP_2(\R^d) \to \R$, and we will employ the calculus on the space of measures which is commonly used in mean field game theory; see, e.g., \cite{CDLL,cardelar} for details. In particular, we will write $D_m \cG$ for the ``Wasserstein" (or ``Lions", or ``intrinsic'') derivative, which is a map $D_m \cG \colon \R^d_x \times \cP_2(\R^d) \times \R^d_y \to \R^d$, and satisfies $D_m \cG = D_y \frac{\delta \cG}{\delta m}$, where $\frac{\delta \cG}{\delta m}$ is the ``linear'' (or ``flat'') derivative of $\cG$.

\subsubsection{Recurrent norms of (second) derivatives}\label{subs:not}
We will often need to indicate upper bounds on (the derivatives of) vector of functions $\bm\phi = (\phi^i)_{i=1,\dots,N} \colon \R^d \to \R^d$ which hold uniformly over $i$. To do this, we will use notation like
   $     \|D_{xx} \bm\phi\|_{\infty} = \max_{i=1,\dots,N} \| |D_{xx} \phi^i |_{\op} \|_{\infty}$, %, \quad  \|D_{xa} L\|_{\infty} = \max_{i = 1,\dots,N} \| |D_{xa} L^i |_{\op} \|_{\infty}
    where $\|\cdot\|_\infty$ denotes the usual supremum ($L^\infty$) norm.
    As suggested above, we will use the operator norm when measuring the size of second derivatives, unless otherwise noted. For $\bm\phi \colon \R^d \times \R^d \to \R^d$, $\bm\phi = \bm\phi(x,y)$, we will also use notation like
    $    \|D^2 \bm\phi \|_{\infty} = \|D_{xx} \bm\phi\|_{\infty} + \|D_{xy} \bm \phi \|_{\infty} + \|D_{yy} \bm \phi \|_{\infty}$
     to describe a bound on the second-order derivatives of $\phi^i$ which is uniform over $i$.

\subsubsection{Key matrix-valued functions}
In \Cref{subsec.pont.LL}, we will work extensively with the $(\R^{d \times d})^{N \times N}$-valued function $D\bv \coloneqq (D_j v^i)_{i,j = 1,\dots,N}$.
We note that after making the identification of $(\R^{d \times d})^{N \times N}$ with $\R^{Nd \times Nd}$, $D \bv(t,\cdot)$ is nothing but the Jacobian matrix of the map 
\begin{equation*}
     \R^{Nd} \simeq  (\R^d)^N  \ni \bx \mapsto \bv(t,\bx) = \bigl(v^1(t,\bx),\dots,v^N(t,\bx)\bigr) \in (\R^d)^N \simeq \R^{Nd}\,,
\end{equation*}
and so in particular the Lipschitz constant of $\bv$ is equal to the supremum of the operator norm of $D\bv$, i.e.\
\begin{equation*}
    \sup_{\substack{t \in [0,T],\, \bx,\by \in (\R^d)^N \\ \bx \neq \by}} \frac{|\bv(t,\bx) - \bv(t,\by)|}{|\bx - \by|} = \sup_{t \in [0,T],\, \bx \in (\R^d)^N} | D\bv(t,\bx) |_{\op}\,. 
\end{equation*}
Similarly, in Sections \ref{subsec.nash.disp} and \ref{subsec.nash.LL}, we will work extensively with the $(\R^{d \times d})^{N \times N}$-valued function
\begin{equation} \label{def.aij}
    \bA \defeq \bigl( D_{ji} u^i \bigr)_{i,j=1,\dots,N}\,,
\end{equation}
which, after making the identification of $(\R^{d \times d})^{N \times N}$ with $\R^{Nd \times Nd}$, is nothing but the Jacobian matrix of the map 
\begin{equation*}
    \R^{Nd} \simeq  (\R^d)^N  \ni \bx \mapsto \bigl(D_{1}u^1(t,\bx),\dots,D_{N} u^N(t,\bx)\bigr) \in (\R^d)^N \simeq \R^{Nd}\,.
\end{equation*}
These are completely elementary facts, but we state them explicitly here to clarify our choice of notation when working with elements of $(\R^{d \times d})^{N \times N}$.

\subsubsection{Further recurring notation} 
It will often be important to distinguish between diagonal and off-diagonal terms, for example of $\bA$ defined in \eqref{def.aij}. Given a set $S$ (typically $S = \N$), we will denote by $\1_{i=j}$ the indicator function of the set $\{(k,h) \in S :\, k=h\}$ computed at $(i,j)$, i.e.\ $\1_{i=j} \in \{0,1\}$ and $\1_{i=j} = 1$ if and only if $i=j$. Similarly, $\1_{i \neq j}$ will be the indicator function of the complement, i.e.\ $\1_{i \neq j} = 1 - \1_{i = j}$.

Also, we will often be working with differences of values of the same function at different points, so we find it convenient to adopt the following compact notation: given a function $\phi$, we write $\phi\bigr|_y^x \defeq \phi(x) - \phi(y)$.

\subsection{Closed-loop, open-loop, and distributed formulations of the \texorpdfstring{$N$}{N}-player game} \label{ssec.cod}

The $N$-player game of interest is described by functions
\[
    L^i \colon (\R^d)^N \times \R^d \to \R\,, \quad F^i,G^i \colon (\R^d)^N \to \R\,, \quad  i = 1,\dots,N\,.
\]
In addition, we are given constants
\[
    \sigma \in (0,1], \qquad \sigma_0 \in [0,1], \qquad T > 0.
\]
We work on a fixed filtered probability space $\big(\Omega, \scrF, \bbF = (\scrF_t)_{t \in [0,T]} \big)$ which hosts independent $d$-dimensional Brownian motions $(W^i)_{i \in \N}$, and such that $\scrF_0$ is atomless. We note that this ensures that for any $\bm{m}_0 = (m^1_0\dots,m_0^N) \in \cP_2(\R^d)^N$, we can find independent, square-integrable random variables $(\zeta^i)_{i = 1,\dots,N}$ such that $\cL(\zeta^i) = m_0^i$.  We are now going to describe three different ways in which the relevant stochastic differential game can be played.

\subsubsection{The closed-loop formulation}
In the closed-loop formulation of the game, player $i$ chooses a feedback function $\alpha^i(t,\bx) \colon [t_0,T] \times (\R^d)^N \to \R^d$, and the states $\bX = (X^1,\dots,X^N)$ evolve according to 
\begin{equation} \label{cldynamics}
    \d X_t^i = \alpha^i(t,\bX_t) \,\d t + \sqrt{2\sigma}\, \d W_t^i + \sqrt{2 \sigma_0}\, \d W_t^0 \quad (t_0\leq t \leq T), \qquad X_{t_0}^i = \zeta^i_0\,, 
\end{equation}
where $t_0 \in [0,T]$ and each $\zeta^i_0$ is an $\scrF_{t_0}$-measurable, square-integrable random vector taking values in $\R^d$. For simplicity, we assume that each $\alpha^i$ is chosen from $\cA_{t_0}^{\CL}$, the set of measurable maps $\alpha(t,\bx) \colon [t_0,T] \times (\R^d)^N \to \R^d$, satisfying the growth condition $|\alpha(t,\bx)| \lesssim 1 + |\bx|$,
and we recall that for any $(t_0,\bzeta_0) = (t_0,\zeta^1,\dots,\zeta^N)$ as above and $\bm \alpha = (\alpha^1,\dots,\alpha^N) \in \big(\cA_{t_0}^{\CL}\big)^N$, \eqref{cldynamics} has a unique strong solution.
Player $i$ aims to minimize the cost functional 
\begin{equation*}
    J^i_{\CL}(t_0,\bzeta_0, \bm \alpha) \defeq \E\bigg[\int_{t_0}^T L^i\big(\bX_t, \alpha^i(t,\bX_t)\big)\,  \d t + G^i(\bX_T)  \bigg],
\end{equation*}
where implicitly $\bX$ depends on $\bm \alpha$ through the dynamics \eqref{cldynamics}.
A \emph{closed-loop equilibrium} (started from $(t_0,\bzeta_0)$) is a tuple $\bm \alpha = (\alpha^1,\dots,\alpha^N) \in (\cA_{t_0}^{\CL})^N$ such that for each $i = 1,\dots,N$, and each $\beta \in \cA_{t_0}^{\CL}$, 
\begin{equation*}
    J^i_{\cl}(t_0,\bzeta_0, \bm \alpha) \leq J^i_{\cl}\big(t_0,\bzeta_0, (\bm \alpha^{-i}, \beta) \big), 
\end{equation*}
where $(\bm \alpha^{-i}, \beta) \defeq (\alpha^1,\dots,\alpha^{i-1},\beta,\alpha^{i+1},\dots,\alpha^N)$.
It is well-known that closed-loop equilibria are described by the \emph{Nash system}
\begin{equation} \label{nash} \tag{NS}
    \begin{dcases}  - \partial_t u^i - \sigma \sum_{j=1}^N \Delta_{x^j} u^i - \sigma_0 \sum_{j,k=1}^N \tr\big( D_{x^j x^k} u^i \big) + H^i(\bx, D_{x^i} u^i) 
    \\  
    \qquad + \sum_{\substack{j=1 \\ j \neq i}}^N D_{p} H^j(\bx, D_{x^j} u^j) \cdot D_{x^j} u^i = 0 \,, \qquad (t,\bx) \in [0,T] \times (\R^d)^N
    \\[1pt]
    u^i(T,\bx) = G^i(\bx)\,, \qquad \bx \in (\R^d)^N, 
    \end{dcases} 
\end{equation}
where for each $i = 1,\dots,N$, the \emph{Hamiltonian} $H^i \colon (\R^d)^N \times \R^d \to \R$ is given by 
\begin{equation} \label{intro.defHi}
    H^i(\bx,p) \defeq \sup_{a \in \R^d} \bigl( - a \cdot p - L^i(\bx,a) \bigr).
\end{equation}
More precisely, if \eqref{nash} has a sufficiently regular solution $\bm u= (u^1,\dots,u^N) \colon [0,T] \times (\R^d)^N \to \R^N$, then a verification argument shows that
\begin{align*}
    \alpha^i(t,\bx) = - D_pH^i\big(\bx, D_{x^i} u^i(t,\bx) \big)
\end{align*}
is a closed-loop Nash equilibrium (for any initial condition $(t_0,\bzeta_0)$).

\subsubsection{The open-loop formulation}
In the open-loop case, the dynamics are given by 
\begin{equation} \label{oldynamics}
    \d X_t^i = \alpha^i_t \,\d t + \sqrt{2\sigma} \,\d W_t^i + \sqrt{2 \sigma_0} \,\d W_t^0 \quad (t_0\leq t \leq T), \qquad X_{t_0}^i = \zeta^i_0\,, 
\end{equation}
with the $\zeta^i_0$'s as above. {Player $i$ chooses a control $\alpha^i \in \cA_{t_0,\bzeta_0}^{\ol}$, where $\cA_{t_0,\bzeta_0}^{\ol}$ denotes the set of square-integrable, $\R^d$-valued processes $\alpha = (\alpha_t)_{t \in [t_0,T]}$, progressively measurable with respect to the filtration $\bbF^{t_0,\bzeta_0} = (\scrF^{t_0,\bzeta_0}_t)_{t \in [t_0 , T]}$ which is generated by $\bzeta_0$ and $\{W_{\cdot}^i - W_{t_0}^i :\, i = 0,\dots,N \}$.} 
The cost to player $i$ is
\begin{equation*}
    J^i_{\OL}(t_0,\bzeta_0, \bm \alpha) \defeq \E\bigg[\int_{t_0}^T L^i\big(\bX_t, \alpha_t^i \big) \, \d t + G^i(\bX_T)  \bigg],
\end{equation*}
where implicitly $\bX$ depends on $\bm \alpha$ through the dynamics \eqref{oldynamics}. An \emph{open-loop equilibrium} (started from $(t_0,\bzeta_0)$) is a tuple $\bm \alpha = (\alpha^1,\dots,\alpha^N) \in (\cA_{t_0,\bzeta_0}^{\ol})^N$ such that for each $i = 1,\dots,N$, and each $\beta \in \cA_{t_0,\bzeta_0}^{\ol}$, 
\begin{equation*}
    J^i_{\ol}(t_0,\bzeta_0, \bm \alpha) \leq J^i_{\ol}\big(t_0,\bzeta_0, (\bm \alpha^{-i}, \beta) \big)\,.% \quad \text{where} \quad (\bm \alpha^{-i}, \beta)  (\alpha^1,\dots,\alpha^{i-1},\beta,\alpha^{i+1},\dots,\alpha^N)\,.
\end{equation*}
Open-loop Nash equilibria can be described by the \emph{Pontryagin FBSDE system}, a system of forward-backward stochastic differential equations of the form
\begin{equation} \label{pontryagin} \tag{PS}
    \begin{dcases}
        \d X_t^i = - D_p H^i(\bX_t, Y_t^i)\,\d t + \sqrt{2 \sigma}\, \d W_t^i + \sqrt{2 \sigma_0} \,\d W_t^0
        \\
        \d Y_t^i =  D_{x^i} H^i(\bX_t,Y_t^i) \,\d t + \sum_{j=1}^N Z_t^{i,j}\, \d W_t^j 
       \\
       X_{t_0}^i = \zeta^i, \quad Y_{T}^i = D_{x^i} G^i(\bX_T).
    \end{dcases}
\end{equation}
More precisely, if an open-loop equilibrium exists, then it must take the form
\[
\alpha_t^i = - D_p H^i(\bX_t, Y_t^i)
\]
for some solution $(\bX,\bY,\bZ)$ to \eqref{pontryagin}. Actually, open-loop Nash equilibria are also connected to a PDE system, because the system \eqref{pontryagin} is expected to admit a \emph{decoupling field} which is described by the system 
\begin{equation} \label{pontpde} \tag{PS\textsubscript{PDE}}
    \begin{dcases}
        - \partial_t v^i - \sigma \sum_{j=1}^N \Delta_{x^j} v^i - \sigma_0 \sum_{j,k=1}^N \tr\big( D_{x^jx^k} v^i \big) + D_{x^i} H^i(\bx, v^i)
        \\
        \qquad +  \sum_{j=1}^N  D_{x^j} v^i D_pH^j(\bx, v^j) = 0, \qquad (t,\bx) \in [0,T] \times (\R^d)^N
     \\[1pt]
        v^i(T,\bx) = D_{x^i} G^i(\bx)\,, \qquad \bx \in (\R^d)^N\,.
    \end{dcases}
\end{equation}
We will refer to \eqref{pontpde} as the \textit{Pontryagin PDE system}. More precisely, it is expected that the solution of \eqref{pontryagin} should take the form
\begin{equation} \label{pontdecoupling}
    Y_t^i = v^i(t,\bX_t)\,, \quad Z_t^{i,j} = \sqrt{2\sigma} D_{x^j} v^i(t,\bX_t)\,, \quad Z_t^{i,0} = \sqrt{2 \sigma^0} \sum_{j=1}^N D_{x^j} v^i(t,\bX_t).
\end{equation}

\subsubsection{The distributed formulation}
For technical reasons (which are discussed in \Cref{rmk.distcommonnoise}), we will only consider the distributed formulation of the game under the condition $\sigma_0 = 0$. In the distributed formulation, player $i$ chooses a feedback function $\alpha^i(t,x) \colon [t_0,T] \times \R^d \to \R^d$, and the states $\bX = (X^1,\dots,X^N)$ evolve according to 
\begin{equation} \label{distdynamics}
    \d X_t^i = \alpha^i(t,X_t^i) \,\d t + \sqrt{2\sigma} \,\d W_t^i \quad (t_0\leq t \leq T), \qquad X_{t_0}^i = \zeta^i_0\,, 
\end{equation}
where the $\zeta^i_0$'s are as above and also \emph{independent}. For simplicity, we assume that each $\alpha^i$ is chosen from the set $\cA_{t_0}^{\dist}$, the set of measurable maps $\alpha(t,x) \colon [t_0,T] \times \R^d \to \R^d$, satisfying the growth condition $|\alpha(t,x)| \lesssim 1 + |x|$.
Player $i$ aims to minimize the cost functional 
\begin{equation*}
    J^i_{\dist}(t_0,\bzeta, \bm \alpha) \defeq \E\bigg[\int_{t_0}^T L^i\big(\bX_t, \alpha^i(t,X_t^i)\big)\,\d t + G^i(\bX_T)  \bigg],
\end{equation*}
where implicitly $\bX$ depends on $\bm \alpha$ through the dynamics \eqref{distdynamics}. A distributed equilibrium (started from $(t_0,\bzeta_0)$) is a tuple $\bm \alpha = (\alpha^1,\dots,\alpha^N) \in (\cA_{t_0}^{\dist})^N$ such that for each $i = 1,\dots,N$, and each $\beta \in \cA_{t_0}^{\dist}$, 
\begin{equation*}
    J^i_{\dist}(t_0,\bzeta_0, \bm \alpha) \leq J^i_{\dist}\big(t_0,\bzeta_0, (\bm \alpha^{-i}, \beta) \big)\,. 
\end{equation*}
For $i = 1,\dots,N$, we define the distributed Hamiltonian $ \bbH^i : \R^d \times \R^d \times \big(\cP_2(\R^d)\big)^{N-1}$ as
\begin{equation*}
   \bbH^i\bigl(x,p, (m^j)_{j \neq i}\bigr) = \sup_{a \in \R^d} \Bigl\{-a \cdot p - \bbL^i\big(x,a, (m^j)_{j \neq i} \big) \Bigr\},
\end{equation*}
where
\begin{equation} \label{def.distl}
\bbL^i\big(x,a, (m^j)_{j \neq i} \big) \defeq \int_{(\R^d)^{N-1}} L^i\big( (\by^{-i},x), a \big) \prod_{j \neq i} m^j(\d y^j),
\end{equation}
with $(\by^{-i},x) \defeq (y^1,\dots,y^{i-1},x,y^{i+1},\dots,y^N)$.
By fixing $(\alpha^j)_{j \neq i}$ and considering the optimization problem faced by player $i$, one sees that distributed equilibria are described by the PDE system 
\begin{equation} \label{distpde}
    \begin{dcases}
        - \partial_t w^i - \sigma \Delta_x w^i + \bbH^i\big(x,D_x w^i, (m^j_t)_{j \neq i} \big) = 0, \vspace{.1cm}
        \\
        \partial_t m^i - \sigma \Delta_x m^i - \text{div}_x\big( m^i D_p \bbH^i\big(x,D_x w^i,  (m^j_t)_{j \neq i} \big) \big) = 0
        \\[7pt]
        w^i(T,\bx) = \int_{(\R^d)^{N-1}} G^i(\by^{-i},x) \prod_{j \neq i} m_t^{j}(\d y^j)\,, \quad m_{t_0}^i = \cL(\zeta^i_0)\,,
    \end{dcases}
\end{equation}
for $(t,x) \in [t_0,T] \times \R^d$, where  and $\cL(\zeta)$ denotes the law of the random variable $\zeta$. More precisely, given a solution $(w^i,m^i)_{i = 1,\dots,N}$ to \eqref{distpde}, the choice
\[
\alpha^i(t,x) = - D_p \bbH^i\big(x, D_x w^i(t,x), (m_t^j)_{j \neq i} \big)
\]
gives a distributed equilibrium started from $(t_0,\bzeta_0)$. {To see this, %\marginpar{\tiny DF: I think we can also remove this explanation. M: I agree JJ: One of the referees asked for either a proof or a reference, so if we remove I think we would need to find something to reference}
note that if we fix $(\alpha^j)_{j \neq i}$ and let $m_t^j= \cL(X_t^j)$ with $X^j$ defined by \eqref{distdynamics}, then player $i$ faces the optimization problem 
\begin{equation*}
    \inf_{\alpha} \E\bigg[ \int_{t_0}^T \bbL^i\big(X_t,\alpha_t, (m_t^j)_{j \neq i} \big) \d t + G^i(\bX_T) \bigg], 
\end{equation*}
subject to 
\begin{equation*}
    \d X_t^i = \alpha^i(t,X_t^i)\,\d t + \sqrt{2\sigma}\, \d W_t^i, \quad X_{t_0}^i = \zeta^i_0. 
\end{equation*} }
Alternatively, by applying the maximum principle to solve the optimization problem faced by player $i$, we obtain the FBSDE system 
\begin{equation} \label{distfbsde} 
    \begin{dcases}
         \d X_t^i = - D_p \bbH^i\bigl(X_t^i,Y_t^i, \bigl(\cL(X_t^j)\bigr)_{j \neq i} \bigr) \,\d t + \sqrt{2\sigma} \,\d W_t^i
        \\  
        \d Y_t^i = D_x\bbH^i\bigl(X_t^i,Y_t^i, \bigl(\cL(X_t^j)\bigr)_{j \neq i} \bigr) \, \d t + \sum_{j=1}^N Z_t^{i,j}\, \d W_t^j
        \\
        X_{t_0}^i = \zeta^i_0\,, \quad Y_T^i = \E\big[D_{x^i} G^i(\bX_T) \mid X_T^i \big],
    \end{dcases}
\end{equation}
which is connected to \eqref{distpde} formally by $m_t^i = \cL(X_t^i)$ and $Y_t^i = D_x w^i(t,X_t^i)$.

\subsection{Admissible solutions and corresponding equilibria}\label{sec:admiss}

We now define a notion of \emph{admissible solutions} to the PDE systems which we study. Throughout the paper, we will always work with such admissible solutions, even if we do not specify that every time.

\subsubsection{Closed-loop equilibria}
We say that $\bu = (u^1,\dots,u^N)$ is an \textit{admissible solution to the Nash system \eqref{nash}} if $u^i \in C^{1,2}([0,T] \times (\R^d)^N)$ for each $i$, \eqref{nash} is satisfied in a classical sense, $D_{j} u^i \in C^{1,2}([0,T] \times (\R^d)^N ; \R^d)$ for each $i,j = 1,\dots,N$, and we have the estimates 
\begin{align*}
    |u^i(t,\bx)| \leq C(1 + |\bx|^2), \quad |D_{kj} u^i(t,\bx)| \leq C, 
\end{align*}
for some $C > 0$, for each $i,j,k = 1,\dots,N$, and each $(t,\bx) \in [0,T] \times (\R^d)^N$.

Given an admissible solution $\bu$ to the Nash system, and an initial condition $(t_0,\bzeta_0)$ for the game, we denote by $\bX^{\cl,t_0,\bzeta_0}$ the \emph{corresponding closed-loop equilibrium trajectory}, which is the (unique) solution to the SDE system on $[t_0,T]$
\begin{equation} \label{def.xt0x0}
  \d X_t^{\cl, t_0,\bzeta_0,i} = - D_p H^i\big(\bX_t^{\cl, t_0,\bzeta_0}, D_{i} u^i(t,\bX_t^{\cl, t_0,\bzeta_0}) \big) \,\d t + \sqrt{2 \sigma}\, \d W_t^i + \sqrt{2\sigma_0} \,\d W_t^i
\end{equation}
with initial condition $ X_{t_0}^{\cl,t_0,\bzeta_0,i} = \zeta^i_0$ and $i=1,\dots,N$.
When $(t_0,\bzeta_0)$ can be understood from context, we will often write $\bX^{\cl} = \bX^{\cl, t_0,\bzeta_0}$ for simplicity. We note that for much of the paper, we will work with deterministic $\bzeta_0 = \bx_0 \in (\R^d)^N$. 

\subsubsection{Open-loop equilibria}
We say that $\bv = (v^1,\dots,v^N)$ is an \textit{admissible solution to \eqref{pontpde}} if $v^i \in C^{1,2}([0,T] \times (\R^d)^N ; \R^d)$ for each $i$, \eqref{pontpde} is satisfied in a classical sense, $D_{j} v^i \in C^{1,2}([0,T] \times (\R^d)^N ; \R^{d \times d})$ for each $i,j = 1,\dots,N$, and we have the estimates 
\begin{equation*}
    |v^i(t,\bx)| \leq C(1 + |\bx|), \quad |D_{j} v^i| \leq C, \quad |D_{kj} v^i| \leq C, 
\end{equation*}
for some constant $C > 0$, for each $i,j,k = 1,\dots,N$, and for each $(t,\bx) \in [0,T] \times (\R^d)^N$.

Given an admissible solution $\bv$ to \eqref{pontpde}, we denote by $\bX^{\ol,t_0,\bzeta_0}$ the \emph{corresponding open-loop equilibrium trajectory}, which is the (unique) solution to the SDE system on $[t_0,T]$
\begin{align} \label{ol.trajectory} %\label{def.xt0x0ol} 
  \d X_t^{\ol, t_0,\bzeta_0,i} = - D_p H^i\big(X_t^{\ol, t_0,\bzeta_0,i}, v^i(t,\bX_t^{\ol, t_0,\bzeta_0}) \big) \,\d t + \sqrt{2 \sigma} \,\d W_t^i + \sqrt{2\sigma_0} \,\d W_t^i
\end{align}
with initial condition $ X_{t_0}^{\ol,t_0,\bzeta_0,i} = \zeta^i$ and $i=1,\dots,N$.
When $(t_0,\bzeta_0)$ can be understood from context, we will often write $\bX^{\ol} = \bX^{\ol, t_0,\bzeta_0}$ for simplicity. We note that if $\bv$ is an admissible solution and $(t_0,\bzeta_0)$ is fixed, then by It\^o's formula, the processes $\bX = \bX^{\ol,t_0,\bzeta_0}$, $\bY = (Y^i)_{i = 1,\dots,N}$, and $\bZ = (Z^{i,j})_{i = 1,\dots,N, j = 0,\dots,N}$ defined by \eqref{pontdecoupling} satisfy the FBSDE \eqref{pontryagin}.

\begin{remark} \label{rmk.olnecessary}
    We note that in the absence of additional convexity conditions, the Pontryagin FBSDE system \eqref{pontryagin} is a necessary, rather than a sufficient condition for equilibria (see, e.g., \cite{CarmonaLectures}). In particular, it is not immediate that 
    \begin{align} \label{olrmk}
      \alpha_t^i =  - D_p H^i\big(X_t^{\ol, t_0,\bzeta_0,i}, v^i(t,\bX_t^{\ol, t_0,\bzeta_0}) \big)
    \end{align}
    defines an open-loop Nash equilibrium. However, under our Assumptions \ref{assump.disp} and \ref{assump.LL}, the FBSDE \eqref{pontryagin} is known to have a unique solution from any initial condition (see, e.g., \cite{delarue}), and so if any open-loop equilibrium exists it must coincide with \eqref{olrmk}. Thus we can conclude that \eqref{ol.trajectory} describes the unique open-loop equilibrium trajectory, provided that any open-loop equilibrium exists. We ignore this subtlety and just call $\bX^{\ol,t_0,\bzeta_0}$ the corresponding open-loop equilibrium, but the reader should keep in mind that to ensure that $\bX^{\ol,t_0,\bzeta_0}$ is an open-loop equilibrium we must also assume the existence of an equilibrium.
\end{remark}

\subsubsection{Distributed equilibria}
Finally, for $\sigma_0 = 0$ we say that $(\bw,\bm{m}) = (w^1,\dots,w^N, m^1,\dots,m^N)$ is an \emph{admissible solution to \eqref{distpde}} if each $w^i \in C^{1,2}([0,T] \times \R^d)$, the spatial derivatives $Dw^i \in C^{1,2}([0,T] \times \R^d; \R^d)$, each $m^i = (m^i_t)_{t \in [0,T]} \in C([0,T] ; \cP_2(\R^d))$, the equation for $m^i$ is satisfied in a classical sense, the equation for $w^i$ is satisfied in a weak sense, and the estimates
\begin{equation*}
    |w^i(t,x)| \leq C(1 + |x|^2)\,, \quad |D w^i| \leq C(1 + |x|)
\end{equation*}
hold for some constant $C > 0$, each $i = 1,\dots,N$ and each $(t,x) \in [0,T] \times \R^d$.

Given initial conditions $(t_0,\bzeta_0)$, and an admissible solution $(\bw, \bm{m})$ to \eqref{distpde}, we denote by $\bX^{\dist,t_0,\bzeta_0}$ the \textit{corresponding distributed equilibrium trajectory}, which is the (unique) solution to the SDE system on $[t_0,T]$
\begin{equation*} %\label{def.xt0x0dist}
  \d X_t^{\dist, t_0,\bzeta_0,i} = - D_p \bbH^i\Big(X_t^{\dist, t_0,\bzeta_0,i}, Dw^i(t,X_t^{\dist, t_0,\bzeta_0}), (m_t^j)_{j \neq i} \Big) \,\d t + \sqrt{2 \sigma} \,\d W_t^i
\end{equation*}
with initial condition $X_{t_0}^{\dist,t_0,\bzeta_0,i} = \zeta^i$ and $i = 1,\dots,N$.
When $(t_0,\bzeta_0)$ can be understood from context, we will often write $\bX^{\dist} = \bX^{\dist, t_0,\bzeta_0}$ for simplicity. We note that if $(\bw, \bm{m})$ is an admissible solution to \eqref{distpde}, and we set $\bX = \bX^{\dist,t_0,\bzeta_0}$, $\bY = (Y^1,\dots,Y^N)$ with $Y^i = Dw^i(t,X_t^i)$, and $\bZ = (Z^1,\dots,Z^N)$ with $Z_t^i = \sqrt{2\sigma} D w^i(t,X_t^i)$, then It\^o's formula shows that $(\bX,\bY,\bZ)$ satisfies \eqref{distfbsde}, and also $m_t^i = \cL(X_t^i)$.

\smallskip

\begin{remark}\label{rmk:exequilibria}
{Assuming uniform convexity of the Hamiltonian in the $p$ variable, and linear growth with respect to the $x$ variable (together with suitable regularity properties), the existence of equilibria that are admissible in the sense of the previous definitions has been discussed in the literature, see for instance \cite{CDBook} (note that \eqref{distpde} can be thought as a multi-population MFG), or \cite{CR24} and references therein. The quadratic growth allowed here accommodates in fact data having quadratic growth; in this situation, existence of equilibria is more delicate, in particular for the Nash system. We refer to \cite[Appendix A]{JM26} for some recent results in this case.}
\end{remark}

\begin{remark} \label{rmk.admissible}
%\marginpar{\raggedright\tiny Maybe one could ask why we also need to define admissible solutions for distributed equilibria, since we are not say anything about that in this remark.}
    The main reason for the regularity and growth conditions which we put on admissible solutions to the Nash system \eqref{nash} and the Pontryagin PDE system \eqref{pontpde} is that we are going to differentiate these equations and then use It\^o's formula to expand their derivatives along an equilibrium trajectory. Thus we assume that the first (spatial) derivatives of the solutions are smooth enough to apply It\^o's formula, while the bounds on the derivatives of $u^i$ and $v^i$ ensure that the local martingale terms appearing in all of these computations are true martingales. For example, we have
    \[
        \d\, \frac{1}{2} \bigl|D_j u^i(t,\bX_t^{\cl, t_0,\bx_0})\bigr|^2 = \beta_t \,\d t + \d M_t\,,  
    \]
    where $\beta$ is some adapted process and
    the local martingale $M$ is of the form
    \[
        \d M_t = D_j u^i \cdot \biggl(\sqrt{2\sigma} \sum_{k=1}^N D_{kj} u^i \,\d W_t^k + \sqrt{2\sigma_0} \Bigl( \sum_{k=1}^N D_{kj} u^i \Bigr) \d W_t^0 \biggr)\,,
    \]
    with (the derivatives of) $u^i$ evaluated at $(t,\bX_t^{\cl, t_0,\bx_0})$.
    If $(u^1,\dots,u^N)$ is an admissible solution, then $|D_j u^i(t,\bx)| \leq C(1 + |\bx|)$, $|D_{kj} u^i(t,\bx)| \leq C$,
    and also $\bX^{\cl,t_0,\bx_0}$ has a bounded initial condition and satisfies an SDE with a drift which has linear growth; thus we can conclude that, e.g.,
    \[
        \E\biggl[\, \int_{t_0}^T \bigl|D_j u^i(t,\bX_t^{\cl,t_0,\bx_0})\bigr|^2 \bigl|D_{kj} u^i(t,\bX_t^{\cl,t_0,\bx_0})\bigr|^2 \,\d t \biggr] \leq C \E\biggl[\, \int_{t_0}^T \bigl|\bX_t^{\cl, t_0,\bx_0}\bigr|^2 \,\d t \biggr] < \infty\,, 
    \]
    which is enough to conclude that $\E\big[\langle M \rangle_T \big] < \infty$, so that $M$ is a true martingale. In the body of the text, similar local martingale terms will appear repeatedly, and we will not justify each time the fact that they are true martingales, since this always follows easily from our assumptions on admissible solutions.
\end{remark}

\subsection{Assumptions and main results}
We next introduce our main assumptions and state our main \emph{dimension-free} estimates for the Nash system \eqref{nash}. Recall the definition \eqref{def.aij} of the matrix $\bA$ (associated to an admissible solution $\bu$) which will play a key role throughout the paper.

\subsubsection{The displacement monotone setting}
Our main assumption takes into account the semi-monotonicity of $\bm{L} = (L^i)_{i = 1,\dots,N}$ and $\bm{G} = (G^i)_{i = 1,\dots,N}$, as well as the length of the time horizon.

\begin{assumption}[Displacement semi-monotonicity and regularity] \label{assump.disp}
The functions $L^i$ (and $H^i$), $F^i$, and $G^i$ are each $C^2$, with bounded derivatives of order $2$ (but not necessarily of order $1$ in the direction $i$). Moreover, there are non-negative constants $C_{\bm{L},\ba}$, $C_{\bm{L}, \bx}$, and $C_{\bG,\dis}$ such that the following conditions hold:
\begin{itemize}[leftmargin=\parindent]
\item we have
\begin{equation} \label{strongconvLns}
    \sum_{i=1}^N D_{x^i} L^i\Bigr|^{(\bx,a^i)}_{(\ov{\bx},\ov{a}^i)} \cdot (x^i - \ov{x}^i) +   \sum_{i=1}^N D_a L^i\Bigr|^{(\bx,a^i)}_{(\ov{\bx},\ov{a}^i)} \cdot (a^i - \ov{a}^i) \geq C_{\bm L, \ba} |\ba - \ov{\ba}|^2 - C_{\bL,\bx}|\bx - \ov{\bx}|^2\,;
\end{equation}
for each $\bx, \ov{\bx}, \ba, \ov{\ba} \in (\R^d)^N$;
\item $\bG = (G^1,\dots,G^N)$ is \emph{displacement $C_{\bG, \dis}$-semi-monotone}, i.e.\ we have
\begin{equation}
    \label{FCsemimon}
    \sum_{i=1}^N D_{i} G^i\Bigr|^{\bx}_{\ov{\bx}} \cdot (x^i - \ov{x}^i) \geq - C_{\bG,\dis} |\bx - \ov{\bx}|^2\,,
\end{equation}
for each $\bx, \ov{\bx} \in (\R^d)^N$, and likewise for $\bG$ -- note that \eqref{FCsemimon} is equivalent to requiring, for each $\bx \in (\R^d)^N$,
\begin{equation*}%\label{dmonpoint}
    \bigl( D^2_{ji} G^i(\bx) \bigr)_{i,j=1,\dots,N} \geq - C_{\bG,\dis} I_{Nd}\,;
\end{equation*}
\item we have the \emph{displacement semi-monotonicity condition}
\begin{equation} \label{smallness'}
         C_{\bG,\dis} + \frac{T}2 C_{\bL,\bx} < \frac{C_{\bm L, \ba}}{T}\,.
\end{equation}
\end{itemize}
\end{assumption}

\begin{remark}
Note that \eqref{strongconvLns} is a joint strong convexity condition in $a$ and displacement $C_{\bL,\bx}$-semi-monotonicity in $\bx$ for $\bL$. The last condition \eqref{smallness'} implies that $C_{\bG,\dis} \lesssim T^{-1}$ and $C_{\bL,\bx} \lesssim T^{-2}$, so it explicitly quantifies the need for ``almost monotone'' data when the horizon is long, while no monotonicity is required for short-time results.
\end{remark}

\begin{convention}
Given this last condition, we will often work with the displacement monotonicity constant
\begin{equation} \label{cdispdef}
    C_{\dis} \coloneqq C_{\bm L, \ba} - \frac{T^2}{2} C_{\bL,\bx} - T C_{\bG,\dis} > 0\,.
\end{equation}
In addition, we denote by $C_{D \bL, \lip}$ the smallest constant such that we have
\begin{equation} \label{dL.lip}
  \sum_{i = 1}^N \big| D_{(x^i,a)} L^i(\bx,a^i) - D_{x^i} L^i(\ov \bx, \ov a^i) \big|^2 \leq C_{D \bL, \lip}^2 \bigl(|\bx - \ov \bx|^2 + | \ba - \ov \ba |^2 \bigr), 
\end{equation}
for all $\bx, \ba , \ov \bx, \ov \ba \in (\R^d)^N$,
and we denote by 
$C_{D\bG,\lip}$ the Lipschitz constant of the vector field $(D_{i} G^i)_{i = 1,\dots,N}$, ; i.e.\ we have
    \begin{equation} \label{CDLipdef}
        \sum_{i=1}^N\, \bigl|D_{i} G^i(\bx) - D_{i} G^i(\ov{\bx})\bigr|^2 \leq C_{D\bG,\lip}^2 |\bx - \ov{\bx}|^2
   \end{equation}
    for each $\bx, \ov{\bx} \in (\R^d)^N$. 
    Note that we could more precisely denote such constants by $C_{\mathrm{diag}(D\bL),\lip}$ and $C_{\mathrm{diag}(D\bG),\lip}$, but we opted for a simpler notation.
\end{convention}

\begin{remark} \label{lem.DH.lip} 
    The conditions we have put on $\bL = (L^1,\dots,L^N)$ can be translated to conditions on $\bH = (H^1,\dots,H^N)$. For example, the strict convexity of $L^i$ in $a$ implies that each $H^i$ is $C^2$, and there is an explicit constant $C_{D\bH, \lip} = C_{D\bH,\lip}( C_{\bL,\bx}, C_{\bL,\ba}, C_{D \bL, \lip})$
    such that 
    \begin{equation*}
        \sum_{i = 1}^N \big|D_{(x^i,p)} H^i\big(\bx,p^i) - D_{x^i} H^i\big(\ov \bx, \ov p^i\big)\big|^2 \leq C_{D \bH, \lip}^2 \bigl( |\bx - \ov \bx|^2 + | \bp - \ov \bp|^2 \bigr)
    \end{equation*}
    for each $\bx,\ov\bx,\bp,\ov\bp \in (\R^d)^N$.
     Indeed, letting $\hat a^i(\bx,\bp) \defeq -D_pH^i(\bx,p^i)$ and using the optimality condition $D_a L^i(\bx, \hat a^i(\bx,\bp))+p^i=0$, \eqref{strongconvLns} yields
    \begin{equation*}
    \begin{split}
       &C_{\bL,\ba} | \hat{\ba}(\bx,\bp) - \hat{\ba}(\ov \bx, \ov \bp) |^2 
       \\
       &\qquad \leq C_{\bL,\bx} |\bx - \ov \bx|^2  + \sum_{i = 1}^N D_{x^i} L^i \Big|^{(\bx, \hat a^i(\bx, \bp))}_{(\ov \bx, \hat a^i(\ov \bx, \ov \bp))} \cdot (x^i - \ov x^i) + \sum_{i = 1}^N (p^i - \ov p^i) \cdot (\hat a^i(\bx,\bp) - \ov a^i(\bx, \bp))
       \\
       &\qquad \leq \frac{C_{\bL,\ba}}{2} | \hat{\ba}(\bx,\bp) - \hat{\ba}(\ov \bx, \ov \bp)|^2 + \bigl(C_{\bL,\bx} + C_{D\bL,\lip} \bigr) |\bx - \ov \bx|^2 +  \frac{1 + C_{D\bL, \lip}^2}{ C_{\bL,\ba}} |\bp - \ov \bp|^2
    \end{split}
    \end{equation*}
    and so we find that $\sum_{i} \big| D_pH^i\big(\bx,p^i) - D_p H^i\big(\ov \bx, \ov p^i\big) \big|^2 \lesssim |\bx - \ov\bx|^2 + |\bp - \ov\bp^2|$, with implied constant depending only on $C_{\bL,\bx}$, $C_{\bL,\ba}$, and $C_{D \bL, \lip}$; then, using that $D_{x^i} H^i(\bx,p^i) = - D_{x^i} L^i(\bx, \hat{a}^i(\bx,\bp))$, from \eqref{dL.lip} and the previous estimate we get $\sum_{i} \big| D_{x^i}H^i\big(\bx,p^i) - D_{x^i} H^i\big(\ov \bx, \ov p^i\big) \big|^2 \lesssim |\bx - \ov\bx|^2 + |\bp - \ov\bp^2|$.
\end{remark}

%\begin{remark} 
%Note that as $L^i$ and $H^i$ are linked by \eqref{intro.defHi}, meaning that $H^i(x,\cdot)$ is the Legendre transform of $L^i(x,\cdot)$, then the more natural convexity condition \eqref{strongconvL} is equivalent to
%\begin{equation} \label{strongconvLforH}
%    - D_{x^i} H^i\Bigr|^{(x,p)}_{(\ov{x},\ov{p})} \cdot (x - \ov{x}) + D_p H^i\Bigr|^{(x,p)}_{(\ov{x},\ov{p})} \cdot (p - \ov{p}) 
%    \geq C_{\bm L} \Bigl|D_pH^i\Bigr|^{(x,p)}_{(x,\ov{p})} \Bigr|^2\,.
%\end{equation}
%\end{remark}

\begin{convention} \label{conv.dimfree}
    When \Cref{assump.disp} is in force, we will say that a constant is \emph{dimension-free} if it depends only on the quantities 
    \begin{equation} \label{dimfree}
        C_{\bm L,\ba},\ \ C_{\bm L,\bx}, \ \ C_{\bG,\dis}, \ \ C_{D\bL, \lip}, \ \ C_{D\bG,\lip}, \ \ d, \ \text{ and } \ T,
    \end{equation} 
    but \underline{not} on $N$ or on $\sigma$ (for the notation $\|\cdot\|_\infty$ we always refer to Section \ref{subs:not}). We use this terminology because in most asymptotic regimes of interest (in which the number $N$ of players goes to infinity), the quantities in \eqref{dimfree} can be bounded independently of $N$. {Note that we include the independence of $\sigma$ in our definition of dimension-free in view of our application of the estimates to the joint vanishing viscosity and large $N$ limit.}

    The dependence on $N$ of our estimates will in turn be seen through the key quantities
\begin{gather*} 
%\label{def.delta}
    \delta^i \coloneqq \Big\|\sum_{j \neq i} |D_{j} G^i|^2 \Big\|_{\infty} + \Big\| \sum_{j \neq i} |D_{j} L^i|^2 \Big\|_{\infty} + \Big\| \sum_{j \neq i} |D_{x^j a} L^j|^2 \Big\|_{\infty}\,, \qquad 
     \delta = \max_{i=1,\dots,N} \delta^i,
\\ %\label{def.kappa}
       \kappa^i \coloneqq \Big\| \sum_{j \neq i} |D_{ij} G^j|^2 \Big\|_{\infty} + \Big\| \sum_{j \neq i} |D_{ij} L^j|^2 \Big\|_{\infty}, \qquad  \kappa \coloneqq \max_{i=1,\dots,N}  \kappa^i\,.
\end{gather*}
\end{convention}

Here is our main estimate between open-loop and closed-loop equilibria.

\begin{theorem} \label{thm.olcldisp}
    Let \Cref{assump.disp} hold. There is a dimension-free constant $C$ such that if 
    \begin{equation} \label{smallness}
       \delta \sum_{i=1}^N \delta^i \leq \frac{\sigma^2}{C},
   \end{equation}
   then, for any $t_0 \in [0,T)$ and any square-integrable,  $(\R^d)^N$-valued, $\scrF_{t_0}$-measurable random vector $\bzeta_0$, we have
    \[
     \E\Bigl[\, \sup_{t \in [t_0,T]} \bigl|\bX_t^{\OL,t_0,\bzeta_0} - {\bX_t}^{\CL,t_0,\bzeta_0}\bigr|^2 \Bigr] \leq C\sigma^{-1} \delta \sum_{i=1}^N \delta^i\,.
    \]
\end{theorem}

We note that, along the way to the proof of the previous theorem, we quantify the distance between solutions $u^i$ to the Nash system and solutions $v^i$ to the system of PDE associated to open-loop equilibria; this is stated in \Cref{lem.clol}.

A second main estimate is between open-loop and distributed equilibria.

\begin{theorem} \label{thm.oldistdisp}
    Let \Cref{assump.disp} hold, and assume also that $\sigma_0 = 0$. Then, for any $t_0 \in [0,T)$ and any square-integrable, $(\R^d)^N$-valued, $\scrF_{t_0}$-measurable random vector $\bzeta_0$ with independent components, we have
    \[
       \E\Bigl[\, \sup_{t \in [t_0,T]} \bigl|\bX_t^{\OL,t_0,\bzeta_0} - {\bX_t}^{\Dis,t_0,\bzeta_0}\bigr|^2 \Bigr] \lesssim \bigl(\sigma + \max_{i=1,\dots,N} C_{P,0}^i \bigr) \sum_{i=1}^N \bigl(\delta^i + \kappa^i\bigr)\,,
    \]
    with dimension-free implied constant, where $C_{P,0}^i$ indicates the Poincar\'e constant of the measure $\cL(\zeta^i_0)$.
\end{theorem}

\begin{remark} \label{rmk.poincaredirac} 
    We say that a measure $\mu \in \cP_2(\R^d)$ satisfies a Poincar\'e inequality with constant $C$ if 
    \begin{equation} \label{poincdef}
      \text{Var}_{\mu}(g) \coloneqq  \int_{\R^d} g^2\, \d \mu - \Big( \int_{\R^d} g\, \d\mu \Big)^2 \leq C \int_{\R^d} |Dg|^2\, \d\mu 
    \end{equation}
    for all $C^1$ functions $g$ with at most quadratic quadratic growth, and the Poincar\'e constant of $\mu$ means the smallest constant $C$ which verifies \eqref{poincdef}.
    Dirac masses have zero Poincar\'e constant, so if $\bzeta_0 = \bx_0$ is deterministic, then $C_{P,0}^i = 0$; thus \Cref{thm.oldistdisp} shows that under \Cref{assump.disp}, there is a dimension-free constant $C$ such that 
    \begin{equation*}
         \sup_{t_0 \in [0,T],\, \bx_0 \in (\R^d)^N} \E\Bigl[\, \sup_{t \in [t_0,T]} \bigl|\bX_t^{\OL,t_0,\bx_0} - {\bX_t}^{\Dis,t_0,\bx_0}\bigr|^2 \Bigr] \leq C\sigma \sum_{i=1}^N \bigl(\delta^i + \kappa^i\bigr).
    \end{equation*}
\end{remark}

\subsubsection{The Lasry--Lions semi-monotone setting}
In order to formulate a Lasry--Lions semi-mo\-no\-to\-ni\-ci\-ty condition, we will need to work with separable Lagrangians, meaning that
\begin{equation} \label{Lagrangian.sep}
    L^i(\bx, a) = L^i_0(x^i,a) + F^i(\bx),
\end{equation}
for some $L^i_0 \colon \R^d \times \R^d \to \R$, $F^i \colon (\R^d)^N \to \R$; this means that the Hamiltonians also have a separable form: 
\begin{equation} \label{Hamiltonian.sep}
    H^i(\bx,a) = H_0^i(x^i,p) - F^i(\bx), \qquad H_0^i(x,p) = \sup_{a \in \R^d} \bigl\{ - a \cdot p - L^i_0(x^i,a) \bigr\}.
\end{equation}
In the Lasry--Lions semi-monotone case, the following takes the place of \Cref{assump.disp}.

\begin{assumption}[Lasry-Lions semi-monotonicity and regularity] \label{assump.LL}
The Lagrangians $L^i$ satisfy \eqref{Lagrangian.sep}. The functions $L_0^i$, $H_0^i$, $F^i, G^i : (\R^d)^N \to \R$ are each $C^2$, with bounded derivatives of order 2 (but not necessarily of order 1 in the direction $i$), and $a \mapsto L^i_0(x^i,a)$ is convex for each $i$. Moreover, there are non-negative constants $C_{\bF, \lm}$, $C_{\bG,\lm}$, $C_{\bF, \lip}$, $C_{\bG, \lip}$, $C_{\bH}$, and $\lambda_{\bH} > 0$ such that the following conditions hold:
\begin{itemize}[leftmargin=\parindent]
\item the Hamiltonians satisfy
\[
    |D_x H^i_0(x,p)| \leq C_{\bH} ( 1 + |p| )\,, \qquad
    D^2_{pp} H^i_0(x,p) \ge \lambda_{\bH} I_d\,, 
\]
for each $(x,p) \in \R^d \times \R^d$, and each $i = 1,\dots,N$;
\item $\bF = (F^1,\dots,F^N)$ is \emph{Lasry--Lions $C_{\bF,\lm}$-semi-monotone}, and $\bG = (G^1,\dots,G^N)$ is \emph{Lasry--Lions $C_{\bG,\lm}$-semi-monotone}, i.e.\ we have
\begin{equation}\label{llmonpoint}
    \bigl( D^2_{ji} F^i(\bx) \1_{i \neq j} \bigr)_{i,j=1,\dots,N} \geq - C_{\bF,\lm} I_{Nd}\,,
\end{equation}
for each $\bx \in (\R^d)^N$, and likewise for $\bG$;
\item we have
\begin{equation*}
    \max_{i=1,\dots,N} \| D_{i} F^i\|_{\infty} \leq C_{\bF,\lip}\,, \quad \max_{i=1,\dots,N} \| D_{i} G^i\|_{\infty} \leq C_{\bG,\lip}\,.
\end{equation*}
\end{itemize}
\end{assumption}

\begin{remark} 
In \Cref{assump.LL} we are missing a necessary counterpart of \eqref{smallness'} in the Lasry--Lions semi-monotone setting, i.e.\ a relation between $C_{\bF,\lip}, C_{\bG,\lip}$ and $T$. That is introduced below (see~\eqref{smallness2}) as an hypothesis of \Cref{thm.olclLL} (which is the counterpart of \Cref{thm.olcldisp}), and not directly among the assumptions above since, unlike \eqref{smallness'}, it has a cumbersome explicit expression in terms of constants related to the data, so we preferred to state it in terms of some more implicit dimension-free constants.
\end{remark}

\begin{convention}
    Similarly as in the dispacement semi-monotone case, we denote by $C_{D\bF,\lip}$ and $C_{D\bG,\lip}$ the Lipschitz constants of the vector fields $(D_{i} F^i)_{i = 1,\dots,N}$ and $(D_{i} G^i)_{i = 1,\dots,N}$, i.e. we have \eqref{CDLipdef} and likewise for $C_{D\bF,\lip}$.
\end{convention}

\begin{convention} \label{conv.dimfreell}
    When \Cref{assump.LL} is force, we will say that a constant $C$ is \emph{dimension-free} if it depends only on the quantities 
    \begin{equation*} %\label{dimfreell}
       C_{\bH}, \ \ \lambda_{\bH}, \ \ C_{\bF,\lm}, \ \ C_{\bG, \lm}, \ \ C_{\bF, \lip}, \ \ C_{\bG,\lip}, \ \ C_{D\bF, \lip}, \ \ C_{D\bG,\lip}, \ \ \|D^2 \bH_0\|_{\infty}, \ \ d, \ \text{ and } \ T,
    \end{equation*} 
    but \underline{not} on $N$ or $\sigma$. The dependence on $N$ of our estimates will be seen through the key quantities
\begin{gather*} 
%\label{def.deltall}
    \delta^i \coloneqq \Big\|\sum_{j \neq i} |D_{j} G^i|^2 \Big\|_{\infty} + \Big\| \sum_{j \neq i} |D_{j} F^i|^2 \Big\|_{\infty} , \qquad 
     \delta = \max_{i=1,\dots,N} \delta^i,
\\ %\label{def.kappall}
       \kappa^i \coloneqq \Big\| \sum_{j \neq i} |D_{ij} G^j|^2 \Big\|_{\infty} + \Big\| \sum_{j \neq i} |D_{ij} F^j|^2 \Big\|_{\infty}, \qquad  \kappa \coloneqq \max_{i=1,\dots,N}  \kappa^i, 
       \\
       %\label{def.kappatilde.ll}
        \tilde{\kappa}^i \coloneqq \Big\| \sum_{j \neq i} |D_{ij} G^i|^2 \Big\|_{\infty} + \Big\| \sum_{j \neq i} |D_{ij} F^j|^2 \Big\|_{\infty}, \qquad  \tilde \kappa \coloneqq \max_{i=1,\dots,N}  \wt \kappa^i. 
\end{gather*}
\end{convention}

Here are the estimates on the gaps between different formulations in the Lasry--Lions semi-mo\-no\-to\-ne case.

\begin{theorem} \label{thm.olclLL}
    Let \Cref{assump.LL} hold. There is a dimension-free constant $C$ such that the following holds: if
    \begin{gather} \label{smallness2}
        C_{\bG,\lm} + T C_{\bF,\lm} \leq \frac{1}{T}\,e^{-C(1+\sigma^{-1})}\,,
        \\ \label{smallness1'}
        \sqrt{\kappa\tilde\kappa} \leq e^{-e^{C(1+\sigma^{-1})}}\,,
    \end{gather}
    and
    \begin{equation}\label{smallness1}
        \delta \sum_{i=1}^N \delta^i \leq e^{-e^{C(1+\sigma^{-1})}}\,,
        \end{equation}
        then, for any $t_0 \in [0,T)$ and any square-integrable,  $(\R^d)^N$-valued, $\scrF_{t_0}$-measurable random vector $\bzeta_0$, we have 
    \[
    \E\Bigl[\, \sup_{t \in [t_0,T]} \bigl|\bX_t^{\OL,t_0,\bzeta_0} - {\bX}^{\CL,t_0,\bzeta_0}_t\bigr|^2 \Bigr] \leq e^{e^{C(1+\sigma^{-1})}} \delta \sum_{i=1}^N \delta^i.
    \]
\end{theorem}

\begin{theorem} \label{thm.oldistLL}
    Let \Cref{assump.LL} hold, and assume also that $\sigma_0 = 0$. There is a dimension-free constant $C$ such that, if \eqref{smallness2} and \eqref{smallness1'} hold, then, for any $t_0 \in [0,T)$ and any square-integrable, $(\R^d)^N$-valued, $\scrF_{t_0}$-measurable random vector $\bzeta_0$ with independent components, we have
   \begin{equation*}
       \E\Bigl[\, \sup_{t \in [t_0,T]} \bigl|\bX_t^{\OL,t_0,\bzeta_0} - {\bX_t}^{\Dis,t_0,\bzeta_0}\bigr|^2 \Bigr] \leq \bigl(\sigma + \max_{i=1,\dots,N} C_P^i \bigr) e^{e^{C(1+\sigma^{-1})}}\! \sum_{i=1}^N \kappa^i\,,
    \end{equation*}
    where $C_P^i$ denotes the Poincar\'e constant of the measure $\cL(\zeta^i)$.
\end{theorem}

\section{Main estimates in the displacement monotone case} \label{sec.disp}

\subsection{A Lipschitz estimate on the Pontryagin system} \label{subsec.pontryagin.disp}

The first result we prove is a dimension-free bound on $|D\bv|_\op$. We recall that a solution $\bv$ to \eqref{pontpde} is assumed to exist and be admissible in the sense described in Section~\ref{sec:admiss}.

\begin{proposition} \label{prop.olboundsdisp}
    Let Assumption \ref{assump.disp} hold. 
	Then there is a dimension-free constant $C$ such that
    \begin{equation*}
        |\bm v(t,\bx) - \bm v(t,\bar{\bx})| \leq C |\bx - \ov{\bx}| \qquad \forall \ t \in [0,T],\ \bx,\bar{\bx} \in (\R^d)^N\,,
    \end{equation*}
    or, equivalently,
    \begin{equation} \label{eq.dvopbound.dis}
    \norm{|D\bv|_\op}_{L^\infty([0,T]\times(\R^d)^N)} \leq C.
    \end{equation}
\end{proposition}

%\marginpar{\tiny DF: I agree with the red observation. Do we want to make it explicit in any way?}{Seems we only use displacement monotonicity constants and the Lipschitz constant of the vector field $(D_{x^i} H^i)_{i = 1,\dots,N}$. {I think it can be removed, it was just a note to self which never got removed}}

\begin{proof}
    We start by fixing $t_0 \in [0,T)$, $\bx_0, \ov{\bx}_0 \in (\R^d)^N$. Let $(\bX, \bY, \bZ)$ denote the solution of \eqref{pontryagin} with initial condition $\bX_{t_0} = \bx_0$, and $(\ov{\bX}, \ov{\bY}, \ov{\bZ})$ denote the solution with $\ov{\bX}_{t_0} = \ov{\bx}_0$. Also let $\alpha_t^i \defeq - D_p H^i(\bX_t, Y_t^i)$, and likewise for $\ov{\alpha}_t^i$. Now set $\Delta \bX \defeq \bX - \ov{\bX}$, and likewise for $\Delta \bY$, $\Delta \bm\alpha_t$, and so forth; keep in mind that $|\Delta\bY_{t_0}| = |\bm v(t_0,\bx_0) - \bm v(t_0,\bar{\bx}_0)|$.
    
    \smallskip
    
    \noindent$\bullet$ \emph{Step 1: evolution of $|\Delta\bY|^2$.}
    First, use the dynamics for $\bY$ and $\ov{\bY}$ to find that 
    \begin{align*}
        \d |\Delta \bY_t |^2 = 2 \sum_{i} \Delta Y_t^i \cdot D_{x^i} H^i\Bigr|^{(\bX_t, Y_t^i)}_{(\ov{\bX}_t, \ov{Y}^i_t)}  \d t + \d S_t\,, 
    \end{align*}
    where $S$ is a sub-martingale whose form does not matter (see Remark \ref{rmk.admissible}). Integrate from $t$ to $T$, take expectations, and apply Young's inequality to get 
    \begin{equation*}% \label{vlip.step1.eq1}
    \begin{split}
       & \E\big[ |\Delta \bY_t|^2 \big]
        \\
        & \leq \E\bigg[\, |\Delta \bY_T |^2 + \int_{t}^T \sum_{i} \,\Bigl| D_{x^i} H^i\Bigr|^{(\bX_t, Y_t^i)}_{(\ov{\bX}_t, \ov{Y}^i_t)} \Bigr|^2 \,\d t+ \int_t^T |\Delta \bY|^2 \,\d t \bigg]
	\\
	& \leq \E\bigg[\, C_{D\bG,\lip}^2\,|\Delta \bX_T |^2 + C_{D\bH, \text{Lip}} \int_t^T |\Delta\bX|^2 \,\d t + \big( 1 + C_{D\bH, \text{Lip}} \big) \int_t^T |\Delta \bY|^2 \,\d t \bigg]\,,
	\end{split}
	\end{equation*}
    with $C_{D\bH, \lip}$ as in \Cref{lem.DH.lip}.
	Then Gr\"{o}nwall's inequality gives
	\begin{equation} \label{vlip.proof.EDY}
	 |\Delta \bY_{t_0}|^2 \leq \E\bigg[\, C_{D\bG,\lip}^2\,|\Delta \bX_T |^2 + C_{D\bH, \text{Lip}} \int_{t_0}^T |\Delta\bX|^2 \,\d t \bigg] e^{T(1+C_{D\bH, \text{Lip}})} \,.
	\end{equation}
    
    \smallskip
    
     \noindent$\bullet$ \emph{Step 2: evolution of $\Delta \bX \cdot \Delta \bY$.}
    Note that, for any $t \in [t_0,T]$,
    \begin{equation} \label{prexalphabound}
        |\Delta \bX_t - \Delta\bx_0|^2 = \sum_{i}\, \biggl|\, \int_{t_0}^t \Delta \alpha_t^i \,\d t \,\biggr|^2 \leq (t-t_0) \int_{t_0}^T |\Delta \bm\alpha_t|^2 \d t\,,
    \end{equation}
    hence, by Young's inequality, for any $\delta > 0$ there is a constant $C_\delta$ such that 
    \begin{equation} \label{xalphabound}
    |\Delta \bX_t|^2 \leq (1+\delta)(t-t_0) \int_{t_0}^T |\Delta \bm\alpha_t|^2 \d t + C_\delta |\bx_0|^2\,.
    \end{equation}
  On the other hand, by It\^o's formula and \eqref{strongconvLns}, we have
    \begin{equation*}
    \begin{split}
        \d\bigl( \Delta X_t^i \cdot \Delta Y_t^i \bigr) &= \Big( \Delta Y_t^i \cdot \Delta \alpha^i_t
        + \Delta X_t^i \cdot D_{x^i} H^i\Bigr|^{(\bX_t, Y_t^i)}_{(\ov{\bX}_t, \ov{Y}_t^i)} \Big)\,\d t + \d M_t
        \\
        &= -\Big( \Delta \alpha_t^i \cdot D_a L^i \Big|^{(\bX_t, \alpha_t^i)}_{(\ov \bX_t, \ov{\alpha}_t^i)} + \Delta X_t^i \cdot D_{x^i} L^i \Big|^{(\bX_t, \alpha_t^i)}_{(\ov \bX_t, \ov{\alpha}_t^i)} \Big)\, \d t + \d M_t \,,
        \\
        &\leq - \Big( C_{\bL,\ba} |\Delta\balpha_t|^2 - C_{\bL,\bx} |\Delta\bX_t|^2 \Big)\, \d t + \d M_t, 
    \end{split}
    \end{equation*}
with $M$ being a martingale whose form is unimportant. Integrating from $t_0$ to $T$, taking expectations to dispense with the martingale terms, and summing over $i=1,\dots,N$, we get 
   \[
        C_{\bm L,\ba} \E\bigg[ \int_{t_0}^T |\Delta \bm\alpha_t|^2 \,\d t \bigg] \leq \sum_{i} \Delta x_0^i \cdot \Delta Y_{t_0}^i + \E\bigg[ C_{\bL,\bx} \int_{t_0}^T |\Delta \bX_t|^2 dt + C_{\bG, \dis} |\Delta \bX_T|^2 \bigg]\,;
    \]
    therefore, combining this with the bound~\eqref{xalphabound}, we deduce that 
   \begin{equation} \label{vlip.proof.Da}
   \begin{multlined}[b][.9\displaywidth]
        \biggl( C_{\bm L,\ba} - (1+\delta) T\Bigl(C_{\bG,\dis} + \frac{T}2 C_{\bL,\bx}\Bigr)\biggr)\, \E\bigg[ \int_{t_0}^T |\Delta \bm\alpha_t|^2\, \d t \bigg] 
        \\[3pt]
        \leq \sum_{i} \Delta x_0^i \cdot \Delta Y_{t_0}^i + C_\delta \bigl( C_{\bG,\dis} + TC_{\bL,\bx} \bigr) |\Delta\bx_0|^2\,.
        \end{multlined}
    \end{equation}
    As $C_\dis > 0$ (recall \eqref{cdispdef}), there is $\delta_0 > 0$ (which depends only on $C_\dis$, $C_{\bG,\dis}$, $C_{\bL,\bx}$, and $T$) such that 
    \[
    C_{\bm L} - (1+\delta_0) T\Bigl(C_{\bG,\dis} + \frac{T}2 C_{\bL,\bx}\Bigr) \geq \frac{C_\dis}2 > 0\,,
    \]
   so, using now \eqref{vlip.proof.Da} with $\delta = \delta_0$, as well as the Cauchy-Schwarz inequality, in \eqref{prexalphabound}, we find that
   \begin{equation} \label{vlip.dis.proof.estX1}
   \sup_{t \in [t_0,T]} \E\bigl[|\Delta\bX_t-\Delta\bx_0|^2\bigr] \lesssim |\Delta\bx_0|^2 + |\Delta\bx_0||\Delta\bY_{t_0}|, 
   \end{equation}
   with dimension-free implied constant.
    
    \noindent$\bullet$ \emph{Step 3: closing the estimates.}
	Combining the bounds \eqref{vlip.proof.EDY} and \eqref{vlip.dis.proof.estX1} we find that 
   \[
   |\Delta\bY_{t_0}|^2 \lesssim |\Delta\bx_0|^2 + |\Delta\bx_0|\,|\Delta\bY_{t_0}| ,
   \]
   with dimension-free implied constant.
  By Young's inequality we find that there is a dimension-free constant $C$ such that $|\Delta\bY_{t_0}|^2 \leq C|\Delta\bx_0|^2$; recalling the definition of $\Delta\bY_{t_0}$ as well as the arbitrariness of $t_0$ and $\bx_0$, this provides the bound $\norm{|D\bm v|_\op}_\infty^2 \leq C$.
  %Now, the procedure to get to \eqref{eq.dvopbound.dis} (which is sharper as $T \to 0$) is easy, so we only sketch it briefly: compute $\d |\Delta \bX_t|^2$ with It\^{o}'s formula, integrate, take expectations, and estimate with Young's inequality to have
  % \[
  %\sup_{t \in [t_0,T]} \E \bigl[| \Delta \bX_t|^2 \bigr] \leq |\bx_0|^2 + T\bigl(\norm{D_{px}\bH}_\infty + \norm{D_{pp}\bH}_\infty^2 \bigr) \sup_{t \in [t_0,T]} \E \bigl[ |\Delta\bX_t|^2 \bigr] + T\tilde C''\,;
  % \]
  % then use \eqref{vlip.dis.proof.estX1}, recalling the bound $|\Delta \bY_{t_0}|^2 \leq \tilde C'' |\Delta \bx_0|^2$, in the right-hand side above, and finally plug the resulting estimate for $\sup_{t \in [t_0,T]} \E \bigl[| \Delta \bX_t|^2 \bigr]$ back into \eqref{vlip.step1.eq1}.
\end{proof}

\subsection{From the Pontryagin systems to the Nash system} \label{subsec.nash.disp}

The goal of this subsection is to prove the following. We recall that a solution $\bu$ to \eqref{nash} is assumed to exist and be admissible in the sense described in Section~\ref{sec:admiss}, and that $\bA$ is defined as in \eqref{def.aij}.

\begin{theorem} \label{thm.maindisp}
    Let \Cref{assump.disp} hold. There is a dimension-free constant $C$ such that, if \eqref{smallness} holds, then
    \begin{equation} \label{mainopbounddisp}
        \sup_{t_0 \in [0, T],\, \bx_0 \in (\R^d)^N} \E\bigg[ \int_{t_0}^T \big| \bA(t,\bX_t^{\CL,t_0,\bx_0}) \big|_{\op}^2 \,\d t \bigg] \leq C
    \end{equation}
    and
    \begin{equation} \label{mainl2bounddisp}
      \Bigl\|\,  \sum_{j \neq i} |D_{j} u^i|^2 \Bigr\|_{\infty}  \leq C\delta^i  \qquad \forall\,i = 1,\dots,N.
    \end{equation}
\end{theorem}

 The key bound \eqref{mainopbounddisp} is proved by using the propagation argument used in \cite{CR24b}, so in the following lemmas we will be \textit{assuming} that for some $T_0 \in [0,T)$, we have a bound of the form 
\begin{equation} \label{aprioriopbound}
        \sup_{ t_0 \in [T_0, T],\, \bx_0 \in (\R^d)^N} \E\bigg[ \int_{t_0}^T \big| \bA(t,\bX_t^{t_0,\bx_0}) \big|_{\op}^2 \,\d t \bigg] \leq M\,,
\end{equation}
for some $M > 0$, where $\bA$ is the matrix defined in \eqref{def.aij} and $\bX^{t_0,\bx_0} = \bX^{\CL,t_0,\bx_0}$ solves \eqref{def.xt0x0}. The novelty of the strategy, in turn, consists in using the control we have on $\bv$ from \Cref{prop.olboundsdisp} and the expected proximity of $D\bv$ and $D(D_iu^i)_i$ for large $N$ to simplify the needed steps.

It will be useful to write $u^{i,j} \defeq D_{j} u^i$ (which takes values in $\R^d$) and record here that by differentiating \eqref{nash} we find the equations
    \begin{equation} \label{nashfirstder}
    \begin{split}
     - \partial_t u^{i,j} - L^{\bm{u}} u^{i,j}
    + \sum_{k \neq i} (A^{k,j})^\TT D_{pp} H^k(\bx,u^{k,k})  u^{i,k} &
        \\
        +  \sum_{k \neq i} D_{x^j p} H^k(\bx, u^{k,k} ) u^{i,k} + D_{x^j} H^i &= 0\,,
    \end{split}
\end{equation}
with terminal condition $u^{i,j}(T,\cdot) = D_{j} G^i$, where $L^{\bu}$ indicates the differential operator 
\begin{equation*}
  L^{\bu} \defeq \sigma \sum_{k=1}^N \Delta_{k} - \sigma_0 \sum_{k,l=1}^N \tr  D_{kl} + \sum_{k=1}^N  D_p H^k(\bx, u^{k,k}) \cdot D_k\,, 
\end{equation*}
which is applied component-wise to $u^{i,j}$.

%We start with an estimate which bounds the derivatives $D_j u^i$ ($j \neq i$) in a pointwise manner.

\begin{lemma} \label{lem.gradbound}
    Let Assumption~\ref{assump.disp} hold. Suppose also that \eqref{aprioriopbound} holds for some $T_0 \in [0,T)$ and $M > 0$. Then, there is a dimension-free constant $C_1$ such that, for each $i = 1,\dots,N$, 
\[
 \sum_{j \neq i} \,\bigl|D_j u^i(t_0,\bx_0)\bigr|^2 + \sigma \, \E\bigg[\int_{t_0}^T \sum_{j \neq i} \sum_{k} \big|D_{jk} u^i(t,\bX_t^{t_0,\bx_0}) \big|^2 \,\d t \bigg] 
    \leq \bigl(1+ C_1 \sqrt{M} e^{C_1 {M}}\bigr) \delta^i\,.
\]
for all $t_0 \in [T_0, T]$, $\bx_0 \in (\R^d)^N$.
\end{lemma}

\begin{proof}
Fix $(t_0,\bx_0) \in [T_0, T] \times (\R^d)^N$ and define, for $i,j,k = 1,\dots,N$,
\begin{equation} \label{processesdef2}
\begin{gathered}
    \bX \defeq \bX^{t_0,\bx_0}\,, \qquad Y_t^{i,j} \defeq u^{i,j}(t,\bX_t)\,,
    \\
     Z_t^{i,j,k} \defeq \sqrt{2 \sigma} D_{k} u^{i,j}(t,\bX_t)\,, \qquad Z_t^{i,j,0} \defeq \sqrt{2 \sigma_0} \sum_{k = 1}^N D_k u^{i,j}(t,\bX_t)\,.
\end{gathered}
\end{equation}
In order to ease the notation, we will write $D_{pp} H^i_t \defeq D_{pp} H^i(\bX_t,Y^{i,i}_t)$ and likewise for $D_{x^jp} H^i_t$, $A^{i,j}_t$ and $D_j F^i_t$.
Using \eqref{nashfirstder} and It\^o's formula, we find that, for $i \neq j$, 
\[
    \d Y_t^{i,j} = \Big(\,\sum_{k \neq i} (A^{k,j}_t)^\TT D_{pp}H^k_t Y^{i,k}_t  + \sum_{k \neq i} D_{x^j p} H^k_t Y_t^{i,k} + D_{x^j} H^i_t \, \Big)\d t  + \sum_{k = 0}^N Z_t^{i,j,k} \,\d W_t^k\,.
\]
As a consequence, 
\[
\begin{split}
    \d \, \sum_{j \neq i} |Y_t^{i,j}|^2 &= \bigg( 2 \sum_{\substack{k,j \\ k \neq i \neq j}} (Y_t^{i,j})^\TT (A^{k,j})^\TT D_{pp}H^k_t  Y_t^{i,k} + 2  \sum_{\substack{k,j \\ k \neq i \neq j}} (Y_t^{i,j})^T D_{x^j p} H^k_t Y_t^{i,k} + \sum_{j \neq i} Y_t^{i,j} \cdot D_{x^j} H_t^i 
    \\
    &\qquad + \sum_{j \neq i} \sum_{k = 0}^N |Z_t^{i,j,k}|^2 \bigg) \d t + \d \tilde M_t\,, 
\end{split}
\]
with $\tilde M$ being a martingale. Integrate from $t_0$ to $T$ and take expectations, to find that 
\begin{equation} \label{smalltimebound}
\begin{split}
    &\E\bigg[ \,\sum_{j \neq i} |Y_{t_0}^{i,j}|^2 + \int_{t_0}^T \sum_{j \neq i} \sum_{k = 0}^N |Z_s^{i,j,k}|^2 \,\d s - \sum_{j \neq i} |Y^{i,j}_T|^2\bigg] 
     \\
    &\qquad\qquad\leq \E\bigg[\int_{t_0}^T \Bigl( \bigl(2 C_{D \bH, \lip} |\bA_s|_{\op} + 2 C_{D \bH, \lip} + 1 \bigr) \sum_{j \neq i} |Y_s^{i,j}|^2 + \sum_{j \neq i} |D_{x^j} H^i_s|^2 \Bigr)\d s \bigg]
     \\
      &\qquad\qquad\leq T\delta^i + C' (1+\sqrt{M}\,)\sqrt{T-t_0} \sup_{s \in [t_0,T]} \Big\| \sum_{j \neq i} |u^{i,j}(s,\cdot)|^2 \Big\|_{\infty} \,,
\end{split}
\end{equation}
where $C'$ is a dimension-free constant. For $\tau$ small enough with respect to $C'$ and $M$, recalling the definitions~\eqref{processesdef2} and taking the supremum over $t_0 \in [T-\tau,T]$ in \eqref{smalltimebound}, we deduce that
\[
\sup_{s \in [T-\tau,T]} \biggl\lVert\, \sum_{j \neq i} |u^{i,j}(s,\cdot)|^2 \biggr\rVert_\infty \leq 2\,\biggl\|\, \sum_{j \neq i} |u^{i,j}(T,\cdot)|^2\biggr\|_{\infty} \!+ 2T\delta^i \leq 2(1+ T)\delta^i\,,
\]
so we can iterate this argument at most $\lceil 4TC'^2(1+\sqrt M\,)^2 \rceil$ times, on intervals of the form $[T_0 \vee (T-k\tau), T-(k-1)\tau]$, to obtain
\[
\sup_{s \in [T_0,T]} \biggl\lVert\, \sum_{j \neq i} |u^{i,j}(s,\cdot)|^2 \biggr\rVert_\infty \leq e^{C''(1+ M)} \delta^i\,,
\]
with $C''$ a dimension-free constant.
Now plug this back in to \eqref{smalltimebound} to complete the proof. 
\end{proof}

%The next lemma says that if the operator norm of $\bA$ is controlled, then under some conditions the open-loop and closed-loop systems are similar. 

\begin{lemma} \label{lem.clol}
   Let Assumption~\ref{assump.disp} hold. Suppose also that \eqref{aprioriopbound} holds for some $T_0 \in [0,T)$ and $M > 0$. Then, there is a dimension-free constant $C_2$ such that
\begin{equation*} \begin{multlined}[b][.85\displaywidth]
 \sum_{i} \,\bigl|(D_i u^i-v^i)(t_0,\bx_0)\bigr|^2 + \sigma \, \E\bigg[\int_{t_0}^T \sum_{i,j} \big|(D_{ji} u^i - D_jv^i)(t,\bX_t^{t_0,\bx_0}) \big|^2 \,\d t \bigg] 
   \\
   \leq C_2\big(1+ M e^{C_2{M}} \big) \sigma^{-1} \delta \sum_j \delta^j
\end{multlined} \end{equation*}
for all $t_0 \in [T_0, T]$, $\bx_0 \in (\R^d)^N$.
\end{lemma}

\begin{proof}
    Fix $(t_0,\bx_0) \in [T_0,T] \times (\R^d)^N$, and define $\bX$, $Y^{i,j}$, $Z^{i,j,k}$ as in the previous proof (see~\eqref{processesdef2}). Similarly, we set 
    \[
        \ov{Y}^i_t = v^i(t,\bX_t)\,, \qquad  \ov{Z}^{i,j}_t = \sqrt{2 \sigma} v^{i,j}(t,\bX_t)\,, \qquad \ov{Z}^{i,0}_t = \sqrt{2\sigma_0} \sum_{j = 1}^N v^{i,j}(t,\bX_t)\,.
    \]
    Then, let $\Delta Y^i = Y^{i,i} - \ov{Y}^i$, and likewise for $\Delta Z^{i,j}$.
    
    By using It\^{o}'s formula together with the equation for $u^{i,i}$ in \eqref{nashfirstder} and the equation~\eqref{pontpde} for $v^i$, we find 
    \begin{equation*}
        \d Y_t^{i,i} = \bigl( D_{x^i} H^i(\bX_t, Y^{i,i}_t)  + E_t^i \,\bigr) \,\d t + \sum_{j = 0}^N Z_t^{i,i,j} \d W_t^j 
    \end{equation*}
    and 
   \[
        \d \ov{Y}_t^i = \Big( D_{x^i} H^i(\bX_t, \ov{Y}_t^i) - \sum_{j = 1}^N v^{i,j}(t,\bX_t) D_p H^j\Bigr|^{(\bX_t, Y_t^{j,j})}_{(\bX_t, \ov{Y}_t^{j})}  \,\Big) \d t 
+ \sum_{j = 0}^N \ov{Z}_t^{i,j} \d W_t^j,
        \]
    where
    \[
   E^i_t \defeq \sum_{j \neq i} \Big(D_{ij} u^j(t,\bX_t)^\TT D_{pp}H^j(\bX_t,Y^{j,j}_t) + D_{x^i p} H^j(\bX_t, Y_t^{j,j}) \Big)  D_{j} u^i(t,\bX_t)\,.
   \]
 As a consequence, we can compute 
    \begin{equation} \label{dysquared}
    \begin{split}
        \d |\Delta \bY_t|^2 &= \Big( 2\sum_{i} (\Delta Y_t^i)^\TT D_{x^i} H^i\Bigr|^{(\bX_t, Y_t^{i,i})}_{(\bX_t, \ov{Y}_t^{i})} + 2 \sum_{i} (\Delta Y_t^i)^\TT E_t^i \nonumber \\
        &\qquad + 2 \sum_{i,j} (\Delta Y_t^i)^\TT v^{i,j}(t,\bX_t) D_p H^j\Bigr|^{(\bX_t, Y_t^{j,j})}_{(\bX_t, \ov{Y}_t^{j})}
         + \sum_{i} \sum_{j = 0}^N |\Delta Z_t^{i,j}|^2 \Big) \d t + \d M_t\,,
    \end{split}
    \end{equation}
    with $M$ being a martingale. Using \Cref{prop.olboundsdisp} and the Lipschitz continuity of $(D_{(x^i,p)}H^i)_{i}$, we can integrate from $t \in [t_0,T]$ to $T$ and take expectations to find that 
\[
         \E\bigg[ |\Delta \bY_t|^2 + \int_{t}^T \sum_{i} \sum_{j = 0}^N |\Delta Z^{i,j}|^2\bigg] \lesssim \E\bigg[ \int_t^T \bigl( |\Delta \bY|^2 + |\bm E|^2 \bigr) \bigg],
\]
with dimension-free implied constant (also below); then apply Gr\"{o}nwall's lemma and plug in $t = t_0$ to get
     \begin{equation} \label{eibound}
      \E\bigg[ |\Delta \bY_{t_0}|^2 + \int_{t_0}^T \sum_{i} \sum_{j = 0}^N |\Delta Z^{i,j}|^2\bigg] \lesssim \E\bigg[ \int_{t_0}^T \!|\bm E|^2 \bigg] .
     \end{equation}
     Notice now that we can use \Cref{lem.gradbound} to estimate
     \[
     \begin{split}
     &\E\bigg[ \int_{t_0}^T \!|\bm E|^2 \bigg] \\
     &\lesssim \max_i \, \Bigl\lVert \sum_{j\neq i} |D_j u^i|^2 \Bigr\rVert_{L^\infty([T_0,T]\times(\R^d)^N)}\E \biggl[ \,\int_{t_0}^T \sum_{\substack{i,j \\ i \neq j}} \Bigl( |D_{ij}u^j(t,\bX_t)|^2 + \big|D_{x^i p} H^j(\bX_t, Y_t^{j,j}) \big|^2\Bigr)\,\d t \biggr] 
     \\
     &\lesssim (1+ \sqrt{M}e^{C_1{M}}\bigr)^2  \sigma^{-1} \delta \sum_j \delta^j \,, 
     \end{split}
     \]
    where we used the fact that 
    \begin{equation} \label{dpxH.formula}
        D_{ p x^i} H^j(\bx,p) =  D_{a x^i} L^j\big(\bx, - D_pH(\bx,p) \big) D_{pp} H(\bx,p)
    \end{equation}
    to deduce that $\sum_{i,j :\, i \neq j} |D_{x^i p} H^j(\bX_t, Y_t^{j,j}) |^2 \lesssim \sum_i \delta^i$.
     Come back to \eqref{eibound}, and recall the definitions of $Y^i$ and $Z^{i,j}$, as well as the arbitrariness of $t_0$ and $\bx_0$, to complete the proof.
\end{proof}

\begin{corollary} \label{prop.improvemenddisp}
     Let Assumption~\ref{assump.disp} hold. Suppose also that \eqref{aprioriopbound} holds for some $T_0 \in [0,T)$ and $M > 0$. Then, there is a dimension-free constant $C_3$ such that
     \begin{equation} \label{aposterioribound.dis}
        \sup_{t_0 \in [T_0,T],\, \bx_0 \in (\R^d)^N} \E\bigg[ \int_{t_0}^T \big| \bA(t,\bX_t^{t_0,\bx_0}) \big|_{\op}^2 \,\d t \bigg] \leq C_3 + C_3 Me^{C_3 {M}}\sigma^{-2} \delta \sum_{i} \delta^i. 
    \end{equation}
\end{corollary}

\begin{proof}
    Fix $(t_0,\bx_0) \in [T_0,T] \times (\R^d)^N$, and set $\bX = \bX^{t_0,\bx_0}$. By the triangle inequality, 
   \[
        \E\bigg[ \int_{t_0}^T \big| \bA(t,\bX_t) |_{\text{op}}^2 \,\d t \bigg] \leq 2T \norm{|D\bv|_\op}_\infty^2 + 2 \E\bigg[ \int_{t_0}^T \big| (\bA - D\bv)(t,\bX_t) \big|_{\op}^2 \,\d t \bigg] 
   \]
    hence it suffices to use \Cref{prop.olboundsdisp}, and \Cref{lem.clol}, recalling that the operator norm is bounded by the Frobenius norm. 
\end{proof}

We have now all the ingredients to prove our first main theorem.

\begin{proof}[Proof of \Cref{thm.maindisp}]
    With the notation of \Cref{prop.improvemenddisp}, choose $M_0 \defeq 2 C_3$. Note that the right-hand side of \eqref{aposterioribound.dis}, with $M = \frac32M_0 = 3C_3$, is bounded by $M_0$, provided that
    \[
   3C_3 e^{3C_3^2} \sigma^{-2} \delta \sum_i \delta^i \leq 1\,.
    \]
    Now set 
    \[
       T_0^* \defeq \inf \bigg\{T_0 \in [0,T] :\ \sup_{t_0 \in [T_0,T],\, \bx_0 \in (\R^d)^N} \E\bigg[ \int_{t_0}^T \big| \bA(t,\bX_t^{t_0,\bx_0}) \big|_{\op}^2 \,\d t \bigg] \leq M_0 \bigg\}.
    \]
   % which is well-defined as a real numbers since $M_0 > C_{D\bG,\lip}^2$.
    If $T_0^* > 0$, we can use the boundedness of $\bA$ to find some $\eps  > 0$ such that 
    \begin{equation} \label{epsstep}
        \sup_{t_0 \in [T_0^* - \eps, T_0^*],\, \bx_0 \in (\R^d)^N} \E\bigg[ \int_{t_0}^{T_0} \big| \bA(t,\bX_t^{t_0,\bx_0}) \big|_{\op}^2 \,\d t \bigg] \leq \frac{M_0}2\,,
    \end{equation}
    and so, for any $t_0 \in [T_0^* - \eps , T]$, 
    \begin{equation} \label{conditioning}
    \begin{split}
        &\E\bigg[ \int_{t_0}^{T} \big| \bA(t,\bX_t^{t_0,\bx_0}) \big|_{\op}^2 \,\d t \bigg]
         \\
        &\qquad\qquad= 
        \E\bigg[ \int_{t_0}^{T_0^*} \big| \bA(t,\bX_t^{t_0,\bx_0}) \big|_{\op}^2 \,\d t \bigg]
        + \E\bigg[ \int_{T_0^*}^{T} \big| \bA(t,\bX_t^{t_0,\bx_0}) \big|_{\op}^2 \, \d t \bigg]
         \\
        &\qquad\qquad=  \E\bigg[ \int_{t_0}^{T_0^*} \big| \bA(t,\bX_t^{t_0,\bx_0}) \big|_{\op}^2 \,\d t \bigg]
        + \E \bigg[\E \bigg[ \int_{T_0^*}^{T} \big| \bA(t,\bX_t^{t_0,\bx_0}) \big|_{\op}^2 \,\d t \bigg] \,\Big|\, \bX_{T_0^*}^{t_0,\bx_0} \bigg]
         \\
        &\qquad\qquad= \E\bigg[ \int_{t_0}^{T_0^*} \big| \bA(t,\bX_t^{t_0,\bx_0}) \big|_{\op}^2 \,\d t \bigg]
        + \E\bigg[ \int_{T_0^*}^{T} \Big| \bA\Bigl(t,\bX_t^{T_0^*,\bX_{T_0^*}^{t_0,\bx_0}}\Bigr) \Big|_{\op}^2 \, \d t \bigg]\,.
    \end{split}
    \end{equation}
    By the definition of $T_0^*$, the last expectation is bounded by $M_0$, hence combining \eqref{epsstep} and \eqref{conditioning} we have
    \[
     \sup_{t_0 \in [T_0^* - \eps, T],\, \bx_0 \in (\R^d)^N} \E\bigg[ \int_{t_0}^{T} \big| \bA(t,\bX_t^{t_0,\bx_0}) \big|_{\op}^2 \,\d t \bigg] \leq \frac{3}{2} M_0\,.
    \]
    By the observation at the beginning of the proof, we deduce that actually the right-hand side above can be improved to $M_0$; this contradicts the minimality of $T_0^*$, thus proving by contradiction that $T_0^*=0$, as desired. The second bound \eqref{mainl2bounddisp} follows from \Cref{lem.gradbound}.
\end{proof}

{
\begin{remark} The crucial estimates \eqref{eq.dvopbound.dis} and \eqref{aposterioribound.dis} give dimension-free Lipschitz regularity of open-loop and closed-loop equilibria. This is, in fact, all that is needed to ensure that the equilibrium dynamics are robust with respect to perturbations, making it possible to view the closed-loop system as a perturbation of the open-loop one. From the perspective of the master equation in MFG, Lipschitz regularity turns out to be enough to define a notion of equilibrium \cite{BLLLip}; the importance of Lipschitz properties shows up also in \cite{MouZhangAMS}.
\end{remark}
}

\section{Main estimates in the Lasry-Lions monotone case} \label{sec.LL}

\subsection{Bounds on the Pontryagin systems}

\label{subsec.pont.LL}

Let us denote, in this subsection, $v^{i,j} \coloneqq D_j v^i$. We recall that $v^{i,j}$ takes values in $\R^{d \times d}$, with $(v^{i,j})_{q,r} = v^{i_q, j_r} = D_{j_r} v^{i_q}$ for $q,r = 1,\dots,d$, and $D \bv = (v^{i,j})_{i,j = 1,\dots,N}$, which we view as a function taking values in $(\R^{d \times d})^{N \times N}$, is the Jacobian matrix of the vector field $\bv = (v^{i})_{i = 1,\dots,N} \colon (\R^d)^N \to (\R^d)^N$.

We will now derive a series of estimates under the additional assumption that for some $T_0 \in [0,T)$, we have a bound of the form 
\begin{equation} \label{dvansatz}
    \norm{ |D \bv|_{\text{op}}}_{L^{\infty}([T_0,T] \times (\R^d)^N)} \leq M
\end{equation}
for some $M < \infty$. This is equivalent to assuming that the vector field $\bv$ is $M$-Lipschitz in space on $[T_0,T] \times (\R^d)^N$. It will be useful to record here the (matrix-valued) equation for $v^{i,j}$ in $[0,T] \times (\R^d)^N$ obtained by differentating \eqref{pontpde}; recalling the separated structure of the Hamiltonian under Assumption \ref{assump.LL}, we find
\begin{equation} \label{eq.vij}
\begin{cases}
	 \ds - \partial_t v^{i,j} - L^{\bv} v^{i,j} + \sum_{1 \leq k \leq N} v^{i,k} D_{pp}H^k_0(x^k,v^k) v^{k,j} 
    \\[3pt]
    \quad {} + v^{i,j} D_{xp} H_0^j(x^j,v^j) + D_{px} H_0^i(x^i,v^i) v^{i,j} + D_{xx} H_0^i(x^i,v^i) \1_{i = j} = D_{ji} F^i
    \\[8pt]
    v^{i,j}(T,\bx) = D_{ji} G^i(\bx) \,,
\end{cases}
    \vspace{3pt}
\end{equation}
where $L^{\bv}$ indicates the differential operator 
\[
    L^{\bv} \defeq \sigma \sum_{k=1}^N \Delta_{k}\, + \sigma^0 \sum_{k,l=1}^N \tr\, D_{kl} \,- \sum_{k=1}^N D_p H_0(x^k,v^k) \cdot D_{k}\,, 
\]
which is applied entry-wise to the function $v^{i,j}$.

\begin{lemma} \label{lem.gradvbound}
Let \Cref{assump.LL} hold. Suppose also that \eqref{dvansatz} holds for some $T_0 \in [0,T)$ and $M > 0$. Then, there is a dimension-free constant $C_1$ %(depending only on $\norm{D_{px} \bH}_\infty$, $\norm{D_{pp} \bH}_\infty$, and $T$) 
such that, for each $i = 1,\dots, N$, we have 
\begin{equation} \label{est.vij.LL}
\Bigl\|\, \sum_{j \neq i} |D_j v^i|^2\Bigr\|_{L^{\infty}([T_0,T] \times (\R^d)^N)} \leq  C_1e^{C_1M} \kappa^i\,.
\end{equation}
\end{lemma}

\begin{proof}
Fix $(t_0,\bx_0) \in [T_0, T] \times (\R^d)^N$, set $\bX = \bX^{\OL, t_0,\bx_0}$, and define, for $i,j,k = 1,\dots,N$,
\begin{equation} \label{processesdef}
\begin{gathered} 
     Y_t^i \defeq v^i(t,\bX_t)\,, \quad Y_t^{i,j} \defeq v^{i,j}(t,\bX_t)\,, \\ Z_t^{i,j,k} \defeq \sqrt{2 \sigma} D_{k} v^{i,j}(t,\bX_t)\,, \quad Z_t^{i,j,0} \defeq \sqrt{2 \sigma_0} \sum_{k = 1}^N D_k v^{i,j}(t,\bX_t)\,.
\end{gathered}
\end{equation}
Then using \eqref{eq.vij} and It\^o's formula, we find that for $i \neq j$, 
\begin{align} \label{yijdynamics}
    \d Y_t^{i,j} &=  \Big(\,\sum_{k=1}^N Y_t^{i,k} D_{pp} H^k_t Y_t^{k,j}  + Y_t^{i,j} D_{xp} H^j_{t} + D_{px} H^i_{t} Y^{i,j}_t  - D_{ji}F^i(\bX_t)\Big) \d t 
    + \sum_{k = 0}^N Z_t^{i,j,k} \d W_t^k \,,
\end{align}
where we used the notation $D_{pp}H^i_{t} \defeq D_{pp}H^i_0(X^i_t,Y^i_t)$ and likewise for $D_{px} H^i_{t}$.
As a consequence, 
\begin{align}  \label{dyijcomp}
    \d \Bigl(\,\sum_{j \neq i} |Y_t^{i,j}|^2 \Bigr) &= \biggl( 2 \sum_{k,\,j \neq i} \tr \bigl( Y_t^{i,k} D_{pp} H^k_t Y_t^{k,j} (Y_t^{i,j})^\TT \bigr) + 2 \sum_{j \neq i} \tr\bigl( Y_t^{i,j} (D_{xp} H^j_t + D_{px} H^i_t) (Y_t^{i,j})^\TT \bigr) 
     \notag \\ 
    &\qquad -2 \sum_{j \neq i} \tr\bigl(  D_{ji} F^i(\bX_t)(Y_t^{i,j})^\TT \bigr) + \sum_{j \neq i} \sum_{k = 0}^N |Z_t^{i,j,k}|^2 \biggr) \d t + \d M_t\,, 
\end{align}
with $M$ being a martingale. The first term in this expression can be bounded in terms of the constant $M$ appearing in \eqref{dvansatz}: indeed, we use the Cauchy--Schwarz inequality $\tr(AB) \leq |A||B|$ to get, with implied constants depending only on $\lVert D_{pp}\bH_0 \rVert_\infty$,
\begin{align} \label{cyclic}
    \sum_{k,\,j \neq i} \tr \bigl( (Y_t^{i,j})^\TT Y_t^{i,k} D_{pp} H^k_t Y_t^{k,j} \bigr)
    &=  \sum_{k,\,j \neq i} \tr \bigl( Y_t^{i,k} D_{pp} H^k_t Y_t^{k,j} (Y_t^{i,j})^\TT \bigr) 
   \notag \\
    &\lesssim \sum_{k \neq i} |Y_t^{i,k}| \Bigl| \sum_{j \neq i} Y_t^{k,j} (Y_t^{i,j})^\TT \Bigr|  + |Y_t^{i,i}| \sum_{j \neq i} |Y_t^{i,j}|^2 
    \notag \\
    &\lesssim M \sum_{j \neq i} |Y_t^{i,j}|^2\,.
\end{align}
 Now, coming back to \eqref{dyijcomp}, we integrate from $s \geq t_0$ to $T$ and take expectations, and then use the above estimate, to find that 
\[
    \E\bigg[\, \sum_{j \neq i} |Y_{t}^{i,j}|^2 + \int_{t}^T \sum_{j \neq i} \sum_{k = 0}^N |Z^{i,j,k}|^2 \bigg] 
     %\\
    %&\qquad \leq \E\bigg[ \sum_{j \neq i} |D_{ij} G^i(t,\bX_t)|^2 - \int_{t_0}^T  2 \sum_{k,j \neq i} \tr \Big( (Y_t^{i,j})^T Y_t^{i,k} D_{pp} H^k(X_t^k, Y_t^{k}) Y_t^{k,j} \Big) + 2 \sum_{j \neq i} \tr\Big( (Y_t^{i,j})^T D_{xp} H^j(X_t^j, Y_t^{j}) Y_t^{i,j} \Big)-2 \sum_{j \neq i} \tr\Big( (Y_t^{i,j})^T D_{x^ix^j} F^i(t,\bX_t) \Big) dt \bigg] \\
    \lesssim \kappa^i + (1+M) \E\biggl[\,\int_{t}^T \sum_{j \neq i} |Y^{i,j}|^2 \biggr] ,
\]
with implied constant depending only on $\lVert D^2\bH_0 \rVert_\infty$ and $T$; Gr\"{o}nwall's lemma yields
\[
\E\bigg[\, \sum_{j \neq i} |Y_{t_0}^{i,j}|^2 \bigg] \leq C_1e^{C_1M} \kappa^i,
\]
with $C_1$ dimension-free, and we conclude by the arbitrariness of $t_0$ and $\bx_0$, recalling the definition of $Y^{i,j}_{t_0}$.
\end{proof}

\begin{lemma} \label{lem.gradvbound2}
Let Assumption \ref{assump.LL} hold. Suppose also that \eqref{dvansatz} holds for some $T_0 \in [0,T)$ and $M > 0$. Then, there is a dimension-free constant $C_2$ such that, for each $i = 1,\dots, N$, we have 
\begin{equation} \label{est.vij.LL2}
\Bigl\|\, \sum_{j \neq i} |D_i v^j|^2\Bigr\|_{L^{\infty}([T_0,T] \times (\R^d)^N)} \leq C_2 e^{C_2M} \tilde \kappa^i\,.
\end{equation}
\end{lemma}

\begin{proof}
It is sufficient to mimic the proof of \Cref{lem.gradvbound}, with the only major difference being that this time we have, using also the invariance of the trace under cyclic permutations,
\begin{equation*}
\begin{split} 
    \sum_{k,\,j\neq i} \tr \bigl( (Y_t^{j,i})^\TT Y_t^{j,k} D_{pp} H^k_t Y_t^{k,i} \bigr)
    &\lesssim \sum_{k \neq i} |Y_t^{k,i}| \Bigl| \sum_{j \neq i} (Y_t^{j,k})^\TT Y_t^{j,i} \Bigr|  + |Y_t^{i,i}| \sum_{j \neq i} |Y_t^{j,i}|^2 
    \lesssim M \sum_{j\neq i} |Y_t^{j,i}|^2
\end{split}
\end{equation*}
in place of \eqref{cyclic}.
\end{proof}

\begin{lemma} \label{prop.oldiagbounds}
    Let Assumption \ref{assump.LL} hold. Suppose also that \eqref{dvansatz} holds for some $T_0 \in [0,T)$ and $M > 0$. Then, there is a dimension-free constant $C_3$ such that
    \begin{equation}\label{diviabsolute}
 \norm{D_i v^i}_{L^\infty([T_0,T]\times(\R^d)^N)} \leq C_3\bigl(1 + \sigma^{-1} + e^{C_3M}\sqrt{\kappa^i \tilde\kappa^i}\,\bigr)\,.
\end{equation}
\end{lemma}

\begin{proof}
Fix $(t_0,\bx_0) \in [T_0,T] \times (\R^d)^N$, and define $\bX, Y^i, Y^{i,j}, Z^{i,j,k}$ as in the proof of \Cref{lem.gradvbound}; see \eqref{processesdef}. Using It\^o's formula we have 
\[ %\label{eq.yii}
    \d Y_t^{i} = \bigl(D_{x} H^i_t - D_{i} F^i( \bX_t) \bigr) \d t + \sqrt{2 \sigma} \sum_{k} Y_t^{i,k} \d W_t^k + \sqrt{2 \sigma_0} \sum_{k} Y_t^{i,k} \d W_t^0\,,
    \]
as well as, also using \eqref{eq.vij},
\begin{equation} \label{eq.yii}
    \d Y_t^{i,i} = \Bigl(\,\sum_{k} Y_t^{i,k} D_{pp} H^k_t Y_t^{k,i} + Y_t^{i,i} D_{xp} H^i_t + D_{px} H^i_t Y_t^{i,i} + D_{xx} H^i_t - D_{ii}F^i(\bX_t) \Bigr)\d t + \d M_t\,,
\end{equation}
where $D_{pp}H^i_t \defeq D_{pp}H^i_0(X^i_t,Y^i_t)$ (and likewise for $D_{px} H^i_t$ and $D_{xx} H^i_t$) and $\d M_t = \sum_{k=0}^N Z^{i,i,k}_t \d W^k_t$. By computing $\d | Y_t^i|^2$, we easily find that 
\begin{equation} \label{est.yi2.LL}
\begin{split}
   \E\bigg[\, |Y_t^i|^2 + 2\sigma \int_t^T \sum_{k} |Y_s^{i,k}|^2 \,\d s \bigg]
   & \leq \E\bigg[ |Y^{i}_T|^2 + 2 \int_t^T |Y_s^i| \big( |D_x H^i_s| + |D_{i} F^i(\bX_s)| \big) \,\d s \bigg]
 \\
    & \leq C_{\bm G,\lip}^2 + T(C_{\bm F,\lip}^2 + C^2_{\bH}) + 2 \E\bigg[ \int_t^T |Y_s^i|^2\, \d s  \bigg] \,. 
\end{split}
\end{equation}
By Gr\"{o}nwall's inequality, $\sup_{t \in [t_0,T]} \E\bigl[ |Y^i_t|^2 \bigr] \lesssim 1$, so plugging this estimate back into \eqref{est.yi2.LL} with $t = t_0$, we obtain
\begin{equation} \label{Yikbound}
    |Y^i_{t_0}|^2 + 2\sigma \E\bigg[\,\int_{t_0}^T \sum_{k = 1}^N |Y_t^{i,k}|^2 \,\d t \bigg] \lesssim 1,
\end{equation}
with implied constants depending only on those appearing in the right-hand side of~\eqref{est.yi2.LL}.
Next, using \eqref{eq.yii} we note that
\[
\begin{split}
|Y_{t_0}^{i,i}| &\leq C_{D\bm G,\lip} + T(C_{D\bm F,\lip} + \norm{D_{xx} \bH}_\infty + \norm{D_{px} \bH}_\infty^2) + (1+\norm{D_{pp} \bH}_\infty) \E\bigg[\,\int_{t_0}^T| Y^{i,i} |^2 \bigg]
\\
&\quad + \norm{D_{pp} \bH}_\infty \E\bigg[\int_{t_0}^T \sum_{k\neq i} |Y^{k,i} Y^{i,k}| \bigg] \Big)\,,
\end{split}
\]
hence, using \eqref{est.vij.LL}, \eqref{est.vij.LL2} and \eqref{Yikbound}, we obtain \eqref{diviabsolute} by the arbitrariness of $t_0$ and $\bx_0$.
\end{proof}

The following lemma shows that $D_i v^i$ is ``almost'' symmetric, as a $d \times d$ matrix. This rather subtle point will be used in the following steps, and is expected since, for large $N$, $v^i \approx D_i u^i$ (see \Cref{lem.llclol} below).

\begin{lemma} \label{lem.asymmetry} 
    Let Assumption \ref{assump.LL} hold. Suppose also that \eqref{dvansatz} holds for some $T_0 \in (0,T)$ and $M > 0$. Then, there is a dimension-free constant $C_4$ such that, for each $i = 1,\dots,N$, we have 
    \begin{equation*} %\label{disc.yii.LL}
        \| D_{i} v^i - (D_{i} v^i)^\TT \|_{L^{\infty}([T_0,T] \times (\R^d)^N)} \leq C_4e^{C_4M} \sqrt{\kappa^i\tilde\kappa^i}\,.
       \end{equation*}
\end{lemma}

\begin{proof}
Let $t_0$, $\bx_0$, $\bX$, $Y^i$, $Y^{i,j}$, $Z^{i,j,k}$ be as in the proof of \Cref{prop.oldiagbounds}. Moreover, set $\wt{Y}^{i,j} = (Y^{j,i})^\TT$, and take the transpose of the dynamics~\eqref{eq.yii} for $Y^{i,i}$ to find that
\[
\d \wt Y_t^{i,i} = \Bigl(\,\sum_{k} \wt Y_t^{i,k} D_{pp} H^k_t \wt Y_t^{k,i} + \wt Y_t^{i,i} D_{xp} H^i_t + D_{px} H^i_t \wt Y_t^{i,i} + D_{xx} H^i_t - D_{ii}F^i(\bX_t) \Bigr)\d t + \d \wt{M}_t\,,
\]
with $\wt M_t = \sum_{k=0}^N \wt Z^{i,i,k}_t \d W_t^k$, with obvious notation. Given matrices $A,B$ in $\R^{d\times d}$, we use the notations $\Delta A \defeq A - A^\TT$ and $\Sigma A \defeq A + A^\TT$, as well as the standard one $A : B \defeq \tr(A^\TT B)$, to use It\^o's formula and write 
    \begin{equation} \label{asymmetrycomp}
    \begin{multlined}[b][.85\displaywidth]
        \E\bigg[\, |\Delta Y_t^{i,i}|^2 + \int_t^T \sum_{k = 0}^N |\Delta Z_s^{i,i,k}|^2 \,\d s \bigg] 
        \\
        = -2 \E\bigg[\int_t^T \Delta Y_s^{i,i} : \Bigl( \Delta\bigl(Y_s^{i,i} D_{pp} H^i_s Y_s^{i,i} \bigr) + \Sigma \bigl( \Delta Y_t^{i,i} D_{xp} H^i_t \bigr) + \Delta E_s^i \Bigr)\d s \bigg]\,, 
     \end{multlined}
     \end{equation}
    where 
\[
        E_s^i \defeq \sum_{k \neq i} Y_s^{i,k} D_{pp} H^k_t Y_s^{k,i}\,.
\]
It is easy to estimate
\[
\bigl|\Delta\bigl(Y_s^{i,i} D_{pp} H^i_s Y_s^{i,i} \bigr) +  \bigl( \Delta Y_t^{i,i} D_{xp} H^i_s \bigr)\bigr| \leq 2\bigl(\norm{D_{pp} \bH}_\infty M + \norm{D_{px} \bH}_\infty \bigr) |\Delta Y^{i,i}_s|
\]
and, using \eqref{est.vij.LL} and \eqref{est.vij.LL2},
\[
|\Delta E^i_s|^2 \leq 4|E^i_s|^2 \leq 4\norm{D_{pp} \bH}_\infty^2 C_1C_2e^{(C_1+C_2)M}\kappa^i\tilde\kappa^i\,.
\]
Plugging these estimates into \eqref{asymmetrycomp} and applying Young's inequality, we deduce that 
    \begin{equation*} %\label{asymmetrycomp2}
      \E\bigl[ |\Delta Y_s^{i,i}|^2 \bigr] \lesssim e^{(C_1+C_2)M} \kappa^i\tilde\kappa^i + \int_t^T \E\bigl[ |\Delta Y^{i,i}|^2 \bigr],
    \end{equation*}
    with dimension-free implied constant. We can now conclude by invoking Gr\"{o}nwall's lemma and recalling the definition of $\Delta Y^{i,i}$ and the arbitrariness of $t_0$ and $\bx_0$.
\end{proof}

%\begin{corollary} \label{cor.oldiagbounds} \marginpar{\raggedright\tiny This Corollary is not used. Do we want to comment on why we're including it?}
%    Let Assumption \ref{assump.LL} hold. Suppose also that \eqref{dvansatz} holds for some $T_0 \in [0,T)$ and $M > 0$. Then, there is a constant $C_5$ (depending only on $\norm{D^2 \bH}_\infty$, $C_{D\bm F,\lip}$, and $T$) such that, for each $i = 1,\dots,N$ and any $(t,\bx) \in [T_0,T]\times(\R^d)^N$,
%\begin{equation*}%\label{diviabove}
%D_i v^i(t,\bx) \leq C_{D\bm G,\lip} + TC_5\bigl(1+e^{C_5M}{\kappa^i\tilde\kappa^i}\bigr)\,.
%\end{equation*}
%\end{corollary}

%\begin{proof}
%Recall the dynamics~\eqref{eq.yii} of $Y^{i,i}$ and note that by the strong convexity of $H^i$ and the bound \eqref{disc.yii.LL} we have, for any unit vector $\xi \in \R^d$,
%\[
%\xi^\TT Y^{i,i}_t D_{pp} H^i_t Y^{i,i}_t \xi \geq \lambda_{\bH} |Y^{i,i}_t|^2 - \norm{D_{pp}\bH}_\infty C_4(M)\sqrt{\kappa^i\tilde\kappa^i}|Y^{i,i}_t|\,.
%\]
%Then we have, by the weighted Young's inequality and estimates \eqref{est.vij.LL} and \eqref{est.vij.LL2},
%\[
%\xi^\TT Y^{i,i}_{t_0} \xi + \lambda_{\bH} \E\Bigl[ \, \int_{t_0}^T  |Y_t^{i,i}|^2 \,\d t \Bigr] \leq C_{D\bm G,\lip} + TC_{\bH,\bm F}\bigl(1+C_4(M)^2{\kappa^i\tilde\kappa^i}\bigr) + \frac{\lambda_ {\bH}}{2}\, \E\Bigl[ \, \int_{t_0}^T  |Y_t^{i,i}|^2 \,\d t \Bigr]\,,
%\]
%where $C_{\bH,\bm F}$ depends only on $\norm{D^2 \bH}_\infty$, $C_{D\bm F,\lip}$ and $T$, so here we do not need to use \eqref{Yikbound} to close the estimate.
%\end{proof}

\begin{proposition} \label{prop.olmonbounds} %\marginpar{\tiny DF: I simplified the estimates in this statement and lightened notation in the proof; I also removed the last part since we don't actually need the sharper bound that was proven therein.}
    Let Assumption \ref{assump.LL} hold. Suppose also that \eqref{dvansatz} holds for some $T_0 \in [0,T)$ and $M > 0$. There are dimension-free constants $C_5$ and $C_5'$ such that, if
     \begin{equation} \label{smallnessll2}
C_{\bG,\lm} + TC_{\bF,\lm} + C_4e^{C_4M} \sqrt{\kappa\tilde\kappa} \leq \frac{\lambda_{\bH}}{2T}\,e^{-C_5'(1+\sigma^{-1}+C_3e^{C_3M}\sqrt{\kappa\tilde\kappa}\,)}\,,
 \end{equation}
    then
    \begin{equation} \label{dvop.bound.LL}
        \norm{|D\bm v|_\op}_{L^\infty([T_0,T]\times(\R^d)^N)} \leq %C_5\bigl(1+\sigma^{-1} + e^{C_5M} \sqrt{\kappa\tilde\kappa}\, \bigr)
    e^{C_5(1+\sigma^{-1}+C_3e^{C_3M}\sqrt{\kappa\tilde\kappa}\,)}\,.
    \end{equation}
    (The constants $C_3$ and $C_4$ those in \Cref{prop.oldiagbounds,lem.asymmetry}.)
\end{proposition}

\begin{proof}
    Fix $(t_0,\bm x_0) \in [T_0,T] \times (\R^d)^N$ and $\bm \xi_0 \in (\R^d)^N$. Let $\bX \defeq \bX^{t_0,\bx_0}$, and define $Y_t^i$, $Y_t^{i,j}$, $Z_t^{i,j,k}$ as in \eqref{processesdef}. Let $\bxi$ be the $(\R^d)^N$-valued process with dynamics
    \begin{equation} \label{xidynamics}
    \begin{cases}
       \d \xi_t^i = -\Bigl( D_{pp} H^i_0(X_t^i,Y_t^i) \ds\sum_{1 \leq k \leq N} Y_t^{i,k} \xi_t^k + D_{xp} H^i_0(X_t^i,Y_t^i) \xi_t^i \Bigr) \d t \\
        \xi^i_{t_0} = \xi^i_0 \,.
        \end{cases}
        \end{equation}
        In the following, in order to ease the notation, we are going to omit the arguments of various functions, like we did in the previous proofs, by using notation like $D_{pp} H^i_t = D_{pp} H^i_0(X_t^i,Y_t^i)$, $D_{ij} F_t^i = D_{ij} F^i(\bX_t)$, and so forth. Dimension-free constants that appear throughout the proof will be denoted by $c_i$ ($i \in \N$).
        
        \smallskip
        
    \noindent$\bullet$ \emph{Step 1: evolution of $\sum_i |\sum_j Y_t^{i,j} \xi_t^j|^2$.} 
    Recall the dynamics \eqref{yijdynamics} and \eqref{eq.yii} and notice that, for each fixed $i$, we have, after a few cancellations,  
    \begin{equation*}
        \d \sum_{j} Y_t^{i,j} \xi_t^j = \sum_{j} \bigl( D_{px} H_t^i Y_t^{i,j} - D_{ij} F_t^i \bigr) \xi_t^j \,\d t + D_{xx} H^i_t \xi^i_t \,\d t + \d M_t \,,
    \end{equation*}
    with $M$ being a martingale.
    It follows that 
    \begin{equation*}
    \begin{split}
        \d \sum_{i} \Big| \sum_{j} Y_t^{i,j} \xi_t^j \Big|^2 &= 2\sum_{i} \Big( \Bigl(\,\sum_{j} Y_t^{i,j} \xi_t^j\Bigr)^\TT D_{px} H_t^i \Bigl(\,\sum_{j} Y_t^{i,j} \xi_t^j \Bigr) \d t \\
        &\quad + 2 \sum_{i} \Big(\sum_{j} Y_t^{i,j} \xi_t^j\Big)^\TT D_{xx}H^i_t \xi^i_t \,\d t
        \\
        &\quad - 2 \sum_{i} \Big(\sum_{j} Y_t^{i,j} \xi_t^j\Big)^\TT \sum_{j} D_{ji} F_t^i \xi_t^j \,\d t  + \d M'_t, 
        \end{split}
    \end{equation*}
    with $M'$ being a martingale. Integrating and taking expectations, and then using Young's inequality, we find that 
    \begin{equation*} %\label{est.s2.ll}
    \begin{split}
        \E\Big[\sum_{i} \Big| \sum_{j} Y_t^{i,j} \xi_t^j \Big|^2 \Big] &\leq  \E\Big[\sum_{i} \Big| \sum_{j} Y_T^{i,j} \xi_T^j \Big|^2 \Big] + \bigl( C_{D\bm F,\lip}^2 + \norm{D_{xx} \bH}_\infty^2 \bigr) \E\bigg[ \int_t^T |\bm\xi|^2 \bigg]
        \\
        &\quad + 2\norm{D_{px} \bH}_\infty \E\bigg[ \int_{t}^T \sum_{i} \Big| \sum_{j} Y^{i,j} \xi^j \Big|^2 \bigg]\,,
        \end{split}
    \end{equation*}
    hence Gr\"{o}nwall's inequality gives 
       \begin{equation} \label{fromxi2toyxi}
   \sum_{i} \Big| \sum_{j} Y_{t_0}^{i,j} \xi_0^j \Big|^2 \leq c_1\biggl( \E\bigl[ |\bm\xi_T|^2 \bigr] + \E\bigg[ \int_{t_0}^T |\bm\xi|^2 \bigg] \biggr).
        \end{equation}

     \smallskip
    
    \noindent$\bullet$ \emph{Step 2: evolution of $|\bm \xi|^2$.} From \eqref{xidynamics} we see that 
    \[
    |\xi^i_t| \leq |\xi^i_0| + \norm{D_{pp} \bH}_\infty \int_{t_0}^t \,\Bigl|\,\sum_{j \neq i} Y^{i,j}\xi^j \Bigr| + \bigl( \norm{D_{px}\bH}_\infty + \norm{D_{pp} \bH}_\infty \norm{Y^{i,i}}_\infty\bigr) \int_{t_0}^t |\xi^i|\,,
    \]
    so Gr\"{o}nwall's inequality, along with \Cref{prop.oldiagbounds}, yields
    \[
    |\xi^i_t| \leq \biggl( |\xi^i_0| + \norm{D_{pp} \bH}_\infty \int_{t_0}^t \Bigl|\,\sum_{j \neq i} Y^{i,j}\xi^j \Bigr| \,\biggr) e^{c_2\eta}\,,
    \]
    with 
    %\begin{equation*} %\label{prooflunga.v.etadef}
    $\eta \defeq 1+\sigma^{-1}+C_3 e^{C_3M}\sqrt{\kappa\tilde\kappa}.$ 
    %\end{equation*}
     Taking now the squares, using H\"{o}lder's inequality, and summing over $i$, we obtain
    \begin{equation} \label{finest.step2}
    |\bxi_t|^2 \leq 2\biggl( |\bxi_0|^2 + \norm{D_{pp} \bH}_\infty^2 \,T \int_{t_0}^t \sum_i \, \Bigl|\,\sum_{j \neq i} Y^{i,j}\xi^j \Bigr|^2 \,\biggr) e^{2c_2\eta}\,.
    \end{equation}
   
    \smallskip
    
    \noindent$\bullet$ \emph{Step 3: evolution of ${\sum_{i \neq j} (\xi_t^i)^T Y_t^{i,j} \xi_t^j}$.} Combining \eqref{yijdynamics} and \eqref{xidynamics}, we find that
%\begin{align*}
% d \Big( \sum_{\substack{1 \leq i, j \leq N \\ j \neq i} } &(\xi_t^i)^T Y_t^{i,j} \xi^j_t \Big) =
% - \sum_{\substack{1 \leq i, j \leq N \\ j \neq i}} (\xi^i_t)^T Y_t^{i,j} \Big( \sum_{k = 1}^N D_{pp} H^j_t Y_t^{j,k} \xi_t^k + D_{px} H^j_t \xi_t^j \Big) dt
%\\ &+ \sum_{\substack{1 \leq i, j \leq N \\ j \neq i}} (\xi_t^i)^T \Big(  \sum_{k = 1}^N Y_t^{i,k} D_{pp}H^k_t Y_t^{k,j} + Y_t^{i,j} D_{px} H^j_t + \big(D_{px} H^i_t \big)^T Y_t^{i,j} - D_{x^ix^j} F^i_t \Big) \xi^j_t dt
%\\
%&- \sum_{\substack{1 \leq i, j \leq N \\ j \neq i}} \Big( \sum_{k = 1}^N (\xi_t^k)^T (Y_t^{i,k})^T D_{pp} H^i_t + (\xi_t^i)^T \big( D_{px} H^i_t\big)^T\Big) Y_t^{i,j} \xi_t^j dt + dM_t, 
%\end{align*}
%Notice that the terms involving $D_{px} H^i$ or $D_{px} H^j$ cancel. This brings us to 
    \begin{equation} \label{dxiyijxij}
        \d \biggl(\, \sum_{\substack{1 \leq i, j \leq N \\ j \neq i} } (\xi_t^i)^\TT Y_t^{i,j} \xi^j_t \biggr) = \biggl( S_t^1 + S_t^2 + S_t^3 - \sum_{\substack{1 \leq i, j \leq N \\ j \neq i}} (\xi^i_t)^\TT D_{ji} F^i_t \xi_t^j \biggr) \d t + \d M_t\,, 
    \end{equation}
    with 
    \begin{equation*}\begin{aligned}
        S_t^1 &= - \sum_{\substack{i, j \\ j \neq i}} (\xi^i_t)^\TT Y_t^{i,j}  \sum_{k} D_{pp} H^j_t Y_t^{j,k} \xi_t^k \,, \\
        S_t^2 &= \sum_{\substack{i,j \\ j \neq i}} (\xi_t^i)^\TT  \sum_{k} Y_t^{i,k} D_{pp}H^k_t Y_t^{k,j}  \xi^j_t \,, 
        \\
        S_t^3 &= - \sum_{\substack{i,j \\ j \neq i}} \sum_{k} (\xi_t^k)^\TT (Y_t^{i,k})^\TT D_{pp} H^i_t Y_t^{i,j} \xi_t^j \,,
    \end{aligned} \end{equation*}
and with $M_t$ being a martingale. 
    Now notice that
    \begin{equation*}\begin{aligned}
        S_t^1 &= - \sum_{i,j,k} (\xi_t^{i,j})^\TT Y_t^{i,j} D_{pp} H^j_t Y_t^{j,k} \xi_t^k + \sum_{i,k} (\xi_t^i)^\TT Y_t^{i,i} D_{pp} H^i_t Y_t^{i,k} \xi_t^k\,, 
        \\
        S_t^2 &= \sum_{i,j,k} (\xi_t^{i,j})^\TT Y_t^{i,j} D_{pp} H^j_t Y_t^{j,k} \xi_t^k - \sum_{i,k} (\xi_t^i)^\TT Y_t^{i,k} D_{pp} H^k_t Y_t^{k,i} \xi_t^i\,, 
        \\
        S_t^3 &= - \sum_{i} \Bigl(\, \sum_{j \neq i} Y_t^{i,j} \xi^j \Bigr)^\TT D_{pp} H^i_t \Bigl(\, \sum_{j \neq i} Y_t^{i,j} \xi^j \Bigr) - \sum_{\substack{i,j \\ j \neq i}} (\xi_t^i)^\TT (Y_t^{i,i})^\TT D_{pp}H_t^i Y_t^{i,j} \xi_t^j  \,,
    \end{aligned} \end{equation*}
    so the first term in the expression for $S_t^1$ cancels with the first term for $S_t^2$; then, combining the remaining terms we find that
    \begin{equation*}
    \begin{split}
        S_t^1 + S_t^2 + S_t^3 &=  - \sum_{i} \Bigl(\, \sum_{j \neq i} Y_t^{i,j} \xi^j \Bigr)^\TT D_{pp} H^i_t \Bigl( \,\sum_{j \neq i} Y_t^{i,j} \xi^j \Bigr) + \sum_{\substack{i,j \\ j \neq i}} (\xi_t^i)^\TT \Bigl( \bigl(Y_t^{i,i} - (Y_t^{i,i})^\TT\bigr) D_{pp} H_t^i Y_t^{i,j} \Bigr) \xi_t^j
        \\
        &\quad - \sum_{i} (\xi_t^i)^\TT  \Bigl(\, \sum_{j \neq i} Y_t^{i,j} D_{pp}H_t^j Y_t^{j,i} \Bigr) \xi_t^i\,.
        \end{split}
    \end{equation*}
    We now plug this into~\eqref{dxiyijxij}, we integrate from $t_0$ to $T$ and take expectations, and then we use the semi-monotonicity of $\bm F$ and $\bm G$, the strong convexity of $H^i$, to deduce that
    \[
    \begin{split}
    \lambda_{\bH} \E\bigg[ \int_{t_0}^T \sum_{i}\, \Bigl|\, \sum_{j \neq i} Y^{i,j} \xi^j \Bigr|^2\bigg]
        &\leq \sum_{\substack{i, j \\ j \neq i}} (\xi_{0}^i)^\TT Y_{t_0}^{i,j} \xi_0^j + C_{\bG,\lm} \E\big[|\bm\xi_T|^2 \big] + C_{\bF,\lm} \E\bigg[ \int_{t_0}^T |\bm\xi|^2 \bigg]
         \\[-2pt]
         &\quad+ \norm{D_{pp} \bH}_\infty \Bigl(\sup_i\, \bigl\lVert Y^{i,i} - (Y^{i,i})^\TT\bigr\rVert_\infty\Bigr) \E \biggl[\, \int_{t_0}^T\! |\bm\xi| \Bigl(\, \sum_i\, \Bigl\lvert\, \sum_{j \neq i} Y^{i,j} \xi^j \Bigr\rvert^2 \Bigr)^{\frac12} \biggr]
         \\
         &\quad+ \norm{D_{pp}\bH}_\infty \sup_i \Bigl(\, \Bigl\lVert \sum_{j \neq i} |Y^{i,j}|^2 \Bigr\rVert_\infty\,\Bigl\lVert \sum_{j \neq i} |Y^{j,i}|^2 \Bigr\rVert_\infty \,\Bigr)^{\frac12} \,\E \biggl[\, \int_{t_0}^T |\bm\xi|^2 \biggr]\,,
    \end{split}
    \]
where we also used the Cauchy-Schwarz inequality to obtain the last two terms, and the $\infty$-norms of terms involving $Y^{i,j}$ are understood as norms in $L^\infty([t_0,T] \times \Omega)$.
%    \begin{align*}
%        \E&\Big[ \sum_{\substack{1 \leq i, j \leq N \\ j \neq i}} (\xi_T^i)^j Y_T^{i,j} \xi_T^j \Big] - \sum_{\substack{1 \leq i, j \leq N \\ j \neq i}} (\xi_{0}^i)^T Y_{t_0}^{i,j} \xi_0^j
 %       \\
%        &=
%        \E\bigg[ \int_{t_0}^T \Big( - \sum_{i = 1}^N \big( \sum_{j \neq i} Y_t^{i,j} \xi^j \big)^T D_{pp} H^i_t \big( \sum_{j \neq i} Y_t^{i,j} \xi^j \big) + \sum_{\substack{1 \leq i, j \leq N \\ j \neq i}} (\xi_t^i)^T \Big( (Y_t^{i,i} - (Y_t^{i,i})^T) D_{pp} H_t^i Y_t^{i,j} \Big) \xi_t^j
%        \\
%        &\qquad - \sum_{i = 1}^N  (\xi_t^i)^T  \Big( \sum_{j \neq i} Y_t^{i,j} D_{pp}H_t^j Y_t^{j,i} \Big) \xi_t^i \Big) dt \bigg] + C_{\bF} \E\bigg[ \int_{t_0}^T \sum_{i = 1}^N |\xi_t^i|^2 \bigg] .
%    \end{align*}
Using Young's inequality to deal with the penultimate term, as well as \Cref{lem.gradvbound,lem.asymmetry}, we obtain
	\begin{align} \label{finest.step3}
     \lambda_{\bH} \E\bigg[ \int_{t_0}^T \sum_{i}\, \Bigl|\, \sum_{j \neq i} Y^{i,j} \xi^j \Bigr|^2\bigg]
        &\leq \sum_{\substack{i, j \\ j \neq i}} (\xi_{0}^i)^\TT Y_{t_0}^{i,j} \xi_0^j + C_{\bG,\lm} \E\big[|\bm\xi_T|^2 \big] + C_{\bF,\lm} \E\bigg[ \int_{t_0}^T |\bm\xi|^2 \bigg]
        \notag \\[-4pt]
         &\quad+ \norm{D_{pp} \bH}_\infty \Bigl(\, \frac 12\, C_4e^{C_4M} +  \sqrt{C_1C_2}\,e^{\frac12(C_1+C_2)M} \,\Bigr) \sqrt{\kappa\tilde\kappa} \,\,\E \biggl[\, \int_{t_0}^T |\bm\xi|^2 \biggr]
         \notag \\
         &\quad+ \frac 12\, \norm{D_{pp} \bH}_\infty C_4e^{C_4M} \sqrt{\kappa\tilde\kappa}  \,\,\E \biggl[\, \int_{t_0}^T \sum_i\, \Bigl\lvert\, \sum_{j \neq i} Y^{i,j} \xi^j \Bigr\rvert^2 \biggr]\,.
         \end{align}

\smallskip

\noindent$\bullet$ \emph{Step 4: closing the estimates}. Combining estimates~\eqref{finest.step2} and \eqref{finest.step3} obtained in steps~$2$ and~$3$ respectively, we have
	\[
	\begin{split}
     \lambda_{\bH} \E\bigg[ \int_{t_0}^T \!\sum_{i}\, \Bigl|\, \sum_{j \neq i} Y^{i,j} \xi^j \Bigr|^2\bigg] 
     &\leq  \sum_{\substack{i, j \\ j \neq i}} (\xi_{0}^i)^\TT Y_{t_0}^{i,j} \xi_0^j + \bigl(1+C_{\lm,T} +\wt C_M \sqrt{\kappa\tilde\kappa} \,\bigr) e^{c_3\eta} |\bxi_0|^2
     \\[-5pt]
     &\quad
     + T\bigl( C_{\lm,T} 
     + C_4e^{C_4M} \sqrt{\kappa\tilde\kappa}\,\bigr) e^{c_3\eta}  \E \biggl[\, \int_{t_0}^T\! \sum_i\, \Bigl\lvert\, \sum_{j \neq i} Y^{i,j} \xi^j \Bigr\rvert^2 \biggr]\,,
     \end{split}
	\]
	where
	\[
	C_{\lm,T} \defeq C_{\bG,\lm} + TC_{\bF,\lm}\,,
	\qquad
	 \wt C_M \defeq C_4e^{C_4M} +  \sqrt{C_1C_2}\,e^{\frac12(C_1+C_2)M}\,.
	\]
	Therefore, if \eqref{smallnessll2} holds with $C_5' = c_3$, we obtain
	\begin{equation} \label{step4.est1}
	\E\bigg[ \int_{t_0}^T \sum_{i}\, \Bigl|\, \sum_{j \neq i} Y^{i,j} \xi^j \Bigr|^2\bigg]
	\leq \frac{2}{\lambda_{\bH}} \biggl(\, \sum_{\substack{i, j \\ j \neq i}} (\xi_{0}^i)^\TT Y_{t_0}^{i,j} \xi_0^j + \bigl(1+C_{\lm,T} +\wt C_M \sqrt{\kappa\tilde\kappa} \,\bigr) |\bxi_0|^2 e^{c_3\eta} \biggr)\,.
	\end{equation}
	Plugging \eqref{step4.est1} back into \eqref{finest.step2}, we obtain
	\begin{equation} \label{step4.est2}
	\sup_{t \in [t_0,T]} \E \bigl[ |\bxi_t|^2 \bigr] \leq c_4 \biggl( \bigl(1+C_{\lm,T} +\wt C_M \sqrt{\kappa\tilde\kappa} \,\bigr)e^{c_3\eta}|\bxi_0|^2 + \sum_{\substack{i, j \\ j \neq i}} (\xi_{0}^i)^\TT Y_{t_0}^{i,j} \xi_0^j \biggr) e^{2c_2\eta} \,.
	\end{equation}
	% \[
	% \mathcal{C}_1 \defeq  2e^{2c_1\eta} + \mathcal{C}_2\bigl(1+C_{\lm,T} +\wt C(M) \sqrt{\kappa\tilde\kappa} \,\bigr) C'' e^{TC''\eta}\,,
	% \qquad
	% \mathcal{C}_2 \defeq  \frac{4T\norm{D_{pp}\bH}_\infty}{\lambda_{\bH}}\,e^{TC'\eta}\,.
	% \]
	Note now that using \Cref{prop.oldiagbounds} we have
	\begin{equation} \label{step4.est4}
	 \sum_{\substack{i, j \\ j \neq i}} (\xi_{0}^i)^\TT Y_{t_0}^{i,j} \xi_0^j \leq  \sum_{i, j} (\xi_{0}^i)^\TT Y_{t_0}^{i,j} \xi_0^j + c_5\eta |\bxi_0|^2\,,
	\end{equation}
	 %depends only on $\norm{D^2 \bH}_\infty$, $C_{\bH}$, $C_{\bF,\lip}$, $C_{\bG,\lip}$, $C_{D\bm F,\lip}$, $C_{D\bm G,\lip}$ and $T$, 
    so combining \eqref{fromxi2toyxi}, \eqref{step4.est2} and \eqref{step4.est4} we get
	\begin{equation*} %\label{step4.est3'}
	\sum_{i} \Big| \sum_{j} Y_{t_0}^{i,j} \xi_0^j \Big|^2 \leq c_6 \biggl( \bigl(1+C_{\lm,T} +\wt C_M \sqrt{\kappa\tilde\kappa} + \eta \bigr)e^{c_3\eta}|\bxi_0|^2 + \sum_{i,j} (\xi_{0}^i)^\TT Y_{t_0}^{i,j} \xi_0^j \biggr) e^{2c_2\eta}\,;
	\end{equation*}
	% with
	% \begin{equation} \label{step4.defC3}
	% \mathcal{C}_3 \defeq \Bigl( C_{D\bm G,\lip}^2 + T\bigl( C_{D\bm F,\lip}^2 + \norm{D_{xx} \bH}_\infty^2 \bigr) \Bigr) e^{2T\norm{D_{xp}\bH}_\infty}\,;
	% \end{equation}
	therefore, exploiting a weighted Young's inequality, we obtain
	\begin{equation*} %\label{step4.est5}
	 \sum_{i} \Big| \sum_{j} Y_{t_0}^{i,j} \xi_0^j \Big|^2 \leq c_7 \bigl(1+C_{\lm,T} +\wt C_M \sqrt{\kappa\tilde\kappa} + \eta \bigr)e^{c_3\eta} |\bxi_0|^2\,,
	\end{equation*}
	By the arbitrariness of $t_0$, $\bx_0$ and $\bxi_0$, this proves \eqref{dvop.bound.LL}.
	\end{proof}

\begin{corollary} \label{prop.olboundslip}
   Let Assumption \ref{assump.LL} hold. Let $C_3$ and $C_5$ be the constants in \Cref{prop.olmonbounds}. Let
   \begin{equation} \label{LL.defM0}
   \bar M \defeq e^{2C_5(1+\sigma^{-1})},
   \end{equation}
   and suppose that we have both \eqref{smallnessll2} and
   \begin{equation} \label{LL.smallness2'}
   \sqrt{\kappa\tilde\kappa} \leq \frac{C_5}{C_3}e^{-C_3M},
   \end{equation}
for some $M > \bar M$.
   Then
   \[
   \norm{|D\bm v|_\op}_{L^\infty([0,T]\times(\R^d)^N)} \leq \bar M\,.
   \]
\end{corollary}

\begin{proof} 
Define
\begin{equation*}
    T_0^* \defeq \inf \Bigl\{ T_0 \in [0,T] :\ \norm{|D\bm v|_\op}_{L^\infty([T_0,T]\times(\R^d)^N)} \leq \bar M \Bigr\};
\end{equation*}
note that $T_0^*$ is well-defined as a real number, and our goal is to show that $T_0^* = 0$. Suppose that this is not the case (hence $T^*_0 \in (0,T]$); as $D\bm v$ is uniformly continuous, we can find $\wt T_0 \in (0,T_0^*)$ with the property that $\norm{|D\bm v|_\op}_{L^\infty([\wt T_0,T] \times (\R^d)^N)} \leq M$,\footnote{Note that we should guarantee that $M > C_{D\bG,\lip}$; we can simply add this as an assumption, but by carefully tracking all dimension-fre constants one can see that we already have $\bar M > C_{D\bG,\lip}$, as expected for the bound \eqref{dvop.bound.LL}.}
%\marginpar{\tiny DF: Added footnote, since we have less explicit constants now.}
so we can apply \Cref{prop.olmonbounds} and use condition~\eqref{LL.smallness2'} to deduce that actually $\norm{|D\bm v|_\op}_{L^\infty([\wt T_0,T] \times (\R^d)^N)} \leq \bar M$, thus contradicting the minimality of $T_0^*$.
\end{proof}

\begin{remark}
Clearly, it is possible to replace condition~\eqref{LL.smallness2'} with $\sqrt{\kappa\tilde\kappa} \lesssim e^{-C_3M}$, with implied constant independent of $M$, $N$, and $\sigma$. In that case, the thesis of \Cref{prop.olboundslip} holds with some different constant $\wt C_5$ in place of $C_5$ in the definition~\eqref{LL.defM0} of $\bar M$.
\end{remark}

\subsection{From the Pontryagin systems to the Nash system}
\label{subsec.nash.LL}

We now aim to prove an analogue of Theorem \ref{thm.maindisp} in the Lasry--Lions semi-monotone setting. The precise statement is as follows.

\begin{theorem} \label{thm.mainll}
    Let \Cref{assump.LL} hold. There is a dimension-free constant $C$ such that the following is true: if \eqref{smallness1}, \eqref{smallness1'} and \eqref{smallness2} all hold,
then
    \begin{equation*} %\label{mainopboundll}
        \sup_{t_0 \in [0, T],\, \bx_0 \in (\R^d)^N} \E\bigg[ \int_{t_0}^T \big| \bA(t,\bX_t^{\CL,t_0,\bx_0}) \big|_{\op}^2 \,\d t \bigg] \leq  e^{C(1+\sigma^{-1})}
    \end{equation*}
    and
    \begin{equation*} %\label{mainl2boundll}
      \Bigl\|\,  \sum_{j \neq i} |D_{j} u^i|^2 \Bigr\|_{\infty}  \leq  e^{e^{C(1+\sigma^{-1})}} \delta^i \,,
    \end{equation*}
    for all $i = 1,\dots,N$.
    \end{theorem}

As in the displacement monotone setting, we start with a lemma which gives an estimate on $\sum_{j \neq i} |D_j u^i|^2$, conditional on control of the key matrix $\bA$. The proof is almost identical to that of  Lemma \ref{lem.gradbound}, and so is omitted.

\begin{lemma} \label{lem.gradbound.LL}
    Let Assumption~\ref{assump.LL} hold. Suppose also that \eqref{aprioriopbound} holds for some $T_0 \in [0,T)$ and $M > 0$. Then, there is a dimension-free constant $C_6$ such that, for each $i = 1,\dots,N$, 
\[
 \sum_{j \neq i} \,\bigl|D_j u^i(t_0,\bx_0)\bigr|^2 + \sigma \, \E\bigg[\int_{t_0}^T \sum_{j \neq i} \sum_{k} \big|D_{jk} u^i(t,\bX_t^{t_0,\bx_0}) \big|^2 \,\d t \bigg] 
    \leq \bigl(1+ C_6 \sqrt{M} e^{C_6 {M}}\bigr) \delta^i\,.
\]
for all $t_0 \in [T_0, T]$, $\bx_0 \in (\R^d)^N$.
\end{lemma}

We are going to prove now the counterparts, in the Lasry--Lions semi-monotone setting, of \Cref{lem.clol} and \Cref{prop.improvemenddisp}.

\begin{lemma} \label{lem.llclol}
     Let the assumptions of \Cref{prop.olboundslip} hold. Suppose also that \eqref{aprioriopbound} holds for some $T_0 \in [0,T)$ and $M > 0$. Then, there is a dimension-free constant $C_7$ such that
\[
\begin{multlined}[.85\displaywidth]
 \sum_{i} \,\bigl|(D_i u^i-v^i)(t_0,\bx_0)\bigr|^2 + \sigma \, \E\bigg[\int_{t_0}^T \!\sum_{i,j} \big|(D_{ji} u^i - D_jv^i)(t,\bX_t^{t_0,\bx_0}) \big|^2 \,\d t \bigg] 
 \\
 \leq C_7M e^{C_7{M}+e^{C_7(1+\sigma^{-1})}} \!\delta \sum_i \delta^i
 \end{multlined}
\]
for all $t_0 \in [T_0, T]$, $\bx_0 \in (\R^d)^N$.
\end{lemma}

\begin{proof}
    Follow the proof of \Cref{lem.clol}, with the only main difference being that \Cref{prop.olboundslip} is to be invoked (instead of \Cref{prop.olboundsdisp}) to control $|D\bv|_\op$, so, after using Gr\"{o}nwall's lemma we have
    \[
     \E\bigg[ |\Delta \bY_{t_0}|^2 + \int_{t_0}^T \sum_{i} \sum_{j = 0}^N |\Delta Z^{i,j}|^2\bigg] \leq c\,\E\bigg[ \int_{t_0}^T \!|\bm E|^2 \bigg] e^{\ds e^{c(1+\sigma^{-1})}},
     \]
     with dimension-free $c$.
\end{proof}

\begin{corollary} \label{prop.improvementll}
     Let the assumptions of \Cref{prop.olboundslip} hold. Suppose also that \eqref{aprioriopbound} holds for some $T_0 \in [0,T)$ and $M > 0$. Then, there is a dimension-free constant $C_{8}$ such that
     \begin{equation} \label{aposterioribound.ll}
     \begin{multlined}[b][.85\displaywidth]
        \sup_{t_0 \in [T_0,T],\, \bx_0 \in (\R^d)^N} \E\bigg[ \int_{t_0}^T \big| \bA(t,\bX_t^{t_0,\bx_0}) \big|_{\op}^2 \,\d t \bigg] 
        \\
        \leq C_{8}  e^{C_{8}(1+\sigma^{-1})} + 2C_7M e^{C_7{M} + e^{C_7(1+\sigma^{-1})}} \sigma^{-1} \delta \sum_i \delta^i
        \end{multlined}
     \end{equation}
for all $t_0 \in [T_0, T]$, $\bx_0 \in (\R^d)^N$. ($C_9$ is the constant introduced in \Cref{lem.llclol}.)
\end{corollary}

\begin{proof}
Same as that of \Cref{prop.improvemenddisp}, but invoke \Cref{prop.olboundslip} and  \Cref{lem.llclol} (instead of \Cref{prop.olboundsdisp} and \Cref{lem.clol}, respectively). 
\end{proof}

We can now prove our second main result.

\begin{proof}[Proof of \Cref{thm.mainll}]
Choose $M_0 = 3C_{8} e^{C_{8}(1+\sigma^{-1})}$ and note that the right-hand side of \eqref{aposterioribound.ll}, with $M=\frac43 M_0$, is bounded by $M_0$, provided that
\[
4 C_7 \exp\bigl(4C_7{C_{8}}e^{C_{8}(1+\sigma^{-1})} + e^{C_7(1+\sigma^{-1})} \bigr) \delta \sum_i \delta^i \leq 1\,.
\]
The rest of the proof proceed as that of \Cref{thm.maindisp}.
\end{proof}

\section{Comparing closed-loop, open-loop, and distributed equilibria}
\label{sec.olcl}

We recall that, as before, equilibria are assumed to exist and be admissible in the sense described in Section~\ref{sec:admiss}.

\subsection{Comparing closed-loop and open-loop equilibria}

\begin{proof}[Proof of Theorem \ref{thm.olcldisp}]
    Throughout this proof, implied constants will be understood to be dimension-free. Fix an initial condition $(t_0,\bzeta_0)$, and set $\bX = \bX^{\CL,t_0,\bzeta_0}$ and $\wt{\bX} = {\bX}^{\OL,t_0,\bzeta_0}$. Rewrite the dynamics of $\bX$ as 
   \[
        \d X_t^i = - \bigl( D_p H^i(\bX_t, v^i(t,\bX_t)) + E_t^i \bigr) \,\d t + \sqrt{2\sigma} \,\d W_t^i + \sqrt{2\sigma^0} \,\d W_t^i\,,
   \]
    with 
   \[
        E_t^i = D_p H^i(\bX_t, D_{i} u^i(t,\bX_t)) - D_p H^i(\bX_t, v^i(t,\bX_t))\,.
   \]
    Note that, by \Cref{thm.maindisp} and \Cref{lem.clol},
    \begin{equation} \label{errorboundolcl}
      |\bm E|^2 \lesssim \sigma^{-1} {\delta} \sum_{i} \delta^i\,.
    \end{equation}
    Next, we set $\Delta \bX = \bX - \wt{\bX}$, and we compute 
    \[
    \begin{split}
        \d |\Delta \bX_t|^2 &= 2 \sum_i \Bigl( \Delta X_t^i \cdot \bigl(D_p H^i(\wt{\bX}_t, v^i(t,\wt{\bX}_t)) - D_pH^i(\bX_t, v^i(t,\bX_t)) \big) - \Delta X_t^i \cdot E_t^i \Bigr) \d t
        \\
        &\lesssim \bigl( |\bm E_t|^2 + |\Delta\bX_t|^2 \bigr) \d t\,,
    \end{split}
    \]
    where we used \Cref{prop.olboundsdisp} and the Lipschitz continuity of $(D_p H^i)_{i = 1,\dots,N}$ (see \Cref{lem.DH.lip}) to obtain the inequality. Integrate, take expectations, and apply Gr\"{o}nwall's lemma and \eqref{errorboundolcl}, to get 
    \[
        \sup_{t \in [t_0,T]} \E\big[ |\Delta \bX_t|^2 \big]  \lesssim \sigma^{-1} {\delta} \sum_{i} \delta^i\,;
   \]
    then, plug this back into the equation of $|\Delta\bX_t|^2$ to deduce the desired estimate.
    \end{proof}

\begin{proof}[Proof of Theorem \ref{thm.olclLL}]
    Argue as in the proof of \Cref{thm.olcldisp}, but invoke \Cref{thm.mainll} and \Cref{lem.llclol} (instead of \Cref{thm.maindisp} and \Cref{lem.clol}, respectively) to have 
    \begin{equation*} %\label{errorboundolclll}
         \E \biggl[\,\int_{t_0}^T |\bm E|^2 dt \biggr] \lesssim e^{e^{C(1+\sigma^{-1})}} {\delta} \sum_{i} \delta^i\,.
    \end{equation*}
    The rest of the proof is the same.
\end{proof}

\subsection{Comparing open-loop and distributed equilibria}

We begin with some useful lemmas. In particular, the first one which explains how the Lipschitz bounds for the PDE system \eqref{pontpde} we have obtained implies a stability estimate for the FBSDE system \eqref{pontryagin}.

\begin{lemma} \label{lem.lipschitzstability}
    Assume that either Assumption \ref{assump.disp} or Assumption \ref{assump.LL} holds. Suppose that we are given processes $\hat{\bX}, \hat{\bY}, \hat{\bZ}$ satisfying the following FBSDE system on $[0,T]$:
    \begin{equation} \label{disterrorfbsde}
    \begin{dcases}
        \d\hat{X}_t^i = - \big( D_p H^i(\hat{\bX}_t, \hat{Y}_t^i) + E_t^{\bH, p, i} \big)\,\d t + \sqrt{2 \sigma} \,\d W_t^i + \sqrt{2 \sigma_0} \,\d W_t^0
        \\[-3pt]
        \d\hat{Y}_t^i = \bigl( D_{x^i} H^i(\hat{\bX}_t,\hat{Y}_t^i) + E_t^{\bH,\bx,i} \bigr) \d t + \sum_{j = 0}^N \hat{Z}_t^{i,j} \d W_t^j
       \\[-3pt]
       \hat{\bX}_{t_0} = \bzeta_0, \quad \hat{Y}_{T}^i = D_{i} G^i(\ov{\bX}_T) + E^{\bG,i}\,,
    \end{dcases}
\end{equation}
for some $t_0 \in [0,T]$, and some square-integrable random vector $\bzeta_0$ (with i.i.d.\ components), processes $E^{\bH,p,i}$ and $E^{\bH,\bx,i}$, and random variables $E^{\bG,i}$. Then, there is a dimension-free constant $C$ such that
    \begin{equation} \label{stabest.DX}
       \E\Big[\, \sup_{t \in [t_0,T]} \bigl|\hat{\bX}_t - \bX^{\OL,t_0,\bzeta_0}\bigr|^2 \,\Big] \leq C  \E\bigg[ |\bm E^{\bG}|^2 + \int_{t_0}^T \Big( |\bm E_t^{\bH,p}|^2 + |\bm E^{\bH,\bx}_t|^2 \Big) \,\d t \biggr]\,.
    \end{equation}
%     where
%     \begin{equation} \label{stabest.defEi}
% \mathcal E^i \defeq \E\bigg[ |E^{\bG,i}|^2 + \int_{t_0}^T \Big( |E_t^{\bH,p,i}|^2 + |E^{\bH,\bx,i}|^2 \Big) dt \biggr]\,.
% 	\end{equation}
   % Furthermore, if that there are constants $\kappa^i$ such that, for each $i = 1,\dots,N$,
  %  \begin{equation} \label{stabest.skewdec}
 %  \sum_{j \neq i} |D_j v^i|^2 \leq \kappa^i\,,
 %   \end{equation}
 %   then
 %   \begin{equation} \label{stabest.DXsupi}
%\E \Bigl[\, \sup_{t \in [t_0,T]} \bigl|\hat{X}^i_t - X_t^{\OL,i}\bigr|^2 \Bigr] \leq (T-t_0)\exp\Bigl\{Tc\Bigl(1+\sqrt M + \kappa^i \sum_{j \neq i} \kappa^j\Bigr)\Bigr\} \Bigl( \mathcal E^i + \kappa^i \sum_{1 \leq i \leq N} \mathcal E^i \,\Bigr)
%\end{equation}
%for each $i = 1,\dots,N$.
\end{lemma}

\begin{proof}
We recall that by \Cref{prop.olboundsdisp} or \Cref{prop.olboundslip}, there is a dimension-free constant $M$ such that 
\begin{equation*} %\label{stabest.lipass}
    |\bv(t,\bx) - \bv(t,\ov{\bx})|^2 \leq M |\bx - \ov{\bx}|^2 \qquad \forall\, t \in [0,T],\, \bx, \ov{\bx} \in (\R^d)^N.
    \end{equation*}
Set $\wt{Y}_t^i = v^i(t,\hat{\bX}_t)$, $\wt{Z}_t^{i,j} = \sqrt{2 \sigma} D_{j} v^i(t,\hat{\bX}_t)$, and $\wt{Z}_t^{i,0} = \sqrt{2\sigma_0} \sum_{j} D_{j} v^i(t,\hat{\bX}_t)$. Also set $\bX = \bX^{\ol,t_0,\bzeta_0}$, $\Delta X^i \defeq \hat X^i - X^i$ and $\Delta Y_t^i \defeq \hat{Y}_t^i - \wt{Y}_t^i$.

We compute 
%\marginpar{\tiny DF: Fixed notation in the proof; please check if it's correct. {Looks consistent to me.}}
\[
\begin{split}
    \d |\Delta \bY_t|^2 &= - 2\sum_{i} (\Delta Y_t^i)^\TT D_{x^i}H^i\Bigr|^{(\hat{\bX}_t,\hat{Y}_t^i)}_{(\hat{\bX}_t,\wt{Y}_t^i)} \, \d t  + 2 \sum_{i,j} (\Delta Y_t^i)^\TT D_{j} v^i(t,\hat{\bX}_t) D_pH^j\Bigr|^{(\hat{\bX}_t, \hat{Y}_t^j)}_{(\hat{\bX}_t, \wt{Y}_t^j)} \,\d t 
    \\    &\quad + 2\sum_{i} (\Delta Y_t^i)^\TT E_t^{\bH,\bx,i} \,\d t + \sum_{i,j} |\Delta Z_t^{i,j}|^2 \d t + \d M_t\,,
\end{split}
\]
with $M$ being a martingale; then we integrate from $t \in [t_0,T)$ to $T$, take expectations, and use the Lipschitz bound on $\bv$ to conclude that 
\[
   \E\big[ |\Delta \bY_t|^2 \big] \leq \E\bigg[ |\bm E^{\bG}|^2 + \int_{t}^T |\bm E^{\bH,\bx,}|^2  \bigg] + \bigl(1 + 2\norm{D_{px}\bH}_\infty + 2\norm{D_{pp}\bH}_\infty M\bigr) \E\bigg[ \int_t^T |\Delta \bY|^2 \bigg]\,.
\]
Now apply Gr\"{o}nwall's lemma to conclude that 
\begin{equation} \label{proof.olerr.DY}
    \sup_{t \in [t_0,T]} \E\big[ |\Delta \bY_t|^2 \big] \leq e^{TC' (1+ M) } \E\bigg[ |\bm E^{G}|^2 + \int_{t_0}^T |\bm E^{\bH,\bx}|^2 \bigg]\,.
\end{equation}
for some dimension-free $C'$. 
At this point, compute, exploiting the triangle inequality, 
\[
\begin{split}
    \d |\Delta \bX_t|^2  &= -2 \sum_i (\Delta X_t^i)^\TT D_pH^i\Bigr|^{(\hat{\bX}_t, \hat Y^i_t)}_{({\bX}_t, v^i(t,{\bX}_t))}\, \d t - 2\sum_{i} \big(\Delta X_t^i\big)^{\TT} E_t^{\bH,\bp,i} \,\d t
    \\
    &\lesssim |\Delta \bX_t| |\Delta \bY_t| dt + (1+M) |\Delta \bX_t|^2\d t +  |\Delta \bX_t| | \bm{E}_t^{\bH,\bp} |,
\end{split}
\]
with dimension-free implied constant; now use Young's inequality followed by Gr\"{o}nwall's lemma again, then apply \eqref{proof.olerr.DY} to get 
\[
\sup_{t \in [t_0,T]} \E\bigl[ |\Delta \bX_t |^2 \,\bigr] \lesssim  \E\bigg[ |\bm E^{\bG}|^2 + \int_{t_0}^T \Big( |\bm E_t^{\bH,p}|^2 + |\bm E^{\bH,\bx}_t|^2 \Big) \,\d t \biggr].
\]
Finally, plug this back into the equation for $|\Delta \bX_t|^2$ integrated over $t$ to deduce that \eqref{stabest.DX} holds.
\end{proof}

\begin{lemma} \label{lem.wibound}
    Let Assumption \ref{assump.disp} hold. Let $(w^i,m^i)_{i = 1,\dots,N}$ solve \eqref{distpde}. Then $\max_{i=1,\dots,N} \|D^2w^i\|_{\infty} $ has a dimension-free bound.
\end{lemma}

\begin{proof}
    Recall the definition \eqref{def.distl} of $\bbL^i$ and for notational simplicity define
    \begin{equation*}
        g^i\big(x, (m^j)_{j \neq i} \big) \defeq \int_{(\R^d)^{N-1}} G^i(\by^{-i}, x) \prod_{j \neq i} m_T^j(\d y^j)\,.
    \end{equation*}
    % so that in particular 
    % \begin{align*}
    %     \bbH^i\big(x,p, (m^j)_{j \neq i} \big) = \sup_{a \in \R^d} \Big\{ - \bbL^i\big(x,a,(m^j)_{j \neq i} \big) - a \cdot p \Big\}.
    % \end{align*}
    Under \Cref{assump.disp}, we have, for each $i$, and any $(m^j)_{j \neq i}$,
    \begin{equation} \label{littlefsemiconvex}
      D^2_{(x,a)} \bbL^i \geq \begin{pmatrix}
          - C_{\bL,\bx} I_{d} & 0 \\ 0 & C_{\bL,\ba} I_d
      \end{pmatrix}, \qquad D_{xx} g^i \geq - C_{\bG, \dis} I_{d}\,.
    \end{equation}
   Define $\bX$ to be the solution of
    \begin{equation*}
        \d X_t^i = - D_p \bbH^i \Big(X_t^i, Dw^i(t,X^i_t), (m_t^j)_{j \neq i}\Big) \,\d t + \sqrt{2 \sigma} \,\d W_t^i\, ,\qquad X_{t_0} = \zeta^i, \quad m_t^j = \cL(X_t^j);
    \end{equation*}
    then, fix $i \in \{1,\dots,N\}$, $s_0 \in [t_0,T]$, $x_0,\hat{x}_0 \in \R^d$, and define $X_t$ and $\hat{X}_t$ via
    \begin{equation*}
        \d X_t = - D_p \bbH^i\Big(X_t, Dw^i(t,X_t), (m_t^j)_{j \neq i} \Big)\,\d t + \sqrt{2 \sigma} \,\d W_t^i\,, \qquad {X}_{s_0} = {x}_0\,
    \end{equation*}
    and
     \begin{equation*}
        \d \hat{X}_t = - D_p \bbH^i\Big(\hat X_t, Dw^i(t,\hat X_t), (m_t^j)_{j \neq i} \Big)\,\d t + \sqrt{2 \sigma} \,\d W_t^i\,, \qquad \hat{X}_{s_0} = \hat{x}_0\,.
    \end{equation*}
    Set
    \begin{equation*}
        Y_t \defeq D w^i(t,X_t)\,, \qquad Z_t^{i} \defeq \sqrt{2\sigma}\, D^2 w^i(t,X_t)\,, \qquad \alpha_t \defeq - D_p \bbH^i\Big(X_t,Y_t, (m_t^j)_{j \neq i} \Big)\,,
        \end{equation*}
        and likewise for $\hat{Y}_t$, $\hat{Z}_t$ and $\hat{\alpha}_t$; recall that we have 
    \begin{equation*}
        \d Y_t = - D_x \bbH^i\Big(X_t, Y_t, (m_t^j)_{j \neq i} \Big)  \,\d t + Z_t \,\d W_t\,, 
    \end{equation*}
    and likewise for $\hat{Y}$ and $\hat{Z}$. Setting $\Delta Y_t \defeq Y_t - \hat{Y}_t$ and $\Delta{Z}_t \defeq Z_t - \hat{Z}_t$, we compute $\d(\Delta X_t \cdot \Delta Y_t)$, then we integrate from time $s_0$ to $T$, and, using \eqref{littlefsemiconvex} and \Cref{assump.disp}, we find that
    \begin{equation*}
        |Dw^i(s_0,x_0) - Dw^i(s_0,\hat{x}_0)|^2 = \E\big[ |\Delta Y_{t_0}|^2 \big] \lesssim |x_0 - \hat{x}_0|^2\,,
    \end{equation*}
    with dimension-free implied constant.
    We omit the details because the computation is almost identical to the one that appears in the proof of \Cref{prop.olboundsdisp}. This completes the proof.
\end{proof}

\begin{lemma} \label{lem.wibound.LL}
    Let \Cref{assump.LL} hold. Let $(w^i,m^i)_{i = 1,\dots,N}$ solve \eqref{distpde}. Then we have 
    \begin{equation*}
        \max_{i = 1,\dots,N} \|D^2 w^i\|_{\infty} \lesssim \sigma^{-1}\,,
    \end{equation*}
    with dimension-free implied constant.
\end{lemma}

\begin{proof} 
    This is a standard estimate for the equation satisfied by $w^i$, which can be checked by, e.g., the nonlinear adjoint method. We omit the proof, and refer to the proof of \cite[Proposition~3.1]{CDJM} for a similar argument in a related setting. 
\end{proof}

\begin{lemma} \label{lem.poincaredisp} 
    Let \Cref{assump.disp} hold, and let $\bzeta_0 = (\zeta^1_0,\dots,\zeta^N_0)$ be an $\scrF_{t_0}$-measurable, square integrable random vector taking values in $(\R^d)^N$, with $(\zeta^i_0)_{i = 1,\dots,N}$ independent. Suppose that for each $i$, $\cL(\zeta^i_0)$ satisfies a Poincar\'e inequality with constant $C_{P,0}^i$, and set
    \begin{equation} \label{CP0def}
        C_{P,0} \defeq \max_{i=1,\dots,N} C_{P,0}^i\,.
    \end{equation}
    Then there is a dimension-free constant $C_P$ such that for each $t > t_0$, $\cL({\bX}_t^{\Dis,t_0,\bzeta_0})$ satisfies a Poincar\'e inequality with constant $(C_{P,0} + \sigma)C_P$, i.e.\ we have
    \begin{equation*}
        \mathrm{Var}\big(g({\bX}_t^{\Dis,t_0,\bzeta_0}) \big) \leq (C_{P,0} + \sigma)C_P\, \E\Bigl[ \bigl|D g({\bX}_t^{\Dis,t_0,\bzeta_0})\bigr|^2 \Bigr]
    \end{equation*}
    for each function Lipschitz function $g \in C^1((\R^d)^N)$. 
\end{lemma}

\begin{proof}
By a simple extension of \cite[Theorem 1.8]{cattiaux-guillin} (see also \cite[Lemma~4.13]{JacksonLacker}),\footnote{Note that \cite{JacksonLacker} requires $g$ to be Lipschitz and bounded, but using the integrability of $\bX_t^{\Dis,t_0,\bzeta_0}$ there is no issue extending to Lipschitz $g$ via a truncation procedure.} there is a dimension-free constant $C$ such that $m_t^i \defeq \cL({\bX}_t^{\Dis,t_0,\bzeta_0})$ satisfies a Poincar\'e inequality with constant $(C_{P,0} + \sigma)C e^{C \lip(w^i)}$,
where $\lip(w^i)$ indicates the Lipschitz constant of $w^i$. By \Cref{lem.wibound}, $\lip(w^i)$ is dimension-free. 
\end{proof}

\begin{lemma} \label{lem.poincareLL}
    Let \Cref{assump.LL} hold, and let $\bzeta_0$ and be as in \Cref{lem.poincaredisp}. Suppose that for each $i$, $\cL(\zeta^i_0)$ satisfies a Poincar\'e inequality with constant $C_{P,0}^i$, and define $C_{P,0}$ as in \eqref{CP0def}.
    Then there is a dimension-free constant $C_P$ such that for each $t > t_0$, $\cL({\bX}_t^{\Dis,t_0,\bzeta_0})$ satisfies a Poincar\'e inequality with constant $(C_{P,0} + \sigma) e^{C_P(1+\sigma^{-1})}$; i.e.\ we have 
    \begin{equation*}
        \mathrm{Var}\big(g({\bX}_t^{\Dis,t_0,\bzeta_0}) \big) \leq (C_{P,0} + \sigma)e^{C_P(1+\sigma^{-1})} \E\Bigl[ \bigl|D g({\bX}_t^{\Dis,t_0,\bzeta_0})\bigr|^2 \Bigr]
    \end{equation*}
    for each Lipschitz function $g \in C^1((\R^d)^N)$. 
\end{lemma}

\begin{proof}
    Argue as in the proof of \Cref{lem.poincaredisp}, but invoke \Cref{lem.wibound.LL} in place of \Cref{lem.wibound}.
\end{proof}

We are almost ready to prove \Cref{thm.oldistdisp,thm.oldistLL}. For \Cref{thm.oldistdisp}, we first need a lemma which will help us compare the standard Hamiltonian $H^i$ with the distributed Hamiltionian $\bbH^i$, which is trivial in the separable case.

\begin{lemma} \label{lem.bbHi.1} %\marginpar{\tiny DF: Gathered all three estimates in one lemma and optimised the proof.}
    Let Assumption \ref{assump.disp} hold. Let $\bX = (X^1,\dots,X^N)$ and $\bY = (Y^1,\dots,Y^N)$ be square-integrable random vectors, with $X^i$ and $Y^i$ taking values in $\R^d$, and suppose that $Y^i$ is $X^i$-measurable and $m^i \coloneqq \cL(X^i)$ satisfies a Poincar\'e inequality with constant $c_P$.
    %\marginpar{\tiny DF: Renamed Poincar\'e constant $c_P$ because this is different from that in the two previous lemmas, right? JJ: Sounds good to me.} 
    Then, we have 
    \begin{gather} \label{H.Hdist.comp}
      0 \leq  \E\Big[ H^i\big(\bX, Y^i\big) \Big] - \E\Big[ \bbH^i\big(X^i,Y^i, (m^j)_{j \neq i} \big) \Big] \lesssim c_P \delta^i,
    \\ \label{Delta.DpH}
     \E\Big[ \Big|D_p H^i\big(\bX, Y^i\big) - D_p \bbH^i\big(X^i,Y^i,(m^j)_{j \neq i} \big) \Big|^2 \Big] \lesssim c_P \delta^i,
    \\ \label{Delta.DxH}
     \E\Big[  \Big|D_{x^i} H^i\big(\bX, Y^i\big) - D_x \bbH^i\big(X^i,Y^i,(m^j)_{j \neq i} \big) \Big|^2 \Big] \lesssim c_P(\delta^i + \kappa^i),
    \end{gather}
    with dimension-free implied constant. 
\end{lemma}

\begin{proof}
    We have %\marginpar{\tiny DF: Doesn't $\alpha$ also need to be measurable wrt $\bX$? JJ: I think the supremum will be unchanged if we allow for extra randomness, by the convexity of $L^i$ in $a$, but it's fine with me to add $\bX$-measurability}
    \begin{equation} \label{Hi.comp}
        \E\Big[ H^i\big(\bX, Y^i\big) \Big] = \E\Big[ \sup_{a \in \R^d} \Bigl\{ - Y^i \cdot a - L^i\big(\bX, a \big)  \Bigr\} \Big] = \sup_{\alpha} \E\Bigl[ - Y^i \cdot \alpha- L^i\big(\bX, \alpha \big)  \Bigr], 
    \end{equation}
    where the supremum is taken over all square-integrable $\R^d$-valued random vectors $\alpha$ which are measurable with respect to $\bX$, while
    \begin{equation} \label{Hdisti.comp}
    \begin{split}
        \E\Big[ \bbH^i\big(X^i,Y^i,(m^j)_{j \neq i} \big) \Big] &= \E\Big[ \sup_{a \in \R^d} \Bigl\{-a \cdot Y^i - \bbL^i\big(X^i ,a, (m^j)_{j \neq i} \big) \Bigr\}\Big] \\
        &= \sup_{\alpha \in \sigma(X^i)} \E\Bigl[ - Y^i \cdot \alpha- L^i\big(\bX, \alpha \big)  \Bigr],
        \end{split}
    \end{equation}
    where the supremum is taken over square-integrable $\R^d$-valued random vectors $\alpha$ which are measurable with respect to $X^i$; hence, the lower bound in \eqref{H.Hdist.comp} follows immediately. 
    Now let $\hat{\alpha} = - D_p H^i\big(\bX, Y^i\big)$ be the optimizer in \eqref{Hi.comp}; by using $\E[\hat{\alpha} \mid X^i]$ as a competitor in the optimization problem in \eqref{Hdisti.comp}, and the fact that $Y^i = - D_aL^i(\bX,\hat \alpha)$, we find
    \begin{equation*}\begin{split}
        \E\Big[ H^i\big(\bX, Y^i\big) \Big] - \E\Big[ \bbH^i\big(X^i,Y^i, (m^j)_{j \neq i} \big) \Big] & \leq \E\Big[ L^i(\bX, \cdot )\bigr|^{\E[\hat{\alpha} \mid X^i]}_{\hat\alpha} - Y^i \cdot \bigl( \E[\hat{\alpha} \mid X^i] - \hat{\alpha}\bigr) \Big]
        \\
        & \leq %\E\Big[ \Big(Y^i + D_a L^i\big(\bX, \hat{\alpha}\big) \Big) \cdot \big(\E[\hat \alpha | X^i] - \hat \alpha \big) +
        {C_{D\bL, \lip}} \bigl|\hat{\alpha} - \E[\hat \alpha \mid X^i] \bigr|^2  \Big]
        \\
        %& = {C_{D\bL, \lip}} \E\Big[ |\hat{\alpha} - \E[\hat \alpha | X^i] |^2 \Big] =  {C_{D\bL, \lip}}  \E\Big[ |D_pH^i(\bX, Y^i) - \E[D_pH^i(\bX, Y^i) | X^i] |^2 \Big] 
        %\\
        & \leq {C_{D\bL, \lip}} c_P \E\Big[  \sum_{j \neq i} \big| D_{x^j p} H^i(\bX, Y^i) \big|^2 \Big] \lesssim c_P \delta^i, 
    \end{split}\end{equation*}
    where the last inequality uses the formula \eqref{dpxH.formula}. This proves the upper bound in \eqref{H.Hdist.comp}.

    Define now $\ov{\alpha} = - D_p \bbH^i\big(X^i,Y^i,(m^j)_{j \neq i}\big)$, which optimizes \eqref{Hdisti.comp}
    % \begin{align*}
    %     \hat{\alpha}^i = - D_pH^i\big(\bX,Y^i\big), \quad \ov{\alpha}^i = - D_p \bbH^i\big(X^i,Y^i,(m^j)_{j \neq i}\big). 
    % \end{align*}
    % Notice that 
    % \begin{align*}
    %     \E\Big[ H^i\big(\bX, Y^i\big) \Big] = \E\Big[ - L^i\big(\bX, \hat \alpha \big) - \hat \alpha^i \cdot Y^i \Big], \quad \E\Big[ \bbH^i\big(X^i, Y^i, (m^j)_{j \neq i} \big) \Big] = \E\Big[ - L^i\big(\bX, \ov \alpha^i \big) - \ov \alpha \cdot Y^i \Big].
    % \end{align*}
    By the strict convexity of $a \mapsto L^i(\bx,a)$ and \eqref{H.Hdist.comp} (and again the fact that $Y^i = - D_a L^i(\bX, \hat \alpha)$), we deduce 
    \begin{align*}
        c_P  \delta^i &\gtrsim \E\Big[ H^i\big(\bX, Y^i\big) - \bbH^i\big(X^i, Y^i, (m^j)_{j \neq i} \big)  \Big]
        \\
         &= \E\Big[ L^i\big(\bX, \ov \alpha^i \big) - L^i\big(\bX, \hat \alpha^i \big) + \big(\ov \alpha^i - \hat \alpha^i \big) \cdot Y^i   \Big]
         %\\
         \geq %\E\Big[ \frac{C_{\bL,\ba}}{2} \big|\ov \alpha^i - \hat \alpha^i \big|^2 + \Big( D_a L^i\big(\bX, \hat \alpha^i\big) + Y^i \Big) \cdot \big(\ov \alpha^i - \hat \alpha^i\big) \Big]
         %\\
         %&=  
         \E\Big[ \frac{C_{\bL,\ba}}{2} \big|\ov \alpha^i - \hat \alpha^i \big|^2\Big], 
    \end{align*}
    which gives \eqref{Delta.DpH}. 
    Finally, using that 
    \[
    D_{x^i} H^i(\bX, Y^i) = - D_{x^i} L^i(\bX, \hat \alpha), \qquad D_x \bbH^i(X^i,Y^i,(m^j)_{j \neq i} ) = - \E\bigl[ D_{x^i} L^i(\bX, \ov \alpha) \mid X^i \bigr],
    \]
    % \begin{align*}
    %     &D_{x^i} H^i\big(\bX, Y^i\big) = - D_{x^i} L^i\big(\bX, \hat \alpha^i \big), 
    %     \\
    %     &D_x \bbH^i\big(X^i,Y^i,(m^j)_{j \neq i} \big) = - \int_{(\R^d)^{N-1}} D_{x^i} L^i \Big( \big(\by^{-i},X^i\big), \ov \alpha^i \Big) \prod_{j \neq i} m^j(dy^j) = - \E\Big[ D_{x^i} L^i\big(\bX, \ov \alpha^i \big) | X^i \Big].
    % \end{align*}
    % Therefore, 
    we find 
    \begin{equation*}\begin{split}
        &\E\Big[ \Big| D_{x^i} H^i\big(\bX, Y^i\big) - D_{x} \bbH^i\big(X^i,Y^i,(m^j)_{j \neq i} \big) \Big|^2 \Big] 
        \\
        &=  \E\Big[ \Big| D_{x^i} L^i\big(\bX, \hat \alpha^i \big) -\E\big[ D_{x^i} L^i\big(\bX, \ov \alpha^i \big) \mid X^i \big] \Big|^2 \Big]
        \\
        &\leq 2  \E\Big[ \Big| D_{x^i} L^i\big(\bX, \hat \alpha^i \big) -D_{x^i} L^i\big(\bX, \ov \alpha^i \big) \Big|^2 \Big] + 2\E\Big[ \Big| D_{x^i} L^i\big(\bX, \ov \alpha^i \big) -\E\big[ D_{x^i} L^i\big(\bX, \ov \alpha^i \big) \mid X^i \big] \Big|^2 \Big]
        \\
        &\leq 2 \| D_{a} D_{x^i} L^i \|_{\infty} \E\big[ |\ov \alpha - \hat \alpha|^2 \big] + 2 c_P \kappa^i \lesssim c_P (\delta^i + \kappa^i),
    \end{split}\end{equation*}
    so \eqref{Delta.DxH} follows from \eqref{Delta.DpH}.
\end{proof}
\color{black}
\begin{proof}[Proof of \Cref{thm.oldistdisp}]
Set $\ov{\bX} = {\bX}^{\Dis,t_0,\bzeta_0}$, $\wt{\bX} = {\bX}^{\OL,t_0,\bzeta_0}$. Then we can find $\ov{\bY}$, $\ov{\bZ}$ such that the triple $(\ov{\bX}, \ov{\bY}, \ov{\bZ})$ satisfies \eqref{distfbsde}. By adding and subtracting, we see that $(\ov{\bX}, \ov{\bY}, \ov{\bZ})$ satisfies \eqref{disterrorfbsde}, with
\begin{equation*}\begin{split}
    E_t^{\bH,p,i} &= D_p \bbH^i\big( \ov X_t^i, \ov Y_t^i, (m_t^j)_{j \neq i} \big) - D_{p} H^i\big(\ov \bX_t, Y_t^i\big),
    \\
    E_t^{\bH,\bx,i} &= D_x \bbH^i\big( \ov X_t^i, \ov Y_t^i, (m_t^j)_{j \neq i} \big) - D_{x^i} H^i\big(\ov \bX_t, Y_t^i\big), 
    \\
    E^{\bG,i} &= \E\big[ D_{i} G^i(\ov{\bX}_T) \mid \ov{X}_T^i \big] - D_{i} G^i(\ov{\bX_T})\,
\end{split}\end{equation*}
where we used the notation $m_t^j = \cL(\ov X_t^j)$. 
By applying Lemma \ref{lem.bbHi.1}, we obtain
\begin{equation*}
 \E\big[ \big| E_t^{\bH,\bx,i} \big|^2 \big] \lesssim c_P \big(\delta^i + \kappa^i\big), \qquad  \E\big[ \big| E_t^{\bH,p,i} \big|^2 \big] \lesssim c_P \delta^i, \qquad  \E\big[ \bigl|E^{\bG,i}\bigr|^2 \big] \lesssim c_P \kappa^i, 
\end{equation*}
all with dimension-free implied constants, and where the Poincar\'e constant $c_P$ is controlled by using Lemma \ref{lem.poincaredisp}.
The result now follows from combining Lemma \ref{lem.lipschitzstability} with Proposition \ref{prop.olboundsdisp}.
\end{proof}

\begin{proof}[Proof of \Cref{thm.oldistLL}]
   The proof is the same as that of \Cref{thm.oldistdisp} above, except that \Cref{lem.poincareLL} takes the place of \Cref{lem.poincaredisp}, \Cref{prop.olboundslip} takes the place of \Cref{prop.olboundsdisp}.
\end{proof}

\begin{remark} \label{rmk.distcommonnoise}
    When there is common noise, ``distributed" controls would mean $\bbF^0$-adapted random fields of the form $(\alpha_t(x))_{t \in [t_0, T]}$, and the PDE \eqref{distdynamics} would be replaced by a stochastic PDE system analogous to the stochastic MFG system which arises in the theory of MFGs with common noise. Some of our analysis would still apply in this case, but it is not clear how to generalize \Cref{lem.poincaredisp,lem.poincareLL}, so we are unable to generalize the bounds in \Cref{thm.oldistdisp,thm.oldistLL} to the case when there is common noise. This is why we discuss the distributed formulation only when there is no common noise.
\end{remark}

\section{Application to universality of MFG equilibria}
\label{sec.universality}
%\marginpar{\tiny DF: Added some parts in red in this section.}

In this section, we consider a sequence of games built in the following way. We work in the separable regime \eqref{Lagrangian.sep} in order to present the results for displacement and Lasry--Lions monotonicity in a unified way; {nevertheless, we point out that all results are still valid in the non-separable displacement semi-monotone regime, with very similar proofs.}

We assume that we are given functions $L_0 \colon \R^d \times \R^d \to \R$ and $\cF, \cG : \R^d \times \cP_2(\R^d) \to \R$, and, for each $N \in \N$, we are given a collection of non-negative weights $(w^N_{ij})_{i,j =1,\dots,N}$ such that $\sum_j w^N_{ij} = 1$ and $w^N_{ii} = 0$ for all $i$.
We define, for each $N \in \N$ and $i = 1,\dots,N$, functions $F^{N,i}, G^{N,i} \colon (\R^d)^N \to \R$ via
\begin{equation*} %\label{fnidef}
    F^{N,i}(\bx) \defeq \cF(x^i,m_{\bx,w}^{N,i})\,, \qquad G^{N,i}(\bx) \defeq \cG(x^i,m_{\bx,w}^{N,i})\,,
\end{equation*}
where $m_{\bx,w}^{N,i}$ is the weighted empirical measure defined by
\[
m_{\bx,w}^{N,i} \defeq \sum_{j=1}^N w^N_{ij} \delta_{x^j}\,.
\]

\begin{remark}
If we are given an undirected graph $\Gamma^N = (V^N,E^N)$, where $V^N = \{1,\dots,N\}$, without self-loops (i.e., $(i,i) \notin E^N$), and $w^N_{ij} = (\deg_N i)^{-1} \1_{j \sim_N i}$, with $\deg_N i$ being the degree of the vertex $i$ and $j \sim_N i$ meaning that $(i,j) \in E^N$, then $m^{N,i}_{\bx,w}$ is the neighborhood empirical measure
\[
    m_{\bx}^{N,i} = \frac{1}{\deg_N i} \sum_{j \sim_N i} \delta_{x^j}\,.
\]
Moreover, if $\Gamma^N$ is totally connected (i.e., $E^N = V^N \times V^N \setminus \{(i,i):\, i \in V^N\}$), then $\deg_N i = N-1$ for all $i$ and the above measure coincides with the standard empirical measure considered in mean field games.
\end{remark}

%
%and for each $N$ we consider the game with
%\begin{align*}
%    L^i = L, \quad F^{N,i}(\bx) = \cF(x^i,m_{\bx}^{N,-i}), \quad G^{N,i}(\bx) = \cG(x^i,m_{\bx}^{N,-i}), \quad m_{\bx}^{N,-i} \coloneqq \frac{1}{N-1} \sum_{j \neq i} \delta_{x^j}.
%\end{align*}
We will assume that the game is played from i.i.d.\ initial conditions. Thus we fix a vector $\bm\zeta=(\zeta^i_0)_{i = 1,\dots,N}$ of i.i.d., $\scrF_0$-measurable random variables with common law $m_0 \in \cP_p(\R^d)$ for some $p > 4$. For the sake of a clearer notation, and also to make explicit the dependence on $N \in \N$, in this section we will denote by $\bX^N = (X^{N,1},\dots,X^{N,N})$ the closed-loop equilibrium trajectories, and by $\wt{\bX}^N = (\wt{X}^{N,1},\dots,\wt{X}^{N,N})$ the open-loop equilibrium trajectories (started from $(0,\bzeta_0)$). We note that $\bX^N$ evolves, on $[0,T]$, according to 
\begin{equation*} %\label{dynamicsclnemf}
    \d X_t^{N,i} = - D_p H_0(X_t^{N,i}, D_{i} u^{N,i}(t,\bX_t)) \,\d t + \sqrt{2\sigma} \,\d W_t^i + \sqrt{2\sigma_0} \,\d W_t^0\,, \qquad X_{0}^{N,i} = \zeta^i_0\,,
\end{equation*}
({where $H_0$ is obtained from $L_0$ via \eqref{Hamiltonian.sep} with $L^i_0 = L_0$}, and $\sigma$ and $\sigma_0$ are independent of $N$) where $u^{N,i}$ satisfy the Nash system \eqref{nash}.
%\begin{equation} \label{nashmf} 
%    \begin{cases} \ds - \partial_t u^{N,i} - \sigma \sum_{j = 1}^N \Delta_{x^j} u^{N,i} - \sigma_0 \sum_{j,k = 1}^N \tr\big( D_{x^j x^k} u^{N,i} \big) + H(x^i, D_{x^i} u^{N,i}) 
%    \\ \ds 
%    \qquad + \sum_{j \neq i} D_{p} H(x^j, D_{x^j} u^{N,j}) D_{x^j} u^{N,i} = F^{N,i}(\bx), \quad (t,\bx) \in [0,T] \times (\R^d)^N, 
%    \\ \ds 
%    u^{N,i}(T,\bx) = G^{N,i}(\bx), \quad \bx \in (\R^d)^N. 
%    \end{cases} 
%\end{equation}
The open-loop trajectories can be described by the Pontryagin FBSDE system \eqref{pontryagin},
%\begin{align}
%    \begin{cases}
%        \ds d\wt{X}_t^{N,i} = - D_p H(\wt{X}_t^{N,i}, \wt{Y}_t^{N,i}) dt + \sqrt{2\sigma} dW_t^i + \sqrt{2\sigma_0} dW_t^0,
%        \\
%        \ds d \wt{Y}_t^{N,i} = - \Big( D_x L(\wt{X}_t^{N,i}, - D_p H(\wt{X}_t^{N,i}, \wt{Y}_t^{N,i}) ) + D_{x^i}F^{N,i}(\wt{\bX}_t^N) \Big) dt + \sum_{j = 0}^N \wt{Z}_t^{N,i,j} dW_t^j
 %       \\
 %       \ds \wt{X}_0^{N,i} = \xi^i, \quad \wt{Y}_T^{N,i} = D_{x^i}G^{N,i}(\wt{\bX}_T^N), 
 %   \end{cases}
%\end{align}
or, equivalently, by the dynamics 
\[
    \d \wt{X}_t^{N,i} = - D_p H_0(\wt{X}_t^{N,i}, v^{N,i}(t,\wt{\bX}_t^{N,i})) \,\d t + \sqrt{2\sigma} \,\d W_t^i + \sqrt{2\sigma_0} \,\d W_t^0\,, \qquad \wt{X}_0^{N,i} = \zeta^i_0\,,
\]
where $v^{N,i}$ satisfies the Pontryagin PDE system \eqref{pontpde}.
%\begin{align} \label{pontpdeuniv} 
%    \begin{cases}
 %       \ds - \partial_t v^{N,i} - \sigma \sum_{j = 1}^N \Delta_{x^j} v^{N,i} - \sigma_0 \sum_{j,k = 1}^N \tr\big( D_{x^jx^k} v^{N,i} \big) + \sum_{j = 1}^N  D_{x^j} v^{N,i} D_pH^j(x^j, v^{N,i})
%        \\
%        \ds \qquad + D_{x} H^i(x^i, v^{N,i}) = D_{x^i} F^{N,i}, \quad (t,\bx) \in [0,T] \times (\R^d)^N, 
%        \vspace{.2cm} \\
%        \ds v^{N,i}(T,\bx) = D_{x^i} G^{N,i}(\bx), \quad \bx \in (\R^d)^N.
%    \end{cases}
%\end{align}

Our first goal is to show that these equilibria, despite the heterogeneous setting, are both comparable to the ones arising from a limiting mean field game as $N \to \infty$, at least if the graphs $\Gamma^N$ are dense enough. To do so, we start by fixing a random vector $\zeta$ with $\zeta \sim m_0$; in the mean field game of interest, an $\bbF^0$-adapted, continuous $\cP_2(\R^d)$-valued process $(m_t)_{t \in [0,T]}$ representing the mean field is fixed, and a representative player chooses a control $\alpha$, which determines the state process $X = X^\alpha$ via the equation 
\begin{equation*} %\label{mfdynamics}
    \d X_t = \alpha_t\, \d t + \sqrt{2 \sigma} \,\d W_t + \sqrt{2 \sigma_0}\,\d W_t^0\,, \qquad X_{0}= \zeta\,.
\end{equation*}
The player's goal is to minimize the cost function 
\begin{equation*} %\label{jmdef}
    J_m(\alpha) = \E\bigg[\, \int_{0}^T \!\Big( L(X_t^\alpha, \alpha_t) + \cF(X_t^\alpha, m_t) \Big)\,\d t + \cG(X_T^\alpha,m_T) \bigg]\,.
\end{equation*}

\begin{definition}
    A \emph{mean field equilibrium} (started from $(0,m_0)$) is a continuous $\bbF^0$-adapted process $(m_t)_{t \in [0,T]}$ taking values in $\cP_2(\R^d)$ such that for some minimizer $\alpha$ of $J_m$, the corresponding state $X$ satisfies $m_t = \sL(X_t \mid \scrF_t^0)$.
\end{definition}
By the stochastic maximum principle, mean field equilibria are characterized by the following McKean--Vlasov FBSDE: 
\begin{equation} \label{meanfieldpont}
    \begin{dcases}
         \d X_t = - D_p H_0(X_t, Y_t) \,\d t + \sqrt{2\sigma} \,\d W_t + \sqrt{2 \sigma_0} \,\d W_t^0
        \\
         \d Y_t = \bigl( D_x H_0(X_t,Y_t) - D_x \cF(X_t, \cL(X_t \mid \scrF_t^0)) \bigr) \,\d t + Z_t \,\d W_t + Z_t^0 \,\d W_t^0
        \\
        X_0 = \zeta\,, \quad Y_T = D_x \cG(X_T, \cL(X_T \mid \scrF_T^0))\,,
    \end{dcases}
\end{equation}
where $H$ is defined from $L$ as in \eqref{intro.defHi}.
Our convergence results will be stated in terms of a sequence of (conditionally) i.i.d.\ copies of the solution $(X,Y,Z)$ to \eqref{meanfieldpont}, denoted $(\ov{X}^i,\ov{Y}^i,\ov{Z}^i)$, and obtained by taking $\bm\zeta \sim m_0^{\otimes N}$ and solving, for each $i$, the system 
\begin{equation} \label{meanfieldpointiid}
    \begin{dcases}
        \d\ov{X}^i_t = - D_p H_0(\ov{X}^i_t, \ov{Y}^i_t) \,\d t + \sqrt{2\sigma} \,\d W_t^i + \sqrt{2 \sigma_0} \,\d W_t^0
        \\
        \d\ov{Y}^i_t = \bigl( D_x H_0(\ov{X}^i_t,\ov{Y}^i_t) - D_x \cF(\ov{X}^i_t, \cL(\ov{X}^i_t \mid \scrF_t^0)) \bigr) \,\d t +\ov{Z}^i_t \,\d W_t^i +\ov{Z}^{i,0}_t \d W_t^0
        \\
        \ov{X}^i_0 = \zeta^i_0\,, \quad \ov{Y}_T^i = D_x \cG(\ov{X}_T^i, \cL(\ov{X}_T^i \mid \scrF_T^0))\,.
    \end{dcases}
\end{equation}

\begin{remark}
    Like for closed-loop, open-loop and distributed equilibria, we also implicitly assume that under our standing assumptions a mean field equilibrium exists -- in particular that we have access to a strong solution to \eqref{meanfieldpointiid}. We assume that $m_t$ has bounded $p$-th moments, for some $p > 4$; this is expected if $m_0 \in \cP_p(\R^d)$. The restriction $p>4$ could be weakened, at the price of more complicated exponents in the convergence results.
\end{remark}

In order to be in a setting where our estimates proved in the previous sections apply, we revisit the main \Cref{assump.disp,assump.LL} as follows. 

\begin{assumption}[Displacement semi-monotonicity and regularity, revised] \label{assump.uniformdisp}
{$L_0$ satisfies the regularity requirements in \Cref{assump.disp} (with $L_0$ in place of $L^i$), and there is $C_L >0$ such that
\begin{equation} \label{disp.L0}
    D^2_{(x,a)} L_0 \geq C_L\begin{pmatrix}
        0 & 0 \\ 0 & I_d,
    \end{pmatrix}
\end{equation}}
The functions $\cF$ and $D_x \cF$ are $C^1$, with $D_m \cF$, $D_m D_x \cF$ and $D_{xx} \cF$ bounded (and likewise for $\cG$).
Moreover, there are non-negative constants $C_{\bF, \dis}$ and $C_{\bG, \dis}$ such that we have 
    \begin{equation}\label{disp2}
        \sum_{i=1}^N D_x\cF\Bigr|^{(x^i,m_{\bx,w}^{N,i})}_{(\bar x^i,m_{\ov{\bx},w}^{N,i})} \cdot (x^i - \ov{x}^i) \geq - C_{\bF,\dis} |\bx - \ov{\bx}|^2
    \end{equation}
    for each $N \in \N$ and $\bx,\ov{\bx} \in (\R^d)^N$, and likewise for $\cG$. Finally, these constants satisfy condition~\eqref{smallness'} of \Cref{assump.disp} {with $C_{\bL,\ba} = C_L$ and $C_{\bL,\bx}=C_{\bF,\dis}$.}
\end{assumption}

\begin{remark}\label{assump.uniformdisploc}
 By means of the mean value theorem, condition \eqref{disp2} is equivalent to
\begin{equation*}%\label{disp2}
\sum_{i,j=1}^N (\xi^i)^\TT \bigl( D_{xx} \cF(x^i,m_{\bx,w}^{N,i})\1_{j=i} + w^N_{ij} D_mD_x\cF(x^i,m^{N,i}_{\bx,w},x^j)\bigr)\, \xi^j \geq -C_{\bF,\dis}|\bxi|^2
\end{equation*}
    for each $N \in \N$ and $\bx,\bxi \in (\R^d)^N$.
\end{remark}

\begin{assumption}[Lasry--Lions semi-monotonicity and regularity, revised]  \label{assump.uniformLL}
 $H_0$ satisfies the requirements in \Cref{assump.LL} (with $H_0$ in place of $H_0^i$).
 $\cF$ and $D_x \cF$ are $C^1$, with $D_x \cF$, $D_m \cF$, $D_m D_x \cF$, and $D_{xx} \cF$ bounded (and likewise for $\cG$). Moreover, there are constants $C_{\bF,\lm}$ and $C_{\bG,\lm}$ such that we have
\begin{equation}\label{LL2}
   \sum_{1 \leq i,j \leq N} w^N_{ij} \,(\xi^i)^\TT D_mD_x\cF(x^i,m^{N,i}_{\bx,w},x^j)\, \xi^j \geq -C_{\bF,\lm}|\bxi|^2
\end{equation}
for each $N \in \N$ and $\bx , \bm \xi \in (\R^d)^N$, and likewise for $\cG$. Finally, these constants satisfy condition \eqref{smallness2} from \Cref{thm.mainll}.
\end{assumption}

We have the following result. Recall that $|\bm w^N_{i\,\cdot}|_p$ (with $p \geq 1$) denotes the $\ell^p$-norm of the vector $(w^N_{ij})_{j=1,\dots,N} \in \R^N$, and likewise for $|\bm w^N_{\cdot\,j}|_p$, while $|w^N|_{\mathrm{Fr}} = \tr((w^N)^\TT w^N)^{\frac12}$ is the Frobenius norm of the matrix $(w^N_{ij})_{i,j=1,\dots,N} \in \R^{N\times N}$.

\begin{theorem}
\label{thm.universality}
    Suppose that either \Cref{assump.uniformdisp} or \Cref{assump.uniformLL} holds. Assume that
    \begin{gather} \label{wIrrcond}
    \limsup_{N \to \infty} \max_{1 \leq j \leq N} |\bm w^N_{\cdot\,j}|_1 < \infty\,,
    \\ \label{degreedivergence'}
     |w^N|_{\mathrm{Fr}} \max_{i=1,\dots,N} |\bm w^N_{i\,\cdot}|_2  \ \xrightarrow{N \to \infty}\  0\,.
    \end{gather}
    Then we have
    \begin{gather}\label{univest2}
        \E\Big[ \sup_{t \in [0,T]} |\ov{\bX}_t^{N} - \wt{\bX}_t^{N}|^2 \Big] \lesssim  \sum_{i = 1}^N \rho_d\big( |\bm w_{i\cdot}^N|^{-2} \big). 
    \\ \label{univest1}
        \E\Big[ \sup_{t \in [0,T]} |\bX_t^{N} - \wt{\bX}_t^{N}|^2 \Big] \lesssim  |w^N|_{\mathrm{Fr}}^2 \max_{i=1,\dots,N} |\bm w^N_{i\,\cdot}|_2^2, 
    \end{gather}
     with implied constants independent of $N$ and with $\rho_d$ as in \eqref{est.mf.intro}. As a consequence, 
     \begin{align*}
         &\E\Big[ \sup_{t \in [0,T]} \bd_2^2\big(m_{\wt{\bX}_t^N}^N, m_t\big) \Big] \lesssim \frac{1}{N} \sum_{i = 1}^N \rho_d\big( |\bm w_{i\cdot}^N|^{-2} \big) + \rho_d(N), 
         \\
         &\E\Big[ \sup_{t \in [0,T]} \bd_2^2\big(m_{{\bX}_t^N}^N, m_t\big) \Big] \lesssim \frac{1}{N}  |w^N|_{\mathrm{Fr}}^2 \max_{i=1,\dots,N} |\bm w^N_{i\,\cdot}|_2^2 + \frac{1}{N} \sum_{i = 1}^N \rho_d\big( |\bm w_{i\cdot}^N|^{-2} \big) + \rho_d(N), 
     \end{align*}
     where $m_t = \cL\big(\ov{X}_t^i | \cF_t^0 \big)$ denotes the Nash equilibria of the corresponding MFG.
    %and, for any partition $(I_k)_{k=1,\dots,n}$ of $\{1,\dots,N\}$ and each $k = 1,\dots,n$ such that $\wt X^{N,i} \overset{\mathrm{d}}{=} \wt X^{N,j}$ whenever $i,j \in I_k$, 
    %\begin{equation} \label{univest2}
    %    \sup_{i \in I_k} \E\Big[ \sup_{t \in [0,T]} \bigl|\wt{X}_t^{N,i} - \ov{X}^i_t\bigr|^2 \Big]
    %    \lesssim \frac{1}{\#I_k} \bigl( 1 + \mathcal W_k^N \wt{\mathcal W}_k^N \bigr) \biggl(\ \sum_{i \in I_k} \rho^{N,i}_d + \mathcal W_k^N \sum_{j \notin I_k}  \mathcal \rho^{N,j}_d  \biggr)\,, 
    %\end{equation}
    %where
    %\[
    %\mathcal W_k^N \defeq \min\biggl\{1,\, \sum_{i \in I_k} |\bm w^N_{i\,\cdot}|_2^2,\, \sum_{j \notin I_k} |\bm w^N_{\cdot\,j}|_2^2 \biggr\}\,,
    %\qquad
%\wt{\mathcal W}_k^N  \defeq \min\biggl\{1,\, \sum_{i \notin I_k} |\bm w^N_{i\,\cdot}|_2^2,\, \sum_{j \in I_k} |\bm w^N_{\cdot\,j}|_2^2\biggr\}\,,
%    \]
%   and
%   \begin{equation} \label{def.rhodp}
%    \rho^{N,i}_d \defeq \rho_{d} \bigl( |\bm w^N_{i\,\cdot}|_2^{-2} \bigr)\,, %\quad \text{with} \quad
%        \rho_{d}(K) \defeq 
%        \begin{dcases}
%            K^{-\frac12} & \text{if} \ d < 4
%            \\
%            K^{-\frac12} \lvert\log(K)\rvert & \text{if} \ d = 4
%            \\
%            K^{-\frac{2}{d}} & \text{if} \ d > 4\,.
%        \end{dcases}
%  \end{equation}
\end{theorem}

%Note that the formulation \eqref{wIrrcond} is a way to encode possible effects due to the asymmetry of the graph structure; for example, in the standard mean field setting we can take $I_k = \{1,\dots,N\}$ (as all players are indistinguishable) and recover a well-known convergence estimate by noting that, in this case, $\mathcal W_k^N = \wt{\mathcal W}_k^N = 0$. In fact, before proceeding with the proof, let us try to make this result less obscure by considering an interesting case where it applies. 

\begin{remark}
    Note that \eqref{wIrrcond} holds automatically if $w^N$ is symmetric, since by assumption $\sum_j w_{ij}^N = 1$ for each $i$.
\end{remark}

\begin{remark} \label{remark.wtodeg} To understand the conditions on $w^N$ in Theorem \ref{thm.universality}, suppose we given some undirected graph $\Gamma^N$ with no self-loops, and we have $w^N_{ij} = (\deg_N i)^{-1} \1_{j \sim_N i}$. As for condition~\eqref{wIrrcond}, it reads
\begin{equation} \label{wIrrcond.deg}
\limsup_{N \to \infty} \max_{j=1,\dots,N} \sum_{i \sim_N j} \frac1{\deg_N i} < \infty\,;
\end{equation}
for example, we immediately see that \eqref{wIrrcond.deg} holds if the graphs $\Gamma^N$ are regular (i.e., if $\deg_N i$ is independent of $i$), or, more generally, if there is a constant $C_\irr \geq 1$ (independent of $N$) and a sequence $(\mathfrak d_N)_{N \in \N}$ such that eventually
\begin{equation} \label{deg.irrcond}
\frac{\mathfrak d_N}{{C_\irr}}
 \le \deg_N i \le {C_\irr} \mathfrak d_N \qquad \forall \, i = 1,\dots,N\,;
\end{equation}
indeed, the general estimate
\[
\max_{j=1,\dots,N} \sum_{i \sim_N j} \frac1{\deg_N i} \leq \max_{j=1,\dots,N}\,\frac{\deg_N j}{\min_{i \sim_N j} \deg_N i} \leq \frac{\max_{i=1,\dots,N} \deg_N i}{\min_{i=1,\dots,N} \deg_N i}
\]
suggests that condition~\eqref{wIrrcond} corresponds to a control on the \emph{irregularity} of the graphs $\Gamma^N$.
Condition~\eqref{degreedivergence'}, in turn, becomes
\begin{equation} \label{degreedivergence.deg}
\max_{i=1,\dots,N} \frac{1}{\deg_N i} \cdot \sum_{i=1}^N \frac1{\deg_N i} \ \xrightarrow{N \to \infty} \ 0\,;
\end{equation}
for example, in the regime \eqref{deg.irrcond}, the quantity in \eqref{degreedivergence.deg} is proportional $N\mathfrak d_N^{-2}$, so we are requiring that $\mathfrak d_N^{-1} = o(N^{-\frac12})$. Finally, the argument of the function $\rho_d$ in \eqref{univest2} is $|\bm w^N_{i\,\cdot}|_2^{-2} = \deg_N i$.
\end{remark}

In light of the considerations in \Cref{remark.wtodeg}, we can state the following immediate consequence of \Cref{thm.universality}, which offers very clean estimates if $\Gamma^N$ are ``almost regular'', in the sense provided by condition~\eqref{deg.irrcond}.

\begin{corollary}\label{symmetriccomp} 
Suppose that either \Cref{assump.uniformdisp} or \Cref{assump.uniformLL} holds. Let $(\Gamma^N,w^N)$ be a sequence of weighted graphs, with
\[
w^N_{ij} = \mathfrak d_N^{-1} \1_{j \sim_N i}\,,
\]
and satisfying the regularity condition~\eqref{deg.irrcond} as well as
\[
\frac{N}{\mathfrak d_N^{2}} \ \xrightarrow{N \to \infty}\ 0\,.
\]
Then, we have 
\begin{equation*}
 \E\Big[ \sup_{t \in [0,T]} |\ov{\bX}_t^{N} - \wt{\bX}_t^{N}|^2 \Big] \lesssim  N \rho_d( \mathfrak d_N )\,,
 \qquad
 \E\Big[ \sup_{t \in [0,T]} |\bX_t^{N} - \wt{\bX}_t^{N}|^2 \Big] \lesssim \frac{N}{\mathfrak d_N^{2}}\,,
\end{equation*}
with implied constants independent of $N$ and with $\rho_d$ as in \eqref{est.mf.intro}.
\end{corollary}
%and, for any partition $(I_k)_{k=1,\dots,n}$ of $\{1,\dots,N\}$ and each $k = 1,\dots,n$ such that $\wt X^{N,i} \overset{\mathrm{d}}{=} \wt X^{N,j}$ whenever $i,j \in I_k$, 
%    \begin{equation} \label{univest2reg}
%        \sup_{i \in I_k} \E\Big[ \sup_{t \in [0,T]} \bigl|\wt{X}_t^{N,i} - \ov{X}^i_t\bigr|^2 \Big] 
%        \lesssim \frac{N\rho_d(\mathfrak d_N)}{\mathfrak d_N \vee \#I_k}  \,, 
%    \end{equation}
%with $\rho_d$ defined as in \eqref{def.rhodp}.
%\end{corollary}

%\begin{remark}
%Two interesting estimates come from the extremal cases of \eqref{univest2reg}; i.e., when either $\#I_k = 1$ (which we can always suppose) or $\#I_k = N$ (which can be chosen, by symmetry arguments, if $\Gamma^N$ is transitive). In general, we have
%\[
%\sup_{1 \leq i \leq N} \E\Big[ \,\sup_{t \in [0,T]} \bigl|\wt{X}_t^{N,i} - \ov{X}^i_t\bigr|^2 \Big] \lesssim \frac{N\rho_d(\mathfrak d_N)}{\mathfrak d_N} \,,
%\]
%and, for transitive graphs,
%\[
%\sup_{1 \leq i \leq N} \E\Big[ \sup_{t \in [0,T]} \bigl|\wt{X}_t^{N,i} - \ov{X}^i_t\bigr|^2 \Big] \lesssim \rho_d(\mathfrak d_N)\,.
%\]
%Note that in the former case, the right-hand side vanishes as $N \to \infty$ provided that (for example, if $d>4$) we make the stronger assumption that $\mathfrak d_N^{-1} = o(N^{-\frac{d}{d+2}})$.
%\end{remark}

We now approach the proof of \Cref{thm.universality}. Notice that the bound \eqref{univest1} follows directly from \Cref{thm.olcldisp} or \Cref{thm.olclLL}, provided that their hypotheses are in force, and that the respective condition between \eqref{smallness} and \eqref{smallness1} holds. On the other hand, to prove \eqref{univest2} we need a lemma, which is a generalized version of the classic estimate of \cite[Theorem~1]{FG13}. We note that our proof is inspired by that of \cite[Lemma~4.1]{CDCP24}, and in fact we propose an improvement of this recent result.

\begin{lemma} \label{lem.FGgen} 
Let $(\eta^i)_{i=1,\dots,N}$ be independent random variables with $m^i = \cL(\eta^i) \in \cP_q(\R^d)$, for some $q > 0$. Let $(\omega^i)_{i=1,\dots,N}$ be non-negative weights such that $\sum_{i=1}^N \omega^i = 1$. Then, for any $r \in (0,q)$, with implied constant depending only on $d$, $q$, and $r$, we have 
\[
\E\biggl[\, \bd_r\biggl(\,\sum_{1 \leq i \leq N} \omega^i m^i\,,\, \sum_{1 \leq i \leq N}\omega^i \delta_{\eta^i} \biggr)^r\, \biggr] \lesssim M_q(\bm{m})^{\frac{r}{q}} \rho_{d,q,r}\bigl(|\bm\omega|^{-2}\bigr)\,,
\]
where  
\[
M_q(\bm{m}) \coloneqq \max_{i=1,\dots,N} M_q(m^i)\,, \quad M_q(m^i) \coloneqq \int_{\R^d} |x|^q \,m^i(\d x)\,, 
\]
\[
\rho_{d,q,r}(K) \defeq 
\begin{dcases}
K^{-\frac12} + K^{-\frac{q-r}{q}} & \text{if} \ r > \frac d2 \text{ and } q \neq 2r \\
K^{-\frac12}\log(1+K) + K^{-\frac{q-r}{q}} & \text{if} \ r = \frac d2 \text{ and } q \neq 2r \\
K^{-\frac rd} + K^{-\frac{q-r}{q}} & \text{if} \ r < \frac d2 \text{ and } q \neq \frac{d}{d-r}
\end{dcases}
\]
and $|\bm\omega|$ is the Euclidean norm of the vector $\bm\omega = (\omega^i)_{i=1,\dots,N}$.
\end{lemma}

\begin{proof}
First, we show that it suffices to prove the result in the case that $M_q(\bm{m}) = 1$; indeed, if it holds in that case, then, given arbitrary $m^1,\dots,m^N$, we can define, for each $i=1,\dots,N$, 
\[
    \tilde{m}^i \defeq \bigl(M_q(\bm{m})^{-\frac1q} \mathrm{id}_{\R^d}\bigr)_{\sharp} m^i\,, 
\qquad
    \tilde{\eta}^i \defeq M_q(\bm{m})^{-\frac1q} \eta^i\,,
\]
and notice that $M_q(\tilde{\bm m}) = 1$ and $\cL(\tilde{\eta}^i) = \tilde{m}^i$, so that 
\[
\begin{split}
    M_q(\bm{m})^{-\frac{r}{q}} \E\biggl[\, \bd_r\biggl(\,\sum_{i=1}^N \omega^i m^i\,,\, \sum_{i=1}^N\omega^i \delta_{\eta^i} \biggr)^r\, \biggr] = 
    \E\biggl[\, \bd_r\biggl(\,\sum_{i=1}^N \omega^i \tilde{m}^i\,,\, \sum_{i=1}^N\omega^i \delta_{\tilde{\eta}^i} \biggr)^r\, \biggr] 
    \lesssim \rho_{d,q,r}\bigl(|\bm\omega|^{-2}\bigr)\,.
\end{split}
\]
Hence we will suppose that $M_q(\bm m) = 1$.
Let $A \subset \R^d$ be any Borel set. We have
\begin{equation} \label{FGp.e1}
\E\biggl[\, \biggl\lvert\, \sum_{ i } \omega^i \bigl(m^i - \delta_{\eta^i}\bigr) (A) \biggr\rvert \,\biggr] \leq \sum_i \omega^i \,\E \bigl[ \bigl| (m^i - \delta_{\eta^i}) (A) \bigr|^2 \bigr] = \sum_i \omega^i \bigl( m^i(A) - m^i(A)^2 \bigr)\,;
\end{equation}
on the other hand, by the independence of the $\eta^i$,
\begin{equation} \label{FGp.e2}
\E\biggl[\, \biggl\lvert\, \sum_{ i } \omega^i \bigl(m^i - \delta_{\eta^i}\bigr) (A) \biggr\rvert^2 \,\biggr] = \sum_i (\omega^i)^2 \,\E\bigl[\bigl| (m^i-\delta_{\eta^i})(A) \bigr|^2\bigr] = \sum_i (\omega^i)^2 \bigl( m^i(A) - m^i(A)^2 \bigr)\,.
\end{equation}
Then combine \eqref{FGp.e1} and \eqref{FGp.e2} to deduce that
\begin{equation} \label{FGp.E1}
\E\biggl[\, \biggl\lvert\, \sum_{ i } \omega^i \bigl(m^i - \delta_{\eta^i}\bigr) (A) \biggr\rvert \,\biggr] \leq \min \biggl\{ \sum_i \omega^i m^i(A),\, \Bigl(\, \sum_i (\omega^i)^2 m^i(A) \Bigr)^{\frac12} \biggr\}\,.
\end{equation}
Now, for $\ell \in \N$, let $\mathscr P_\ell$ be the partition of $(-1,1]^d$ into the $2^{d\ell}$ dyadic cubes $\mathfrak C_\ell \defeq (-2^{-\ell},2^{-\ell}]^d$. Set $B_n \defeq \mathfrak C_{-n} \setminus \mathfrak C_{-(n-1)}$. Also, for $F \in \mathscr P_\ell$ and $n \in \N$, set $2^n F \defeq \{ x \in \R^d :\ 2^{-n}x \in F \}$. Then, using \eqref{FGp.E1} and the Cauchy--Schwarz inequality, we have
\[
\sum_{F \in \mathscr P_\ell} \E\biggl[\, \biggl\lvert\, \sum_{ i } \omega^i \bigl(m^i - \delta_{\eta^i}\bigr) (2^nF \cap B_n) \biggr\rvert \,\biggr] \leq \min \biggl\{ \sum_i \omega^i m^i(B_n),\ 2^{\frac{d\ell}2} \Bigl( \, \sum_i (\omega^i)^2 m^i(B_n) \Bigr)^{\frac12} \biggr\}\,,
\]
and, since $M_q(m^i) \leq 1$,
\[
m^i(B_n) \leq 2^{-q(n-1)} \int_{B_n} |\cdot|^q \,\d m^i \leq 2^{-q(n-1)}\,.
\]
Therefore, we obtain
\[
\sum_{F \in \mathscr P_\ell} \E\biggl[\, \biggl\lvert\, \sum_{ i } \omega^i \bigl(m^i - \delta_{\eta^i}\bigr) (2^nF \cap B_n) \biggr\rvert \,\biggr] \lesssim \min \bigl\{ 2^{-qn},\, 2^{\frac{d\ell-q}2} |\bm\omega| \,\bigr\}\,,
\]
with implied constant depending only on $q$. The rest of the proof is now the same as for \cite[Theorem~1]{FG13}, where $N$ is replaced with $|\bm\omega|^{-2}$.
\end{proof}

\begin{proof}[Proof of \Cref{thm.universality}]
As anticipated, to prove \eqref{univest1} we only need to check that we can apply \Cref{thm.olcldisp} and \Cref{thm.olclLL}.
In the displacement semi-monotone case, note that \Cref{assump.uniformdisp} is almost a restatement of Assumption \ref{assump.disp} in the case where $F^{N,i}(\bx) = \cF(x^i,m_{\bx,w}^{N,i})$ and $G^{N,i}(\bx) = \cG(x^i,m_{\bx,w}^{N,i})$, except for the Lipschitz property of $\mathrm{diag}(D \bF^{N})$ and $\mathrm{diag}(D \bG^{N})$. We verify this, uniformly in $N$, for $\bF$, the proof for $\bG$ being analogous. We use the Lipschitz continuity of $D_x \cF$ to find that
\[
\sum_{i=1}^N \,\biggl|D_x\cF\Bigr|^{(x^i,m_{\bx,w}^{N,i})}_{(\ov x^i,m_{\ov{\bx},w}^{N,i})} \biggr|^2
\leq 2\|D_{xx} \cF\|_\infty |\bx -\bar{\bx}|^2 + 2\norm{D_{mx} \cF}_\infty \sum_{i=1}^N \bd_2(m_{\bx,w}^{N,i},m_{\bar{\bx},w}^{N,i})^2\,,
\]
where
\[
\sum_{i} \bd_2(m_{\bx,w}^{N,i},m_{\bar{\bx},w}^{N,i})^2 \leq \sum_{i,j} w^N_{ij} |x^j - \bar x^j|^2 \leq  \max_j |\bm w^N_{\cdot\,j}|_1  |\bx - \bar{\bx}|^2\,,
\]
where the constant on the right-hand side is uniformly bounded in $N$ thanks to \eqref{wIrrcond}.
We now show that \eqref{smallness} holds, provided that $N$ is large enough. Since, for $j \neq i$,
\[
D_{j} F^{N,i} = w^N_{ij} D_m \cF(x^i, m_{\bx,w}^{N,i}, x^j) \,,
\]
we have
\begin{equation*}
    \delta^i \leq \bigl( \|D_m \cF\|^2_\infty + \|D_m \cG\|^2_\infty\bigr) |\bm w^N_{i\,\cdot}|_2^2 
\end{equation*}
so we see that \eqref{degreedivergence'} is equivalent to ${\delta} \sum_{i} \delta^{i} \to 0$ as $N \to \infty$.
In the Lasry--Lions semi-monotone setting, one verifies in a similar manner that  \Cref{thm.olclLL} applies. Note that we also need to estimate
\[ %\label{deltatozero}
    \kappa^{i} \leq \bigl(\|D_{mx} \cF\|^2_\infty + \|D_{mx} \cG\|^2_\infty\bigr) |\bm w^N_{i\,\cdot}|_2^2 
\]
and
\[
    \tilde\kappa^i \leq \bigl(\|D_{mx} \cF\|^2_\infty + \|D_{mx} \cG\|^2_\infty\bigr) |\bm w^N_{\cdot\,i}|_2^2 \,,
\]
so that \eqref{degreedivergence'} implies that $\sqrt{\kappa\tilde\kappa} \to 0$ as $N \to \infty$.
   
   For the bound \eqref{univest2}, we note that the equation \eqref{meanfieldpointiid} can be rewritten as 
   \begin{equation*} %\label{meanfieldponterror}
    \begin{dcases}
        \d\ov{X}^i_t = - D_p H(\ov{X}^i_t, \ov{Y}^i_t) \,\d t + \sqrt{2\sigma} \,\d W_t^i + \sqrt{2 \sigma_0} \,\d W_t^0
        \\
         \d\ov{Y}^i_t = \bigl( D_x H(\ov{X}^i_t,\ov{Y}^i_t) - D_{i} F^{N,i}(\ov{\bX}_t^N) - E_t^{F,N,i} \bigr) dt +\ov{Z}^i_t \,\d W_t^0 +\ov{Z}^{i,0}_t \d W_t^0 
        \\
         \ov{\bX}_0 = \bm\zeta\,, \quad \ov{Y}_T^i = D_{i} G^{N,i}(\ov{\bX}_T^N) + E^{G,N,i}_T\,,
    \end{dcases}
\end{equation*}
with 
\[
    E_t^{F,N,i} \defeq D_x \cF\Bigr|^{(\ov{X}^i_t, \cL(\ov{X}^i_t \mid \scrF_t^0))}_{(\ov{X}_t^i, m_{\ov{\bX}_t^N\!\!,w}^{N,i})}\,,
\]
and likewise for $E^{G,N,i}_T$, and where we have set $\ov{\bX}_t^N = (\ov{X}_t^{1},\dots,\ov{X}_t^N)$. 
 Let $m_t \defeq \cL(X_t \mid \scrF_t^0)$, and note that we have $m_t = \cL(\ov{X}_t^i \mid \scrF_t^0)$ for each $i$. Also, note that
\[
        | E_t^{F,N,i} |^2 
        \leq \norm{D_{mx} \cF}_\infty^2  \,\bd_2\bigl(m_t, m_{\ov{\bX}^N_t\!\!,w}^{N,i}\,\bigr)^2\,,
\]
and similarly for $| E^{G,N,i}_T|^2$. Since, by construction, $\ov{X}^1,\dots,\ov{X}^N$ are i.i.d.\ with common law $m_t$, conditionally on $\scrF_t^0$, \Cref{lem.FGgen} yields 
\[
    \E\Bigl[\bd_2\bigl(m_t, m_{\ov{\bX}^N_t\!\!,w}^{N,i}\,\bigr)^2 \Bigr] = \E\Bigl[ \E\Bigl[ \bd_2\bigl(m_t, m_{\ov{\bX}^N_t\!\!,w}^{N,i}\,\bigr)^2 \;\Big|\, \scrF_t^0 \Bigr] \Bigr] 
    \lesssim \E\Bigl[ M_p(m_t)^{\frac2p} \Bigr] \rho_{d}^{N,i} \lesssim \E\bigl[ M_p(m_t) \bigr]^{\frac2p} \rho_d(|w^N_{i\cdot}|)\,,
\]
where the implied constants depend only on $d$ and $p$, and by assumption 
\[
    \sup_t \E\big[ M_p(m_t) \big] = \sup_t \E\big[ |X_t|^p \big] < \infty\,.
\]
We invoke \Cref{lem.lipschitzstability} to complete the proof of \eqref{univest2}. Finally, standard arguments using the convergence rate of i.i.d.\ samples \cite{FG13} show that 
\begin{align*}
    \E\Big[ \sup_{t \in [0,T]} \bd_2\big(m_{\ov{\bX}_t^N}^N, m_t\big)^2 \Big] \lesssim \rho_d(N), 
\end{align*}
so the estimates on the empirical measures follow in a straightforward way from the triangle inequality. 
\end{proof}

{
\begin{remark}
It is clear from the proof above that condition \eqref{wIrrcond} in \Cref{thm.universality} is needed to guarantee that the Lipschitz constants of $\mathrm{diag}(D \bF^{N})$ and $\mathrm{diag}(D \bG^{N})$ are in fact uniformly bounded in $N$; on the other hand, \eqref{degreedivergence'} is needed to ensure the validity of the additional conditions \eqref{smallness}, \eqref{smallness1'}, and \eqref{smallness1}. Note that \eqref{smallness} and \eqref{smallness1} are not required to prove the estimate between open-loop and mean field equilibria, because they are not needed to apply \Cref{lem.lipschitzstability}, and neither is \eqref{smallness1'} in the displacement semi-monotone regime. This also motivates the way \Cref{cor.network3} is stated in the Introduction; analogous considerations apply to \Cref{symmetriccomp} (in light of \Cref{remark.wtodeg}), and to \Cref{thm.vvlS} below.
\end{remark}
}

\begin{remark}%\label{remark.gengraphons}
%\marginpar{\tiny DF: Restored this remark, but in an abridged form with no strange claims.}
 It is also worth noting that, as \Cref{lem.FGgen} does not require the random variables to be identically distributed, with a similar proof it should be possible to generalize our estimate \eqref{univest2} 
 %as follows. Suppose that $\bar X^{N,i}$ solves
% \begin{equation} \label{meanfieldpointiidG}
%    \begin{dcases}
%        \d\ov{X}^{N,i}_t = - D_p H(\ov{X}^{N,i}_t, \ov{Y}^{N,i}_t) \,\d t + \sqrt{2\sigma} \,\d W_t^i + \sqrt{2 \sigma_0} \,\d W_t^0
%        \\ \begin{aligned}
%        \d\ov{Y}^{N,i}_t &= - \Bigl( D_x H(\ov{X}^{N,i}_t,\ov{Y}^{N,i}_t) + D_x \cF\bigl(\ov{X}^{N,i}_t, \mathbb W^{N,i}\bigl[\cL(\ov{X}^{N,\,\cdot}_t \mid \scrF_t^0)\bigr]\bigr) \Bigr) \,\d t 
%        \\
%        &\quad +\ov{Z}^{N,i}_t \d W_t^i +\ov{Z}^{N,i,0}_t \d W_t^0
%        \end{aligned}\\
%        \ov{X}^{N,i}_0 = \zeta^i_0\,, \quad \ov{Y}_T^{N,i} = D_x \cG\bigl(\ov{X}_T^{N,i}, \mathbb W^{N,i}\bigl[\cL(\ov{X}^{N,\,\cdot}_T \!\mid \scrF_T^0)\bigr]\bigr)\,,
%    \end{dcases}
%    \end{equation}
%    where
%    \[
%    \mathbb W^{N,i}\bigl[\cL(\ov{X}^{N,\,\cdot}_t \mid \scrF_t^0)\bigr] \defeq \sum_{j=1}^N w^N_{ij} \cL(\ov{X}^{N,j}_t \!\mid \scrF_t^0)\,,
%    \]
%    in which we allow, this time, $|\bm w^N_{i\,\cdot}|_1 \in [0,1]$; then we have
% \begin{equation} \label{univestgraphon}
% \E\Big[ \sup_{t \in [0,T]} \bigl|\wt{\bX}_t^{N} - \ov{\bX}^{N}_t\bigr|^2 \Big]
%        \lesssim |\bm w^N_{i\,\cdot}|_1^{2/p} \sum_{i=1}^N \rho_d\biggl(\frac{|\bm w^N_{i\,\cdot}|_1}{|\bm w^N_{i\,\cdot}|_2^2}\biggr) \,. 
% \end{equation}
%  This estimate provides a convergence rate that can be used inside 
to the framework of \emph{graphon mean field games}, whenever the weighted graphs $(\Gamma^N,w^N)$ converge to a graphon $W$ (for instance, in the cut norm),\footnote{Recall that this is equivalent to
\[
\sum_{1 \leq i \leq N} \int_{I^N_i} \biggl|\, \sum_{1 \leq j \leq N} \int_{I^N_j} \bigl( w^N_{ij} - W(\gamma,\nu) \bigr) \phi(\nu)\,\d \nu \biggr| \,\d \gamma \ \xrightarrow{N \to \infty} \ 0\,, \qquad \forall\, \phi \in L^\infty([0,1])\,,
\]
where $I^N_i \defeq \bigl(\frac{i-1}N,\frac iN\bigr)$; see, e.g., \cite[Lemma~8.11]{LasGraphon}.
} and we have a suitable discrete approximation of the following graphon game system labelled by $\gamma \in [0,1]$:%, built by choosing $\bar X^i = \bar X^{\gamma_i}$ with $\gamma_i \in I^N_i \defeq \bigl(\frac{i-1}N,\frac iN\bigr)$:
\begin{equation*} %\label{meanfieldpointiidGcont}
\begin{dcases}
       \d\ov{X}^{\gamma}_t = - D_p H(\ov{X}^\gamma_t, \ov{Y}^\gamma_t) \,\d t + \sqrt{2\sigma} \,\d W_t^\gamma + \sqrt{2 \sigma_0} \,\d W_t^0
       \\
       \d\ov{Y}^\gamma_t = - \bigl( D_x H(\ov{X}^\gamma_t,\ov{Y}^\gamma_t) + D_x \cF(\ov{X}^i_t, \mathbb W^\gamma[\cL(\ov{X}^{\,\cdot}_t \mid \scrF_t^0)]) \bigr) \,\d t +\ov{Z}^\gamma_t \,\d W_t^\gamma +\ov{Z}^{\gamma,0}_t \d W_t^0
       \\
       \ov{X}^\gamma_0 = \zeta^\gamma_0\,, \quad \ov{Y}_T^\gamma = D_x \cG(\ov{X}_T^\gamma, \mathbb W^\gamma [\cL(\ov{X}^{\,\cdot}_T \mid \scrF_T^0)])\,,
   \end{dcases}
\end{equation*}
with
\[
\mathbb W^{\gamma}[\cL(\ov{X}^{\,\cdot}_t \mid \scrF_t^0)] \defeq \int_0^1 W(\gamma,\nu) \cL(\ov{X}^{\nu}_t \mid \scrF_t^0)\,\d \nu\,, \qquad W \in L^1([0,1]^2;[0,1])\,.
\]
Note that in this case we need to weaken the former assumption that $|\bm w^N_{i\,\cdot}|_1 = 1$ by asking instead that $|\bm w^N_{i\,\cdot}|_1 \leq 1$; this leads to consider sub-probability measures, so it brings up some technical issues, and addressing them all in detail is beyond the purposes of the present work.%; we only mention that, for instance, it is required that $\cF$ and $\cG$ be defined on sub-probability measures, to extend the notions of derivative with respect to the measure (in order to check \Cref{assump.uniformdisp,assump.uniformLL}).\footnote{For instance, one can let $D_\mu  \cF(x,\mu) \defeq \lambda^{-1} D_m \cF(x,\lambda m)|_{m = \lambda^{-1} \mu} \1_{\mu \neq 0}$ for any sub-probability measure $\mu$ on $\R^d$ with mass $\lambda$, so that, defining $\bd_2(\lambda m,\lambda m') \defeq \lambda \bd_2(m,m')$ for all $m,m' \in \cP_2(\R^d)$ and each $\lambda \geq 0$, we have $|\mathcal F(x,\mu) - \mathcal F(x,\mu')| \leq \norm{D_\mu\cF}_\infty \bd_2(\mu,\mu')$ whenever $\mu$ and $\mu'$ have finite second moments and the same mass.}
%
%Finally, it is interesting to see that estimate~\eqref{univestgraphon} also captures ``limit cases'' such as when $\bm w^N_{i\,\cdot} = \bm 0$ for some $i$; in this situation, player $i$ is playing alone a $1$-player game (i.e., is dealing with a control problem), ignoring the other players, and we see that \eqref{univestgraphon} gives that $\wt X^{N,i} = \bar X{}^{N,i}$ in $L^2(\Omega;L^\infty([0,T]))$, as we expect since we know that they are indeed equal. More in general, if $\bm w^N_{i\,\cdot} \to \bm 0$ as $N \to \infty$, \eqref{univestgraphon} exhibits a better rate of convergence, reflecting the fact that $\wt X^{N,i} = \bar X{}^{N,i}$ would tend to get decoupled from the rest of the game and thus to coincide.
\end{remark}

\subsection{Remarks and examples on the standing assumptions} \label{subsec.univexamples}

A simple example has been anticipated in \Cref{remark.wtodeg} to illustrate how
conditions \eqref{wIrrcond} and \eqref{degreedivergence'} are related to the structure of the graphs. 
Regarding the conditions on monotonicity, we have already discussed in Section~\ref{subsec.pairwise}, in particular \Cref{cor.network3}, an example where the displacement monotonicity condition is satisfied uniformly in $N$. We provide here some other interesting examples where the uniform in $N$ monotonicity conditions can be verified. 

\begin{example} \label{example1}
 Let, for simplicity, $d=1$. Let also $\phi : \R \to \R$ be smooth, bounded and with bounded derivatives; fix $A \in \R$ and non-negative weights $(w^N_{ij})_{i,j=1,\dots,N}$ such that $\sum_{j=1}^N w^N_{ij} = 1 = \sum_{i=1}^N w^N_{ij}$ and $w^N_{ii}=0$ for all $i$, and let  
\[
F^{N,i}(\bx) \defeq \frac A2 \biggl(\,\phi(x^i) - \sum_{j=1}^N w^N_{ij} \phi(x^j)\biggr)^2.
\]
Note that, defining the function $\cF(x,m) \defeq \frac A2 \big(\phi(x)-\int_{\R} \phi\,\d m \big)^2$, we have $F^{N,i}(\bx) = \cF(x^i,m_{\bx,w}^{N,i})$.
The displacement semi-monotonicity condition \eqref{disp2} becomes
\[
\begin{multlined}[.95\displaywidth]
A  \sum_{i} (\phi'(x^i) \xi^i)^2  - A  \sum_{i, j} w^N_{ij} \,  \phi'(x^i) \xi^i \phi'(x^j) \xi^j + \sum_{i} A \Big(\phi(x^i) - \sum_{j} w^N_{ij} \phi(x^j) \Big)\phi''(x^i) (\xi^i)^2 
\\[-3pt]
\geq - C_{\bF, \dis} |\bxi|^2\,.
\end{multlined}
\]
Since the first two terms on the left-hand side together give a non-negative contribution, such inequality is satisfied uniformly in $N$ for 
\[
A > 0\,, \qquad C_{\bF,\dis} = 2 A \|\phi\|_\infty \|\phi''\|_\infty\,.
\]
Note that $\|\phi\|_\infty \|\phi''\|_\infty$ is small whenever $\phi$ is close to the identity on compact regions and ``slowly bends'' to constant as $|x| \to \infty$. 

On the other hand, the Lasry--Lions semi-monotonicity condition \eqref{LL2} becomes
\[
 - A \sum_{\substack{i, j } } w^N_{ij}  \phi'(x^i) \xi^i \phi'(x^j) \xi^j \geq - C_{\bF,\lm} |\bxi|^2\,;
\]
it is satisfied uniformly in $N \ge N_0$ for 
\[
A < 0\,, \qquad C_{\bF,\lm} = |A| \|\phi'\|_\infty \sup_{N \ge N_0} |\lambda_1^N|\,,
\]
where ${\lambda}_1^N$ is the first eigenvalue of the matrix $w^N$ (that needs to be negative since $w^N$ is a hollow matrix with sum of all rows equals $1$). In fact, for many sequences of regular graphs, it is true that $\lim_{N \to \infty} \lambda_1^N = 0$; see for example \cite{Friedman08}. 
\end{example}

\begin{example} Let $d=1$. Let $\psi : \R^2 \to \R$ be smooth, bounded and with bounded derivatives; fix $A \in \R$ and $w^N$ as in \Cref{example1}, and let
\[
F^{N,i}(\bx) \defeq \sum_{j=1}^N w^N_{ij} \psi(x^i, x^j)\,.
\]
In this case, we will only discuss the Lasry--Lions semi-monotonicity condition. It will be useful to recall that \eqref{llmonpoint}, with the previous choice of $F^{N,i}$, is equivalent to
\[
\sum_{i,j} w^N_{ij} \big(\psi(x^i, x^j) - \psi(\bar x^i, x^j) - \psi(x^i, \bar x^j) + \psi(\bar x^i, \bar x^j)\big) \geq - C_{\bF,\lm} |\bx - \bar{\bx}|^2
\]
for all $\bx, \bar{\bx} \in \R^N$; see \cite[Remark 3.3]{CR24b}. Note that the previous inequality can be rewritten as
\[
 \int_{\R^{2}} \psi \sum_{i,j} w^N_{ij}\,\d \big((\delta_{x^i} - \delta_{\bar x^i}) \otimes (\delta_{x^j} - \delta_{\bar x^j})\big) \ge - C_{\bF,\lm} |\bx - \bar{\bx}|^2
\]
Let ${\lambda}_1^N$ be the first eigenvalue of the matrix $w^N$, and $\Omega^N$ be the square root of the symmetric positive semidefinite matrix $\frac12(w^N +  (w^N)^\TT) - {\lambda}_1^N I_N $. Then, setting $\mu^{k,N} = \sum_{i=1}^N \Omega^N_{ki}(\delta_{x^i} - \delta_{\bar x^i})$, the integral on the left-hand side above equals
\[
\sum_{k} \int_{\R^{2}} \psi\,\d \bigl(\mu^{k,N} \otimes \mu^{k,N}\bigr) + {\lambda}_1^N \sum_{i} \big(\psi(x^i, x^i) - \psi(\bar x^i, x^i) - \psi(x^i, \bar x^i) + \psi(\bar x^i, \bar x^i)\big)\,.
\]
If the operator $\mu \mapsto \int \psi \,\d(\mu \otimes \mu)$ induced by $\psi$ in $\mathcal M(\R)$ is nonnegative, then we find that the Lasry--Lions semimonotonicity condition is satisfied (for large $N$) with
\[
C_{\bF,\lm} = \sup_{N \ge N_0} |\lambda_1^N| \|D^2\psi\|_\infty\,.
\]

\end{example}

\begin{example} The displacement convexity Assumption \ref{disp2} can be shown to hold provided that $\cF, \cG$ are convex ``enough'' in the $x$ variable. Indeed, suppose that $L_0$ is $C^2$ and satisfies \eqref{disp.L0},
and suppose that the functions $D_x \cF$ and $D_x \cG$ satisfy for some $C_{\cF, \dis}, C_{\cG, \dis}$ 
    \begin{equation} \label{suffconvex}
    %\begin{aligned}
        D_{xx} \cF(x,m) \geq  \bigl(|w^N|_{\op} \|D_{mx}\cF\|_{\infty} - C_{\cF, \dis} \bigr) I_{d}\,,  %\\
       % D_{xx} \cG(x,m) &\geq  \bigl(|w^N|_{\op} \|D_{mx}\cG\|_{\infty} + C_{\cG, \dis} \bigr) I_{d}
     %  \end{aligned}
        \end{equation}
   and likewise for $\cG$, for each $x \in \R^d$ and $m \in \cP_2(\R^d)$; if
      \begin{equation*} 
        C_{L} - \frac{T^2}{2} C_{\cF,\dis} - T C_{\cG, \dis} > 0
    \end{equation*}
    also holds, then \Cref{assump.uniformdisp} is verified easily by means of \Cref{assump.uniformdisploc}. Moreover, we note that by the Perron--Frobenius Theorem, 
    \begin{equation*}
        |w^N|_{\op}^2 = |(w^N)^\TT w^N|_{\op} \leq \max_i \sum_{jk} w^N_{ki} w^N_{kj} \leq \max_{i} |\bw_{\cdot\, i}|_1\,, 
    \end{equation*}
    with the last bound using the fact that $\sum_j w_{kj}^N = 1$ for each $k$. 
    Thus we have $|w^N|_{\op} \leq \max_{i} \sqrt{|\bw^N_{\cdot\, i}|_1}$, so that \eqref{suffconvex} can be replaced by 
      \begin{equation*} %\label{suffconvex2}
      %\begin{aligned}
        D_{xx} \cF(x,m) \geq  \Bigl(\, \max_{i=1,\dots,N} \sqrt{|\bw^N_{\cdot\, i}|_1}\, \|D_{mx} \cF\|_{\infty} - C_{\cF, \dis} \Bigr) I_{d}\,,%  \\
       % D_{xx} \cG(x,m) &\geq  \Bigl( \max_{i=1,\dots,N} \sqrt{|\bw^N_{\cdot\, i}|_1}\|D_mD_x\cG\|_{\infty} + C_{\cG, \dis} \big) I_{d \times d}.
   %\end{aligned}
   \end{equation*}
    In particular, if each $w^N$ is symmetric, then $|\bw^N_{\cdot\, i}|_1 = |\bw^N_{i\, \cdot}|_1 = 1$, so the condition can be further simplified to 
     \begin{equation*} %\label{suffconvex3}
        D_{xx} \cF(x,m) \geq  \big(\|D_{mx}\cF\|_{\infty} - C_{\cF, \dis} \big) I_{d}\,.
       \end{equation*}
\end{example}

\begin{example}[The MFG case]\label{ex:mfcase} We finally consider the standard symmetric mean field case, where $w^{N}_{ij}=1/(N-1)$ for each $N, i, j$; that is, for $\cF$ and $\cG$ as above,
\begin{equation*} 
    F^{N,i}(\bx) \defeq \cF(x^i,m_{\bx}^{N,i})\,, \qquad G^{N,i}(\bx) \defeq \cG(x^i,m_{\bx}^{N,i})\,,
\end{equation*}
where $m_{\bx}^{N,i} = \frac1{N-1} \sum_{j \neq i} \delta_{x^j}$ are the usual empirical measures. We discuss our results under Lasry--Lions monotonicity conditions on $\cF$ and $\cG$, that read as follows
\begin{equation}\label{cFLL}
\int_{\R^d} \int_{\R^d} \phi(x)^\TT D_{mx} \cF(x, m, y) \phi(y) m(\d x) m (\d y) \ge 0,
\end{equation}
and likewise for $\cG$, for every measurable vector field $\phi$ and $m \in \cP_2(\R^d)$. For technical reasons, we also assume that $m \mapsto D_{mx} \cF(x, m, y)$ is uniformly continuous with respect to the total variation norm $\|\cdot \|_{\text{TV}}$, uniformly in $(x,y)$, and likewise for $m \mapsto D_{mx} \cG(x,m,y)$.

Let us check Assumption \ref{assump.uniformLL}. Under the desired conditions on $L$, we in particular need to verify \eqref{LL2}: we have
\begin{equation}\label{verifyLL}
\frac{1}{N-1}\sum_{i,j:\, i \neq j} (\xi^i)^\TT D_{mx}\cF(x^i,m^{N,i}_{\bx},x^j)\, \xi^j = \frac{1}{N-1}\sum_{i , j} (\xi^i)^\TT D_{mx}\cF(x^i,m^{N}_{\bx},x^j)\, \xi^j + E^N,
\end{equation}
where
\[
E^N =  -\frac{1}{N-1}\sum_{i} (\xi^i)^\TT D_{mx} \cF(x^i,m^{N,i}_{\bx},x^i)\, \xi^i +
\frac{1}{N-1}\sum_{i,j} (\xi^i)^\TT D_{mx}\cF\Bigr|^{(x^i,m^{N,i}_{\bx},x^j)}_{(x^i,m^{N}_{\bx},x^j)} \xi^j.
\]
While the first term of the right hand side in \eqref{verifyLL} is nonnegative by \eqref{cFLL} (that applies with $m = m_{\bx}^{N}$ and $\phi(x^i) = \xi^i$ for all $i$), denoting by $\omega_{\cF}$ the modulus of continuity of $D_{mx}\cF$ with respect to $m$ in the total variation norm (it needs to be uniform with respect to the other variables), we get
\[
|E^N| \le \biggl(\frac{\|D_{mx}\cF\|_{\infty}}{N-1}+\omega_{\cF}\Bigl(\sup_i \mathsf \|m^{N,i}_{\bx} - m^{N}_{\bx}\|_{\text{TV}} \Bigr) \biggr)|\bxi|^2.
\]
Since the right hand side of the previous inequality goes to zero as $N \to \infty$, we obtain that \eqref{LL2} is satisfied for any $C_{\bF,\lm}, C_{\bG,\lm} > 0$, provided that $N$ is large enough. Hence, \Cref{assump.uniformLL} holds for any $T> 0$, provided that $C_{\bF,\lm}, C_{\bG,\lm} > 0$ are taken small enough (and therefore, $N$ is large).

Then \Cref{symmetriccomp} applies, yielding %\marginpar{\tiny DF: Why are these not stated with $sup_i$ in the first estimate? {actually, I guess the second one should not be a sup over $i$ now that we have removed the more precise stability result}}
\[ \E\Big[ \sup_{t \in [0,T]} |\bX_t^{N} - \wt{\bX}_t^{N}|^2 \Big] \lesssim \frac1{N}, \qquad  \E\Big[ \sup_{t \in [0,T]} \bigl|\wt{\bX}_t^{N} - \ov{\bX}_t\bigr|^2 \Big] 
        \lesssim N \rho_d(N) \,, 
\]
with $\rho_d$ defined as in \eqref{est.mf.intro}.

\end{example}

\section{A joint vanishing viscosity and large \texorpdfstring{$N$}{N} limit}

\label{sec.jointlimit}

In this section, we consider a displacement semi-monotone mean field game with purely common noise. The data for this game consists of $\cL : \R^d \times \R^d \times \cP_2(\R^d) \to \R$, $\cG \colon \R^d \times \cP_2(\R^d) \to \R$,  
as well as a constant $\sigma_0 \geq 0$. We make the following assumptions on the data. The Lagrangian $\cL$ determines a Hamiltonian 
\begin{align*}
    \cH : \R^d \times \R^d \times \cP_2(\R^d) \to \R, \quad \cH(x,p,m) = \sup_{a \in \R^d} \big\{ - a \cdot p  - \cL(x,a,m) \big\}.
\end{align*}

\begin{assumption}[Displacement semi-monotonicity and regularity, mean field case] \label{assump.vanishingvisc}
    The Lagrangian $\cL$ (and the corresponding Hamiltonian $\cH$) are jointly $C^2$; $\cG$ is $C^2$, and the derivatives $D_{m} \cL$, $D_{xm} \cL$, $D_{xx} \cL$, $D_{mm} \cL$, $D_m \cG$, $D_{xm} \cG$, $D_{xx} \cG$, and $D_{mm} \cG$ are all bounded.
    In addition, there are constants $C_{\cL,x}$, $C_{\cL,a}$, $C_{\cG,\dis} \geq 0$ such that
\begin{equation*} \begin{multlined}[.85\displaywidth]
        \E\Bigl[ D_x \cL\Bigr|^{(X,\alpha, \cL(X))}_{(X',\alpha',\cL(X'))} \cdot (X - X') +  D_a \cL\Bigr|^{(X,\alpha, \cL(X))}_{(X',\alpha',\cL(X'))} \cdot (\alpha - \alpha')\Bigr] 
        \\[3pt]
        \geq  C_{\cL,a} \, \E\big[ |\alpha - \alpha'|^2 \big] - C_{\cL,x}\, \E\big[|X-X'|^2 \big]\,, 
\end{multlined} \end{equation*}
    as well as
    \begin{equation*}
        \E\Big[ D_x \cG \Big|^{(X,\cL(X))}_{(X',\cL(X'))} \cdot (X - X') \Big] \geq - C_{\cG,\dis} \E\big[ |X- X'|^2 \big]\,,
    \end{equation*}
    for all square-integrable random variables $X,X',\alpha,\alpha'$. Finally, we have
    \begin{equation*}% \label{cdis.vv}
        \calC_\dis \defeq C_{\cL,a} - T\Bigl( C_{\cG,\dis} + \frac{T}2\, C_{\cL, x}\Bigr) > 0\,.
\end{equation*}

\end{assumption}
We consider a mean-field game defined as in \Cref{sec.universality}, but with $\sigma = 0$. For the convenience of the reader we recall here that mean field equilibria are characterized by the FBSDE system
\begin{equation} \label{meanfieldpont.s0}
    \begin{dcases}
         \d X_t = - D_p \cH\big(X_t,Y_t,\cL(X_t \mid \scrF_t^0)\big) \,\d t +  \sqrt{2 \sigma_0} \,\d W_t^0
        \\
         \d Y_t = D_x \cH\big(X_t,Y_t,\cL(X_t \mid \scrF_t^0) \big) \,\d t + Z_t^0 \,\d W_t^0
        \\
        X_0 \sim m_0\,, \quad Y_T = D_x \cG(X_T, \cL(X_T \mid \scrF_T^0))\,,
    \end{dcases}
\end{equation}
and, again, we will consider (conditionally) i.i.d.\ copies of the solution $(X,Y,Z)$ to \eqref{meanfieldpont.s0}, denoted $(\ov{X}^i,\ov{Y}^i,\ov{Z}^i)$, and obtained by taking $\bm\zeta \sim m_0^{\otimes N}$ and solving, for each $i$, the system 
\begin{equation*} %\label{meanfieldpointiid.s0}
    \begin{dcases}
        \d\ov{X}^i_t = - D_p \cH\big( \ov{X}^i_t,\ov{Y}^i_t,\cL(\ov{X}^i_t \mid \scrF_t^0)\big) \,\d t + \sqrt{2 \sigma_0} \,\d W_t^0
        \\
        \d\ov{Y}^i_t = D_x \cH\big( \ov{X}^i_t,\ov{Y}^i_t,\cL(\ov{X}^i_t \mid \scrF_t^0)\big) \,\d t +\ov{Z}^{i,0}_t \d W_t^0
        \\
        \ov{X}^i_0 = \zeta^i_0\,, \quad \ov{Y}_T^i = D_x \cG(\ov{X}_T^i, \cL(\ov{X}_T^i \mid \scrF_T^0))\,.
    \end{dcases}
\end{equation*}
We now fix a \emph{vanishing} sequence of positive coefficients $\sigma_N$ ($N \in \N$)
and consider the open-loop $N$-player with dynamics 
\[
    \d X_t^i = \alpha_t^i + \sqrt{2\sigma_N}\,\d W_t^i + \sqrt{2\sigma_0}\,\d W_t^0\,, \qquad X_0^i = \zeta^i_0
\]
and costs 
\[
    J^{i}_{\ol}(\bm \alpha) = \E\bigg[\int_0^T \cL\big(X_t^i, \alpha_t^i,m_{\bX_t}^{N,i}) \,\d t + \cG(X_T, m_{\bX_T}^{N,i}) \bigg]\,,
\]
where $m^{N,i}_{\bx} \defeq \frac1N \sum_{j=1,\, j \neq i}^N \delta_{x^j}$.
As in \Cref{sec.universality}, we denote by $\wt{\bX}^{N} = (\wt X^{N,1},\dots,\wt X^{N,N})$ the open-loop equilibirum for this game, which is characterized by the FBSDE system
\begin{equation*} %\label{pont.vanishingvisc} 
    \begin{dcases}
         \d \wt X_t^{N,i} = - D_p \cH\big(\wt X_t^{N,i}, \wt Y_t^{N,i}, m_{\wt{\bX}_t^N}^{N,i} \big) \,\d t + \sqrt{2 \sigma_N} \,\d W_t^i + \sqrt{2 \sigma_0} \,\d W_t^0
        \\
         \d \wt Y_t^{N,i} =  D_x \cH\big(\wt X_t^{N,i}, \wt Y_t^{N,i}, m_{\wt{\bX}_t^N}^{N,i} \big) \,\d t + \sum_{j=1}^N \wt Z_t^{N,i,j} \d W_t^j 
       \\[-3pt]
       \wt{\bX}_{t_0}^{N} = \bzeta_0, \quad \wt Y_{T}^{N,i} = D_{x} \cG(\wt X_T^{N,i}, m_{\wt{\bX}_T^{N}}^{N,i})\,.
    \end{dcases}
\end{equation*}
We also consider the closed-loop $N$-player game with dynamics
\[
   \d X_t^{i} = \alpha^i(t,\bX_t) \,\d t + \sqrt{2\sigma_N}\,\d W_t^i + \sqrt{2\sigma_0}\, \d W_t^0\,, \qquad X_0^i = \zeta^i_0\,,
\]
%\marginpar{\tiny DF: Fixed this red part with non-seprarable H; please check if there are any typos left.}{and costs
\[
    J^{i}_{\cl}(\bm \alpha) = \E\bigg[\int_0^T \cL(X_t^i, \alpha^i(t,\bX_t),m_{\bX_t}^{N,i}) \,\d t + \cG(X_T, m_{\bX_T}^{N,i}) \bigg]\,,
\]
and, as in \Cref{sec.universality}, we denote by $\bX^{N} = (X^{N,1},\dots,X^{N,N})$ the closed-loop equilibrium for this game, which satisfies 
\[
    \d X_t^{N,i} = - D_p \cH\bigl(X_t^{N,i}, D_{i}u^{N,i}(t,\bX_t^{N,i}),m^{N,i}_{\bX_t}\bigr) \d t + \sqrt{2\sigma_N} \,\d W_t^i + \sqrt{2\sigma_0} \,\d W_t^0\,,
\]
where $\bu^N = (u^{N,1},\dots,u^{N,N})$ solves the Nash system
\begin{equation*} %\label{nashmf.vanishing} 
    \begin{dcases}
     - \partial_t u^{N,i} - \sigma_N \sum_{j} \Delta_{j} u^{N,i} - \sigma_0 \sum_{j,k} \tr\big( D_{jk} u^{N,i} \big) \\
     \qquad {}+ \cH\big(x^i, D_{i} u^{N,i}, m_{\bx}^{N,i}\big) 
     + \sum_{j \neq i} D_{p} \cH\big(x^j, D_{j} u^{N,j}, m_{\bx}^{N,i}\big) D_{j} u^{N,i} =  0
    \\ 
    u^{N,i}(T,\cdot) = G^{N,i}\,.
    \end{dcases}
    \end{equation*}

Here is the main convergence result. 
\begin{theorem} \label{thm.vvlS}
    Let \Cref{assump.vanishingvisc} hold. Then, with implied constants independent of $N$ and with $\rho_{d}$ as in \eqref{est.mf.intro}, we have
    \begin{equation} \label{vanishingvisc1}
    \frac{1}{N}\, \E\Big[ \sup_{t \in [0,T]} |\ov{\bX}_t^N - \wt{\bX}_t^N|^2 \Big]\lesssim \rho_{d}(N) + \sigma_N\,,
    \end{equation}
    and, as a consequence, 
    \begin{align*}
        \E\Big[ \sup_{t \in [0,T]} \bd_2\big(m_{\wt{\bX}_t^N}^N, m_t\big)^2 \Big] \lesssim \rho_{d}(N) + \sigma_N.
    \end{align*}
    If in addition
     \begin{equation} \label{vanishingsigmaNcond}
     \sigma_N \sqrt{N} \xrightarrow{N \to \infty} \infty\,,
     \end{equation}
     then we also have
    \begin{equation}  \label{vanishingvisc2}
      \frac1N\, \E\Big[ \sup_{t \in [0,T]} |\bar \bX_t - \bX_t^{N}|^2 \Big] \lesssim \rho_{d}(N) + \sigma_N \,, 
    \end{equation}
    and, as a consequence,
    \begin{align*}
          \E\Big[ \sup_{t \in [0,T]} \bd_2\big(m_{\wt{\bX}_t^N}^N, m_t\big)^2 \Big] \lesssim \rho_{d}(N) + \sigma_N.
    \end{align*}
\end{theorem}

\begin{proof} 
%\marginpar{\tiny DF: Rewritten the proof with non-separable H. Please check if it's OK.}
    It is straightforward to check (as in \cite[Remark 3.5]{CR24b}) that for all large enough $N$ (depending on $C_{\cL,a}$, $C_{\cL,x}$, and $C_{\cG,\dis}$), the data $L^{N,i}(\bx,a) = \cL(x^i,a^i,m_{\bx}^{N,i}) $ and $G^{N,i}(\bx) \defeq \cG(x^i,m_{\bx}^{N,i})$ satisfy \Cref{assump.disp} with constants which are independent of $N$. % Furthermore, if we set $C^N_{\cL,\bx}$ and $C^N_{\cG,\dis}$ to be the displacement semi-monotonicity constants of $\bL^{N}$ and $\bG^{N}$, respectively, then for all large enough $N$ we have
    % \begin{equation} \label{cdisN}
    %    C_{\dis}^N \defeq C_{\cL,} - T\Bigl(C_{\cG,\dis}^N + \frac{T}2\, C_{\cF,\dis}^N\Bigr)  \geq \frac{\calC_{\dis}}2 > 0\,,
    % \end{equation}
    % where $\calC_\dis$ is defined as in \eqref{cdis.vv}.   
    For simplicity, we set
    \[
    \bar\alpha_t^{i} \defeq - D_p\cH\bigl(\bar X_t^{i}, \bar Y_t^{i},\cL(\bar X^i_t \mid \scrF^0_t)\bigr)\,, \qquad \wt\alpha_t^{N,i} \defeq - D_p \cH\bigl(\wt X_t^{N,i}, \wt Y_t^{N,i},m^{N,i}_{\wt\bX^N_t}\bigr)\,,
    \]
   as well as
    \[
        \Delta X_t^{i} \defeq \wt X_t^{N,i} - \bar X_t^{i}\,,
        \]
        and likewise for $\Delta Y_t^{i}$, $\Delta \alpha_t^i$, and $\Delta Z_t^{i} \defeq \wt Z^{N,i,i}_t - \bar Z^{i}_t$. Write
        \[
        D_x \cL\Bigr|^{(\wt X_t^{i}, \wt Y^i_t, m_{\wt{\bX}_t}^{N,i})}_{(\bar X_t^{i}, \bar Y^i_t, \cL(\bar X_t^{i}\mid \scrF_t^0))} = D_x \cL\Bigr|^{(\wt X_t^{i}, \wt Y^i_t, m_{\wt{\bX}_t}^{N,i})}_{(\bar X_t^{i}, \bar Y^i_t, m_{\bar{\bX}_t}^{N,i})} + E^{\cL,i}_t\,,
        \qquad 
        E^{\cL,i}_t \defeq D_x \cL\Bigr|^{(\bar X_t^{i}, \bar Y^i_t, m_{\bar{\bX}_t}^{N,i})}_{(\bar X_t^{i}, \bar Y^i_t, \cL(\bar X_t^{i}\mid \scrF_t^0))}\,,
        \]
        and likewise for $\cG$.
    Then, proceed as in the proof of \Cref{prop.olboundsdisp} by computing $\d(\Delta X^i_t \cdot \Delta Y^i_t)$, integrating in time, and taking expectations to get
    \begin{equation} \label{firststepalpha}
    \begin{split}%[b][.9\displaywidth]
         &C_{\bL,\ba} \E\bigg[ \int_0^T |\Delta \bm\alpha|^2 \bigg] 
         \\[3pt]
         &\ \leq  \E\bigg[ C_{\bG,\dis}^N |\Delta \bX_T|^2 + |\bm E^{\cG}| |\Delta \bX_T| + \int_0^T \!\Big( C_{\bL,\bx}^N |\Delta \bX|^2 + |\bm E^{\cL}| |\Delta \bX| + \sqrt{\sigma_N} \sum_i |\Delta Z^i| \Big)\bigg]\,.
         \end{split}
    \end{equation}
    Moreover, since 
\[
        \Delta X_t^i = \int_0^t \Delta \alpha^i + \sqrt{2\sigma_N} \,W_t^i\,, 
\]
   we find that for any $\delta > 0$ there is a constant $C_{\delta}$ such that
    \begin{equation} \label{sigmnbound}
        \E\Bigl[\, \sup_{s \in [0,t]} |\Delta X_s^i|^2 \Bigr] \leq (1 + \delta)t\, \E\biggl[\, \int_0^T |\Delta \alpha^i|^2  \biggr] + C_{\delta} \sigma_N\,.
    \end{equation}
     Plugging this into \eqref{firststepalpha}, and recalling the definition \eqref{cdispdef} of $C_\dis$, we get 
     \begin{equation} \label{secondstepalpha}
         \begin{multlined}[b][.9\displaywidth]
         \Big(C_\dis - \delta T \Bigl(C_{\bG,\dis} + \frac T2\, C_{\bL,\bx}\Bigr)\Big) \,\E\bigg[ \int_0^T |\Delta \bm\alpha|^2 \bigg] 
         \\
         \leq  \E\bigg[ |\bm E^{\cG}| |\Delta \bX_T| + \int_0^T \Big( |\bm E^{\cL}| |\Delta \bX| + \sqrt{\sigma_N}\, \sum_{i} |\Delta Z^i| \Big)\bigg] + C_{\delta}\bigl(C_{\bG,\dis} + TC_{\bL,\bx} \bigr) N\sigma_N\,;
\end{multlined}
\end{equation}
        meanwhile, by expanding $|\Delta \bY_t|^2$ (as in the proof of \Cref{prop.olboundsdisp}), we obtain 
    \begin{equation} \label{penstepalpha}
    \begin{split}
        \sup_{t \in [0,T]} \E\Big[ |\Delta \bY|^2 \Big] + \E\bigg[ \int_0^T |\Delta \bZ_t|^2 \bigg] 
        &\lesssim \E\bigg[ \sup_{t \in [0,T]}  |\Delta \bX_t|^2
      + \int_0^T |\bm E^{\cL}|^2 + |\bm E^{\cG}|^2\bigg]
        \\
        & \lesssim \E\bigg[ \int_0^T |\Delta \bm\alpha|^2 
       + \int_0^T |\bm E^{\cL}|^2 +|\bm E^{\cG}|^2\bigg] + N\sigma_N\,,
       \end{split}
    \end{equation}
    with implied constants (also below) independent of $N$ and where the last line used \eqref{sigmnbound}. Choose now $\delta = C_{\dis}/ (2TC_{\bG,\dis}  + T^2C_{\bL,\bx})$ in \eqref{secondstepalpha}, then apply Young's inequality therein, and finally use \eqref{sigmnbound} and \eqref{penstepalpha}, to find that 
        \[
          \frac{C_\dis}{2} \,\E\bigg[ \int_0^T |\Delta \bm\alpha|^2 \bigg] \leq \frac{C_\dis}{4}\, \E\bigg[ \int_0^T |\Delta \bm\alpha|^2 \bigg] 
+ C\E\bigg[ \int_0^T |\bm E^{\cL}|^2 +|\bm E^{\cG}|^2\bigg] + CN\sigma_N\,,
\]
    where $C$ depends only on the previous constants and $C_\dis$. So, absorbing the first term on the right-hand side and bounding the error terms $|\bm E^{\cL}|^2$ and $|\bm E^{\cG}|^2$ via \cite[Theorem~1]{FG13} (as in the proof of \Cref{thm.universality}), we conclude that
     \[
           \E\bigg[ \int_0^T |\Delta \bm\alpha|^2 \bigg] \lesssim N(\rho_{d}(N) + \sigma_N),
    \]
    whence, by symmetry,
     \[
       \sup_{i}\, \E\bigg[ \int_0^T |\Delta \alpha^i|^2 \bigg] \lesssim \rho_{d}(N) + \sigma_N\,,
    \]
    and then the bound \eqref{vanishingvisc1} follows from \eqref{sigmnbound}. 

    To prove \eqref{vanishingvisc2}, we note that the condition \eqref{vanishingsigmaNcond} guarantees that \eqref{smallness} holds for all large enough $N$, and so Theorem \ref{thm.olcldisp} applies to show that 
    \[
        \E\Big[ \sup_{t \in [0,T]} |\wt{\bX}_t^{N} - \bX_t^{N}|^2 \Big] \lesssim \frac{1}{\sigma_N N}\,,
   \]
    %$by symmetry,\marginpar{\tiny {I believe there is a gap here. Symmetry does not let us obviously go from the first statement to the second, unless we pull the $t$ outside of the expectation!}} we deduce that 
    %\[
    %\sup_i \E\Big[ \sup_{t \in [0,T]} |\wt{X}_t^{N,i} - X_t^{N,i}|^2 \Big] %\lesssim \frac{1}{\sigma_N N^2}\,,
%\]
 so \eqref{vanishingvisc2} follows from the triangle inequality and the fact that $ (\sigma_N N^{2})^{-1} = o(\sigma_N)$ as $N \to \infty$. Again, we use the fact that
\begin{align*}
    \E\Big[ \sup_{t \in [0,T]} \bd_2\big(m_{\ov{\bX}_t^N}^N, m_t\big)^2 \Big] \lesssim \rho_d(N), 
\end{align*}
to infer the estimates on the empirical measures from the triangle inequality. 
\end{proof}

\bibliographystyle{acm}

\bibliography{CJR}

\end{document}